\pdfoutput=1
\documentclass[reqno]{amsart}
\usepackage{amsmath}
\usepackage{amsthm}
\usepackage{amssymb}
\usepackage[left=1.25in, right=1.25in]{geometry}
\usepackage{tikz}[2010/10/13]
\usepackage[all]{xy}
\usepackage{graphicx}
\usepackage{indentfirst}
\usepackage{bm}
\usepackage{mathrsfs}
\usepackage{latexsym}
\usepackage{hyperref}
\usepackage{microtype}

\newcommand{\gra}[1]{\raisebox{-.4cm}{\includegraphics[height=1cm]{TP#1.pdf}}}
\newcommand{\graa}[1]{\raisebox{-.6cm}{\includegraphics[height=1.5cm]{TP#1.pdf}}}
\newcommand{\grb}[1]{\raisebox{-.8cm}{\includegraphics[height=2cm]{TP#1.pdf}}}

\newcommand{\grd}[1]{\raisebox{-1.8cm}{\includegraphics[height=4cm]{TP#1.pdf}}}

\begin{document}
	\title{From skein theory to presentations for Thompson group}
	\author{Yunxiang Ren}
	\maketitle

	\newtheorem{Lemma}{Lemma}
	\theoremstyle{plain}
	\newtheorem{theorem}{Theorem~}[section]
	\newtheorem{main}{Main Theorem~}
	\newtheorem{lemma}[theorem]{Lemma~}
	\newtheorem{approach}[theorem]{Approach~}
	\newtheorem{assumption}[theorem]{Assumption~}
	\newtheorem{proposition}[theorem]{Proposition~}
	\newtheorem{corollary}[theorem]{Corollary~}
	\newtheorem{definition}[theorem]{Definition~}
	\newtheorem{defi}[theorem]{Definition~}
	\newtheorem{notation}[theorem]{Notation~}
	\newtheorem{example}[theorem]{Example~}
	\newtheorem*{remark}{Remark~}
	\newtheorem*{cor}{Corollary~}
	\newtheorem*{question}{Question}
	\newtheorem*{claim}{Claim}
	\newtheorem*{conjecture}{Conjecture~}
	\newtheorem*{fact}{Fact~}
	\renewcommand{\proofname}{\bf Proof}
	
	\begin{abstract}
		Jones introduced unitary representations of Thompson group $F$ constructed from a given subfactor planar algebra, and all unoriented links arise as matrix coefficients of these representations. Moreover, all oriented links arise as matrix coefficients of a subgroup $\vec{F}$ which is the stabilizer of a certain vector. Later Golan and Sapir determined the subgroup $\vec{F}$ and showed many interesting properties. In this paper, we investigate into a large class of groups which arises as subgroups of Thompson group $F$ and reveal the relation between the skein theory of the subfactor planar algebra and the presentation of subgroup related to the corresponding unitary representation. Specifically, we answer a question by Jones about the 3-colorable subgroup.
	\end{abstract}
	
	\section{introduction}
	Jones initiated the modern theory of subfactors to study quantum symmetry \cite{Jon83}.  The standard invariant of a subfactor is the lattice of higher relative commutants of the Jones tower.  A deep theorem of Popa says the standard invariants completely classify strongly amenable subfactors \cite{Pop94}. In particular, the $A$, $D$, $E$ classification of subfactors up to index 4 is a quantum analog of Mckay correspondence.
	Moreover, Popa introduced standard $\lambda$-lattices as an axiomatization of the standard invariant \cite{Pop95}, which completes Ocneanu's axiomatization for finite depth subfactors \cite{Ocn88}.
	
	Jones introduced (subfactor) planar algebras as an axiomatization of the standard invariant of subfactors, which capture its ulterior topological properties.
	He suggested studying planar algebras by skein theory, a presentation theory which allows one to completely describe the entire planar algebra in terms of generators and relations, both algebraic and topological.
	
	An $m$-box generator for a planar algebra is usually represented by a $2m$-valent labeled vertex in a disc, with a choice of alternating shading of the planar regions.  A planar algebra is a representation of \textit{fully labeled planar tangles} on vector spaces in the flavor of TQFT \cite{Ati88}, in the sense that 
	the representation is well defined up to isotopy, and the target vector space only depends on the boundary conditions of diagrams.
	In particular, diagrams without boundary are mapped to the ground field, yielding a map called the partition function.
	A planar algebra is called a subfactor planar algebra if its partition function is positive definite. In this case, the vector spaces become Hilbert spaces.
	
	A skein theory for a planar algebra is a presentation given in terms of generators and relations, such that the partition function of a closed diagram labeled by the generators can be evaluated using only the prescribed relations.  This type of description of a planar algebra is analogous to presentations in combinatorial group theory.  With an evaluation algorithm in hand, a  skein theory uniquely determines a planar algebra, and  therefore one can ask for a skein theoretic classification of planar algebras as suggested by Bisch and Jones \cite{BisJon00,BisJon03}. Bisch and Jones initiated the classification of planar algebras generated by a single 2-box \cite{BisJon00}. Based on the subsequent work of Bisch, Jones and Liu \cite{BisJon03,BJL}, a classification of singly generated Yang-Baxter relation planar algebra was achieved in \cite{LiuYB}, where a new family of planar algebras was constructed. Planar algebras with multiple 2-box generators were discussed in \cite{Liuex}. Planar algebras generated by a single 3-box are discussed by C.Jones, Liu, and the author \cite{JLR}. We summarize the conditions of a skein theory for a subfactor planar algebra as follows:
	
	\begin{itemize}
		\item (Evaluation) There exists an evaluation algorithm for diagrams in $\mathscr{P}_{0,\pm}$ and $\dim\mathscr{P}_{n,\pm}<\infty$.
		\item (Consistency) The evaluation is consistent.
		\item (Positivity) There exists a positive semidefinite trace on $\mathscr{P}_\bullet$.
	\end{itemize}
	In this paper, the main skein theory we use in this paper is the vertical isotopy which is encoded in the definition of isotopy:
	\begin{definition}[Vertical isotopy]
		Suppose $X$ is an $(m,n)-$tangle and $Y$ is an $(k,l)-$tangle, then we have 
		\begin{equation*}
				\begin{tikzpicture}
				\draw (0,0)--(1,0);
				\draw (1,0)--(1,0.6);
				\draw (1,0.6)--(0,0.6);
				\draw (0,0)--(0,0.6);
				\node at (0.5,0.3) {$X$};
				\draw (0.2,0)--(0.2,-1);
				\draw (0.4,0)--(0.4,-1);
				\draw (0.8,0)--(0.8,-1);
				\draw (0.2,0)--(0.2,-1);
				\draw (0.4,0)--(0.4,-1);
				\draw (0.8,0)--(0.8,-1);
				\draw (.2,.6)--(.2,.8);
				\draw (.4,.6)--(.4,.8);
				\draw (.8,.6)--(.8,.8);
				\draw [fill] (.5,-0.5) circle [radius=0.02];
				\draw [fill] (.6,-0.5) circle [radius=0.02];
				\draw [fill] (.7,-0.5) circle [radius=0.02];
					\draw (1.5,0)--(2.5,0);
					\draw (2.5,0)--(2.5,-0.6);
					\draw (2.5,-0.6)--(1.5,-0.6);
					\draw (1.5,0)--(1.5,-0.6);
					\node at (2,-.3) {$Y$};
						\draw (1.5+0.2,-.6)--(1.5+0.2,-1);
						\draw (1.5+0.4,-.6)--(1.5+0.4,-1);
						\draw (1.5+0.8,-.6)--(1.5+0.8,-1);
						\draw (1.5+0.2,-.6)--(1.5+0.2,-1);
						\draw (1.5+0.4,-.6)--(1.5+0.4,-1);
						\draw (1.5+0.8,-.6)--(1.5+0.8,-1);
						\draw (1.5+.2,0)--(1.5+.2,.8);
						\draw (1.5+.4,0)--(1.5+.4,.8);
						\draw (1.5+.8,0)--(1.5+.8,.8);
						\draw [fill] (1.5+.5,.4) circle [radius=0.02];
						\draw [fill] (1.5+.6,.4) circle [radius=0.02];
						\draw [fill] (1.5+.7,.4) circle [radius=0.02];
					\node at (3,-.1) {$=$};
						\draw (6-0,0)--(6-1,0);
						\draw (6-1,0)--(6-1,0.6);
						\draw (6-1,0.6)--(6-0,0.6);
						\draw (6-0,0)--(6-0,0.6);
						\node at (6-0.5,0.3) {$Y$};
						\draw (6-0.2,0)--(6-0.2,-1);
						\draw (6-0.4,0)--(6-0.4,-1);
						\draw (6-0.8,0)--(6-0.8,-1);
						\draw (6-0.2,0)--(6-0.2,-1);
						\draw (6-0.4,0)--(6-0.4,-1);
						\draw (6-0.8,0)--(6-0.8,-1);
						\draw (6-.2,.6)--(6-.2,.8);
						\draw (6-.4,.6)--(6-.4,.8);
						\draw (6-.8,.6)--(6-.8,.8);
						\draw [fill] (6-.5,-0.5) circle [radius=0.02];
						\draw [fill] (6-.6,-0.5) circle [radius=0.02];
						\draw [fill] (6-.7,-0.5) circle [radius=0.02];
						\draw (6-1.5,0)--(6-2.5,0);
						\draw (6-2.5,0)--(6-2.5,-0.6);
						\draw (6-2.5,-0.6)--(6-1.5,-0.6);
						\draw (6-1.5,0)--(6-1.5,-0.6);
						\node at (6-2,-.3) {$X$};
						\draw (6-1.5-0.2,-.6)--(6-1.5-0.2,-1);
						\draw (6-1.5-0.4,-.6)--(6-1.5-0.4,-1);
						\draw (6-1.5-0.8,-.6)--(6-1.5-0.8,-1);
						\draw (6-1.5-0.2,-.6)--(6-1.5-0.2,-1);
						\draw (6-1.5-0.4,-.6)--(6-1.5-0.4,-1);
						\draw (6-1.5-0.8,-.6)--(6-1.5-0.8,-1);
						\draw (6-1.5-.2,0)--(6-1.5-.2,.8);
						\draw (6-1.5-.4,0)--(6-1.5-.4,.8);
						\draw (6-1.5-.8,0)--(6-1.5-.8,.8);
						\draw [fill] (6-1.5-.5,.4) circle [radius=0.02];
						\draw [fill] (6-1.5-.6,.4) circle [radius=0.02];
						\draw [fill] (6-1.5-.7,.4) circle [radius=0.02];
					
				\end{tikzpicture}		
		\end{equation*}
	\end{definition}
	
	Jones introduced unitary representations for Thompson groups \cite{JonTP} motivated by the idea of constructing a conformal field theory for every finite index subfactor in such a way that the standard invariant of the subfactor, or at least the quantum double, can be recovered from the CFT. Following the idea of block spin renormalization, one can construct a Hilbert space from the initial data of a subfactor planar algebra on which Thompson groups $F$ and $T$ have an action. Due to the intrinsic structure of a subfactor planar algebra, one can obtain unitary representations of $F$ and $T$. A significant result is that every unoriented link arises as matrix coefficients of these representations of $F$. Furthermore, all oriented links arise as matrix coefficients of a subgroup $F$, denoted by $\vec{F}$, defined as a stabilizer of a certain vector from these representations, which we call the vacuum vector. Golan and Sapir completely determined this subgroup and revealed many interesting properties \cite{GoSa2015}. In particular, they showed that $\vec{F}$ is
	isomorphic $F_3$, where the Thompson group $F_N$ for $N\in\mathbb{N}$ with $N\geq2$ is defined as 
	\begin{equation*}
	F_N\cong\langle x_1,x_2,\cdots\vert x_k^{-1}x_nx_k=x_{n+N-1}\rangle.
	\end{equation*}
	
	In this paper, we study subgroups of Thompson group $F$ as the stabilizer of the vacuum vector of the unitary representations constructing from subfactor planar algebras. We show that the presentation of these subgroups is encoded with the skein theory of the subfactor planar algebras. In particular, we study a family of subgroups of Thompson group $F$ called singly generated subgroups as an analogy of singly generated planar algebra. 
	
	\begin{theorem}\label{thm:main1}
		The singly generated subgroup with an $(1,N)$-tangle is isomorphic to $F_N$.
	\end{theorem}
	In particular, we apply the techniques to answer the question by Jones about the $3$-colorable subgroup and have the following theorem:
	\begin{theorem}
		The $3$-colorable subgroup is isomorphic to $F_4$. 
	\end{theorem}
	The paper is organized as follows. In \S \ref{sec:pre} we recall the definition of unitary representation for Thompson group $F$ \cite{JonTP}. In \S \ref{sec:def} we give the definition of singly generated groups. In \S \ref{sec:presentation} we introduce a classical presentation of singly generated groups and prove Theorem \ref{thm:main1}. In \S \ref{sec:example} we provide examples of singly generated groups: in \S \ref{sec:vecF} we provide a proof of that $\vec{F}\cong F_3$ from the topological viewpoint; in \S \ref{sec:3color} we show that the $3$-colorable subgroup is isomorphic to $F_4$.

	\section{preliminary}\label{sec:pre}
	We refer readers to \cite{JonPA} for details in planar algebras. In this section we recall the construction of unitary representations for Thompson group $F$ \cite{JonTP}.
	\begin{definition}[Standard dyadic partition]
		We say $\mathcal{I}$ is a standard dyadic partition of $[0,1]$ if for each $I\in\mathcal{I}$ there exists $a,p\in\mathbb{N}$ such that $I=\left[\displaystyle\frac{a}{2^p},\displaystyle\frac{a+1}{2^p}\right]$. We denote the set of all standard dyadic partitions by $\mathcal{D}$ and define $\mathcal{I}\lesssim\mathcal{J}$ for $\mathcal{I},\mathcal{J}\in\mathcal{D}$ if $\mathcal{J}$ is a refinement of $\mathcal{I}$. 
	\end{definition}
	\begin{proposition}
		For each $g\in F$, there exists $\mathcal{I}\in\mathcal{D}$ such that $g(\mathcal{I})\in\mathcal{D}$ and on each $I\in\mathcal{I}$ the slope of $g$ is a constant. We say such an $\mathcal{I}$ is "good" for $g$. 
	\end{proposition}
	
	There is a well known diagrammatic description of Thompson group $F$ \cite{BieS14}. For $\mathcal{I}\lesssim\mathcal{J}\in\mathcal{D}$, one can use binary a binary forest to represent the inclusion of $\mathcal{I}$ in $\mathcal{J}$ and $\mathcal{I}$ is "good" for some $g\in F$. For instance, $\mathcal{I}=\{[0,\displaystyle\frac{1}{2}],[\displaystyle\frac{1}{2},1]\}$, $\mathcal{J}=\{[0,\displaystyle\frac{1}{2}],[\displaystyle\frac{1}{2},\displaystyle\frac{3}{4}],[\displaystyle\frac{3}{4},1]\}$ and $g$ is the function
	\begin{align*}
	g(x)=
	\begin{cases}
	x/2,  & 0\leq x\leq \displaystyle\frac{1}{2}\\
	x-\displaystyle\frac{1}{4}, &\displaystyle\frac{1}{2}<x\leq \displaystyle\frac{3}{4}\\
	2x-1, &\displaystyle\frac{3}{4}<x\leq 1
	\end{cases}
	\end{align*}
	\begin{equation*}
	\begin{tikzpicture}
	\draw (0,0)--(2.5,0);
	\draw (0,-1.25)--(2.5,-1.25);
	\draw [red] (0,-.1)--(0,.1);
	\draw [red] (2.5*.5,-.1)--(2.5*.5,.1);
	\draw [red] (2.5,-.1)--(2.5,.1);
		\draw [red] (0,-.1-1.25)--(0,.1-1.25);
		\draw [red] (2.5*.5,-.1-1.25)--(2.5*.5,.1-1.25);
		\draw [red] (2.5,-.1-1.25)--(2.5,.1-1.25);
		\draw [red] (2.5*.75,-.1-1.25)--(2.5*.75,.1-1.25);
		\draw (2.5*.25,0)--(2.5*.25,-1.25);
		\draw (2.5*.75,0)--(2.5*.75,-.5);
		\draw (2.5*.625,-1.25)--(2.5*.625,-.75);
		\draw (2.5*.875,-1.25)--(2.5*.875,-.75);
		\draw (2.5*.75,-.6125)--(2.5*.75,-.5);
		\draw (2.5*.75,-.6125)--(2.5*.625,-.75);
		\draw (2.5*.75,-.6125)--(2.5*.875,-.75);
		\node [left] at (0,0) {$\mathcal{I}$};
		\node [left] at (0,-1.25) {$\mathcal{J}$};
		\node [left] at (-.5,-.6125) {$\mathcal{F}^\mathcal{I}_\mathcal{J}=$};
		 \draw (6+0,0)--(6+2.5,0);
		 \draw (6+0,-1.25)--(6+2.5,-1.25);
		 	\draw [red] (6+0,-.1)--(6+0,.1);
		 	\draw [red] (6+2.5*.5,-.1)--(6+2.5*.5,.1);
		 	\draw [red] (6+2.5,-.1)--(6+2.5,.1);
		 	\draw [red] (6+2.5*.75,-.1)--(6+2.5*.75,.1);
		 	\draw [red] (6+0,-.1-1.25)--(6+0,.1-1.25);
		 	\draw [red] (6+2.5*.5,-.1-1.25)--(6+2.5*.5,.1-1.25);
		 	\draw [red] (6+2.5,-.1-1.25)--(6+2.5,.1-1.25);
		 	\draw [red] (6+2.5*.25,-.1-1.25)--(6+2.5*.25,.1-1.25);
		 	\draw (6+2.5*.25,0)--(6+2.5*.125,-1.25);
		 	\draw (6+2.5*.625,0)--(6+2.5*.375,-1.25);
		 	\draw (6+2.5*.875,0)--(6+2.5*.75,-1.25);
		 	\node [left] at (6,0) {$\mathcal{J}$};
		 	\node [left] at (6,-1.25) {$g(\mathcal{J})$};
		 	\node [left] at (5.5,-.6125) {$g_\mathcal{J}=$};
		 \end{tikzpicture}
	\end{equation*}
		
	We construct a Hilbert space given a subfactor planar algebra $\mathscr{P}_\bullet$ \cite{JonTP} \cite{JonNogo}. In particular we introduce the following two approaches.
	\begin{approach}\label{app1}
		We start with $\mathscr{P}_\bullet$ a positive-definite planar algebra with a normalised trivalent vertex $S$, i.e, 
		\begin{equation*}
		\begin{tikzpicture}
		\draw (0,0) circle [radius=.5];
		\draw (0,.5)--(0,1);
		\draw (0,-.5)--(0,-1);
			\node [right] at (.5,0) {$=$};
			\draw (1.5,1)--(1.5,-1);
			\node [below] at (0,.5) {\tiny $S$};
			\node [above] at (0,-.5) {\tiny $S^*$};
		\end{tikzpicture}
		\end{equation*}
	
		For each $\mathcal{I}\in\mathcal{D}$, we set $\mathcal{H}(\mathcal{I})=\mathscr{P}_{M(\mathcal{I})}$, where $M(\mathcal{I})$ is the number of midpoints of the intervals of $\mathcal{I}$ and define the inclusion and action as 
		\begin{equation*}
		\begin{tikzpicture}
		\draw (0,0)--(2.5,0);
		\draw (0,-1.25)--(2.5,-1.25);
		\draw [red] (0,-.1)--(0,.1);
		\draw [red] (2.5*.5,-.1)--(2.5*.5,.1);
		\draw [red] (2.5,-.1)--(2.5,.1);
		\draw [red] (0,-.1-1.25)--(0,.1-1.25);
		\draw [red] (2.5*.5,-.1-1.25)--(2.5*.5,.1-1.25);
		\draw [red] (2.5,-.1-1.25)--(2.5,.1-1.25);
		\draw [red] (2.5*.75,-.1-1.25)--(2.5*.75,.1-1.25);
		\draw (2.5*.25,0)--(2.5*.25,-1.25);
		\draw (2.5*.75,0)--(2.5*.75,-.5);
		\draw (2.5*.625,-1.25)--(2.5*.625,-.75);
		\draw (2.5*.875,-1.25)--(2.5*.875,-.75);
		\draw (2.5*.75,-.6125)--(2.5*.75,-.5);
		\draw (2.5*.75,-.6125)--(2.5*.625,-.75);
		\draw (2.5*.75,-.6125)--(2.5*.875,-.75);
		\node [left] at (0,0) {$\mathcal{I}$};
		\node [left] at (0,-1.25) {$\mathcal{J}$};
		\node [left] at (-.5,-.6125) {$T^\mathcal{I}_\mathcal{J}=$};
		\node [below] at (2.5*.75,-.65) {\tiny $S$};
		\draw (6+0,0)--(6+2.5,0);
		\draw (6+0,-1.25)--(6+2.5,-1.25);
		\draw [red] (6+0,-.1)--(6+0,.1);
		\draw [red] (6+2.5*.5,-.1)--(6+2.5*.5,.1);
		\draw [red] (6+2.5,-.1)--(6+2.5,.1);
		\draw [red] (6+2.5*.75,-.1)--(6+2.5*.75,.1);
		\draw [red] (6+0,-.1-1.25)--(6+0,.1-1.25);
		\draw [red] (6+2.5*.5,-.1-1.25)--(6+2.5*.5,.1-1.25);
		\draw [red] (6+2.5,-.1-1.25)--(6+2.5,.1-1.25);
		\draw [red] (6+2.5*.25,-.1-1.25)--(6+2.5*.25,.1-1.25);
		\draw (6+2.5*.25,0)--(6+2.5*.125,-1.25);
		\draw (6+2.5*.625,0)--(6+2.5*.375,-1.25);
		\draw (6+2.5*.875,0)--(6+2.5*.75,-1.25);
		\node [left] at (6,0) {$\mathcal{J}$};
		\node [left] at (6,-1.25) {$g(\mathcal{J})$};
		\node [left] at (5.5,-.6125) {$g_\mathcal{J}=$};
		\end{tikzpicture}
		\end{equation*}
	\end{approach}
	\begin{approach}\label{app2}
		We start with $\mathscr{P}_\bullet$ a positive planar algebra with a normalised $2$-box $R$, i.e, 
		\begin{equation*}
		\begin{tikzpicture}
		\draw (0,0) circle [radius=.5];
		\draw (0,1)--(0,-1);
		\node [right] at (.5,0) {$=$};
		\draw (1.5,1)--(1.5,-1);
		\node [below left] at (0,.5) {\tiny $R$};
		\node [above left] at (0.15,-.5) {\tiny $R^*$};		
		\end{tikzpicture}
		\end{equation*}
				
		For each $\mathcal{I}\in\mathcal{D}$, we set $\mathcal{H}(\mathcal{I})=\mathscr{P}_{E(\mathcal{I})}$, where $E(\mathcal{I})$ is the number of the midpoints and endpoints of intervals of $\mathcal{I}$ except $0,1$ and define the inclusion and action as 
		\begin{equation*}
		\begin{tikzpicture}
		\draw (0,0)--(2.5,0);
		\draw (0,-1.25)--(2.5,-1.25);
		\draw [red] (0,-.1)--(0,.1);
		\draw [red] (2.5*.5,-.1)--(2.5*.5,.1);
		\draw [red] (2.5,-.1)--(2.5,.1);
		\draw [red] (0,-.1-1.25)--(0,.1-1.25);
		\draw [red] (2.5*.5,-.1-1.25)--(2.5*.5,.1-1.25);
		\draw [red] (2.5,-.1-1.25)--(2.5,.1-1.25);
		\draw [red] (2.5*.75,-.1-1.25)--(2.5*.75,.1-1.25);
		\draw (2.5*.25,0)--(2.5*.25,-1.25);
		\draw (2.5*.75,0)--(2.5*.75,-1.25);
		\draw (2.5*.625,-1.25)--(2.5*.625,-.75);
		\draw (2.5*.875,-1.25)--(2.5*.875,-.75);
		\draw (2.5*.875,-.75) arc [radius=2.5*.125, start angle=0, end angle=180];		
		\draw (2.5*.5,0)--(2.5*.5,-1.25);
		\node [left] at (0,0) {$\mathcal{I}$};
		\node [left] at (0,-1.25) {$\mathcal{J}$};
		\node [left] at (-.5,-.6125) {$T^\mathcal{I}_\mathcal{J}=$};
		\node [below] at (2.5*.7,-.5) {\tiny $R$};
		\draw (6+0,0)--(6+2.5,0);
		\draw (6+0,-1.25)--(6+2.5,-1.25);
		\draw [red] (6+0,-.1)--(6+0,.1);
		\draw [red] (6+2.5*.5,-.1)--(6+2.5*.5,.1);
		\draw [red] (6+2.5,-.1)--(6+2.5,.1);
		\draw [red] (6+2.5*.75,-.1)--(6+2.5*.75,.1);
		\draw [red] (6+0,-.1-1.25)--(6+0,.1-1.25);
		\draw [red] (6+2.5*.5,-.1-1.25)--(6+2.5*.5,.1-1.25);
		\draw [red] (6+2.5,-.1-1.25)--(6+2.5,.1-1.25);
		\draw [red] (6+2.5*.25,-.1-1.25)--(6+2.5*.25,.1-1.25);
		\draw (6+2.5*.25,0)--(6+2.5*.125,-1.25);
		\draw (6+2.5*.625,0)--(6+2.5*.375,-1.25);
		\draw (6+2.5*.875,0)--(6+2.5*.75,-1.25);
		\node [left] at (6,0) {$\mathcal{J}$};
		\node [left] at (6,-1.25) {$g(\mathcal{J})$};
		\node [left] at (5.5,-.6125) {$g_\mathcal{J}=$};
		\end{tikzpicture}
		\end{equation*}  
	\end{approach}
	
	In both cases, $T^{\mathcal{I}}_{\mathcal{J}}$ is an isometry from $\mathcal{H}(\mathcal{I})$ to $\mathcal{H}(\mathcal{J})$. Therefore the direct limit of $\mathcal{H}(\mathcal{I})$ for $\mathcal{I}\in\mathcal{D}$ is a prehilbert space. We complete the direct limit to obtain a Hilbert space, denoted by $\mathcal{H}_S$ (or $\mathcal{H}_R$). For each $v\in \mathcal{H}(\mathcal{I})$ and $g\in F$, let $\mathcal{J}\in \mathcal{D}$ such that $\mathcal{I}\lesssim\mathcal{J}$ and $\mathcal{J}$ "good" for $g$. We define $\pi:F\rightarrow \mathcal{B}(\mathcal{H}_S)$ (or $\mathcal{B}(\mathcal{H}_R)$) as
	\begin{equation*}
	\pi(g)v=g_\mathcal{J}T^{\mathcal{I}}_\mathcal{J}(v).
	\end{equation*}
	
	From the theory of subfactors, one can show that $(\pi, \mathcal{H}_R)$ or $(\pi, \mathcal{H}_R)$ is a unitary representation for the Thompson group $F$. Naturally we obtain a subgroup as the stabilizer of the vacuum vector $\xi=\gra{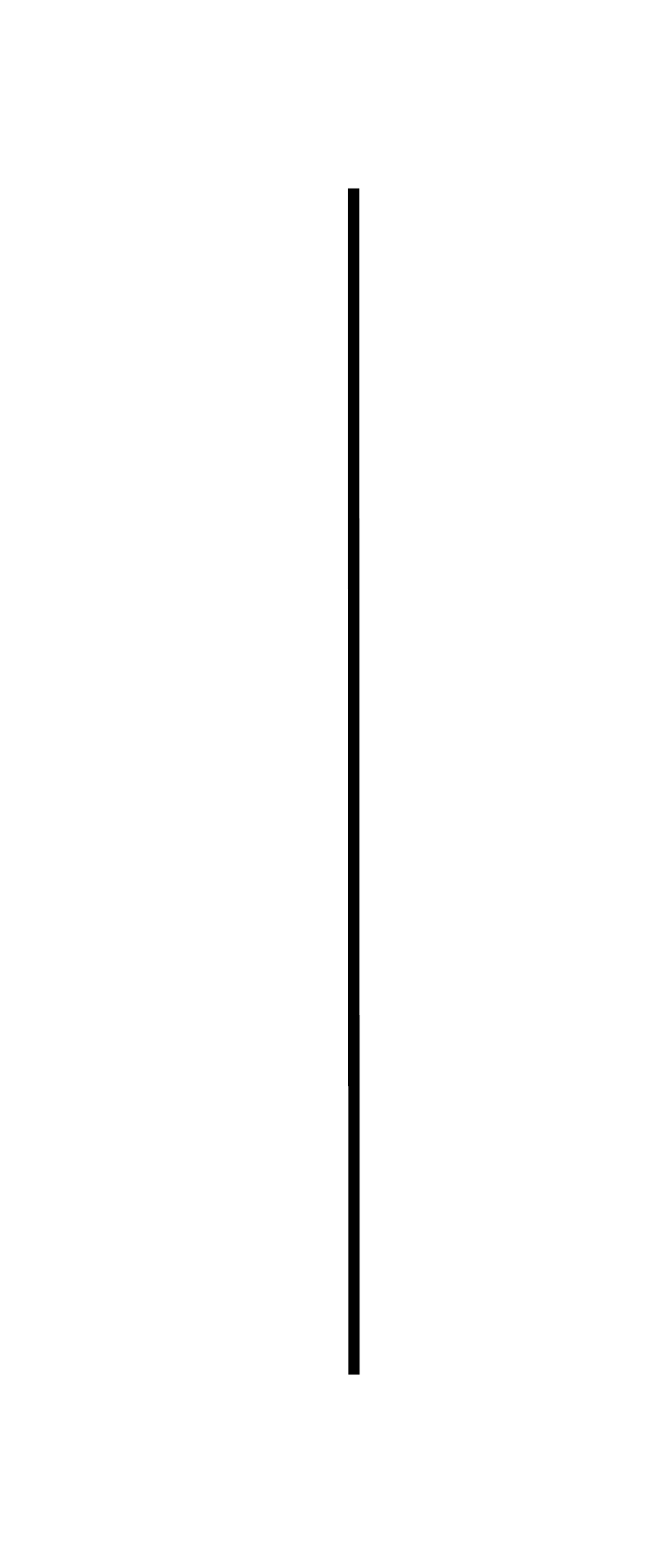}$.
	\begin{definition}
		The subgroup $F_\xi$ is defined as 
		\begin{equation*}
		F_\xi=\{g\in F\vert \pi(g)\xi=\xi\}
		\end{equation*}
	\end{definition}
	With specific choices of $R$ and $S$, we obtain the Jones subgroup $\vec{F}$ in \S \ref{sec:vecF} and the 3-colorable subgroup in \S \ref{sec:3color}.

	\section{Singly generated subgroups}\label{sec:def}
	In this section, we introduce a group motivated by the braid group $B_n,n\in\mathbb{N}$. The braid group $B_n$ is the group formed by appropriate isotopy classes of braids with obvious concatenation operation. A preferred set of generators $\sigma_1,\sigma_2,\cdots,\sigma_{n-1}$ given by the following pictures:
	\begin{equation*}
	\grb{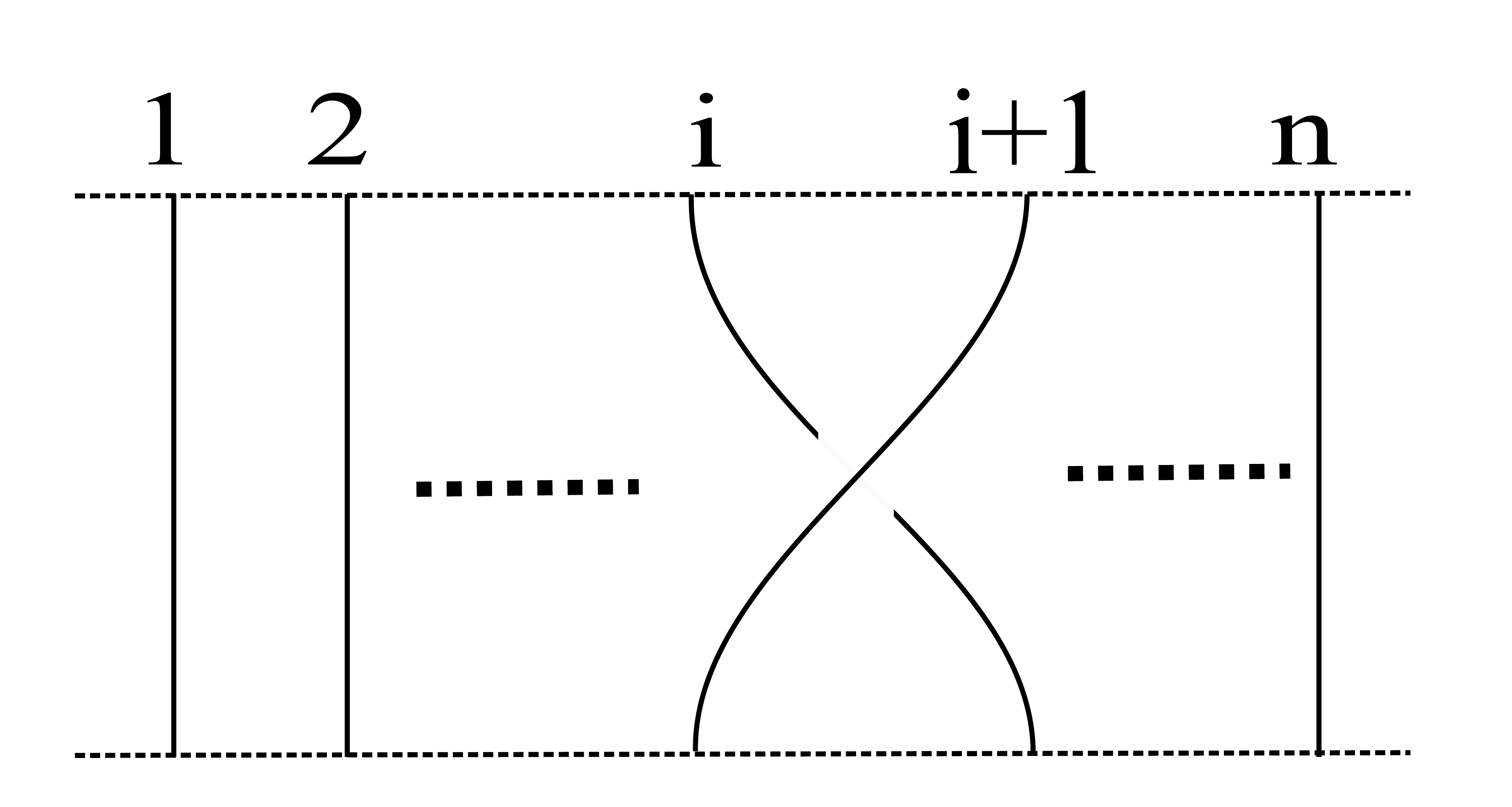}
	\end{equation*}
	One can easily to verify the following relations:
	\begin{align}
	\sigma_i\sigma_{i+1}\sigma_i&=\sigma_{i+1}\sigma_i\sigma_{i+1},~i=1,2,\cdots,n-1\label{rel:braid1}\\
	\sigma_i\sigma_j&=\sigma_j\sigma_i,~\vert i-j\vert\geq2\label{rel:braid2}
	\end{align}
	Relation \eqref{rel:braid1} and \eqref{rel:braid2} give a presentation of $B_n$ was proved by E.Artin. It follows that $B_n$ can be embedded into $B_{n+1}$ by restricting the $(n+1)$-th string to be a through string. Therefore, one can consider the braid group by taking the inductive limit of $B_n,n\in\mathbb{N}$. In this case, every element can be interpreted as a diagram in $B_n$ for some $n\in\mathbb{N}$ with infinitely many through strings on its right side. 
	\begin{definition}
		The braid group $B_\infty$ is defined as the following
		\begin{equation*}
		B_\infty\cong\langle \sigma_1,\sigma_2,\cdots\vert \sigma_i\sigma_{i+1}\sigma_i=\sigma_{i+1}\sigma_i\sigma_{i+1},i=1,2,\cdots; \sigma_i\sigma_j=\sigma_j\sigma_i,\vert i-j\vert\geq2\rangle
		\end{equation*}
	\end{definition}
	
	One can think of $B_\infty$ as a group generated by a $(2,2)$-tangle of \gra{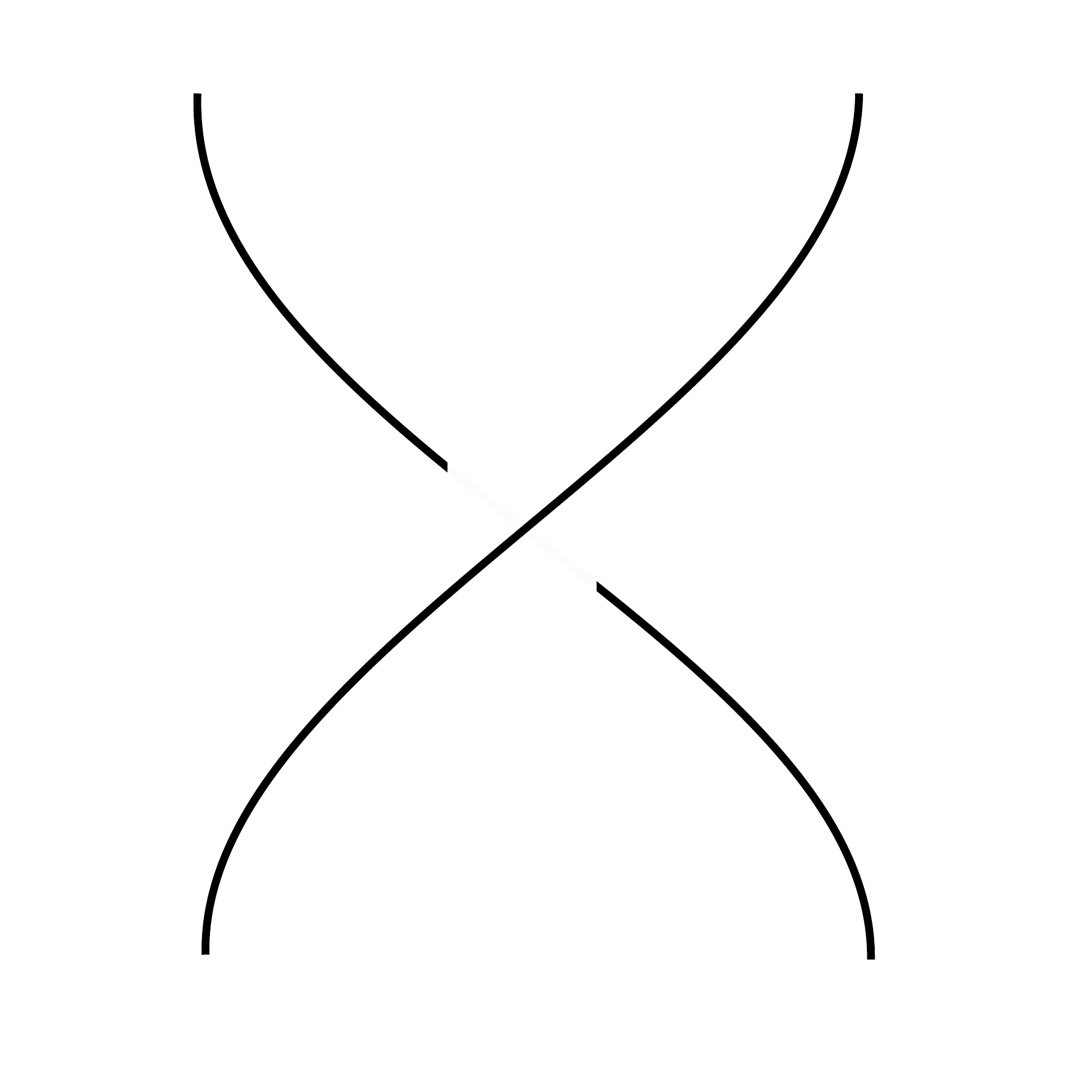}. In this section, we consider groups in this form for an arbitrary $(1,N)$-tangle 
	\begin{tikzpicture}
	\draw (0,0)--(1,0);
	\draw (0,0)--(0,.5);
	\draw (0,.5)--(1,.5);
	\draw (1,.5)--(1,0);
	\node at (.5,.25) {$X$};
	\draw (.5,.5)--(.5,.7);
	\draw (.2,0)--(.2,-.4);
	\draw (.4,0)--(.4,-.4);
	\draw (.8,0)--(.8,-.4);
		\draw [fill] (.5,-0.2) circle [radius=0.02];
		\draw [fill] (.6,-0.2) circle [radius=0.02];
		\draw [fill] (.7,-0.2) circle [radius=0.02];
	\end{tikzpicture}
	which arise as subgroups of Thompson group $F$. 
	
	\begin{definition}[Shifts of $X$]
		Let $X_k$, called the $k$-shift of $X$, be an $(k+1,k+N)$-tangle defined as follows:
		\begin{equation*}
		\begin{tikzpicture}
			\draw (0,0)--(1,0);
			\draw (0,0)--(0,.5);
			\draw (0,.5)--(1,.5);
			\draw (1,.5)--(1,0);
			\node at (.5,.25) {$X$};
			\draw (.5,.5)--(.5,.7);
			\draw (.2,0)--(.2,-.4);
			\draw (.4,0)--(.4,-.4);
			\draw (.8,0)--(.8,-.4);
			\draw [fill] (.5,-0.2) circle [radius=0.02];
			\draw [fill] (.6,-0.2) circle [radius=0.02];
			\draw [fill] (.7,-0.2) circle [radius=0.02];
			\draw (-.2,.7)--(-.2,-.4);
			\draw (-.6,.7)--(-.6,-.4);
			\draw [fill] (-.3,0.15) circle [radius=0.02];
			\draw [fill] (-.4,0.15) circle [radius=0.02];
			\draw [fill] (-.5,0.15) circle [radius=0.02];
			\node [above] at (-.4,.15) {$k$};
			\node [left] at (-.6,.15) {$X_k=$};
		\end{tikzpicture}
		\end{equation*}
		where $k$ stands for $k$ through strings on the left.
	\end{definition}
	\begin{remark}
		For each $X_k$, we identify it as the same diagram with infinitely many strings on the right. Therefore, we define multiplication $\cdot$ of $X_m$ and $X_n$ as stacking the tangle from bottom to top as the multiplication in $\mathscr{P}_\bullet$. We denote this set as $Alg(X)$ singly generated by $X$. Let $Alg(X)_n$ to be set of all $(1,n)$-tangles in $Alg(X)$.
	\end{remark}

	\begin{proposition}\label{pro:verticalisotopy}
		For $k,n\in\mathbb{N}$ with $k<n$, we have 
		\begin{equation}
		X_n\cdot X_k=X_k\cdot X_{n+N-1} \label{equ:verticalisotopy}
		\end{equation}
	\end{proposition}
	\begin{proof}
		Relation \eqref{equ:verticalisotopy} follows from
		\begin{equation*}
			\begin{tikzpicture}
			\node [below] at (-.2,-.4) {$k$};
			\node [below] at (+1.8,.6) {$n$};
			\node [below] at (-.1+5,.6) {$k$};
			\node [below] at (.2+5+1.8,-.4) {$n+N-1$};
			\draw [fill,red] (-.2,-.4) circle [radius=.1];
			\draw [fill,green] (-.2+1.8,.6) circle [radius=.1];
				\draw [fill,red] (-.2+5,.6) circle [radius=.1];
				\draw [fill,green] (-.2+1.8+5,-.4) circle [radius=.1];
			
			\draw (0,0)--(1,0);
			\draw (0,0)--(0,.5);
			\draw (0,.5)--(1,.5);
			\draw (1,.5)--(1,0);
			\node at (.5,.25) {$X$};
			
			\draw (1.8+0,.75)--(1.8+1,.75);
			\draw (1.8+0,.75)--(1.8+0,1.25);
			\draw (1.8+0,1.25)--(1.8+1,1.25);
			\draw (1.8+1,1.25)--(1.8+1,.75);
			\node at (1.8+.5,1) {$X$};
			
			\draw (-.2,-.4)--(-.2,1.65);
			\draw [fill] (-.3,.25) circle [radius=0.02];
			\draw [fill] (-.4,0.25) circle [radius=0.02];
			\draw [fill] (-.5,0.25) circle [radius=0.02];
			\draw [fill] (-.3,1) circle [radius=0.02];
			\draw [fill] (-.4,1) circle [radius=0.02];
			\draw [fill] (-.5,1) circle [radius=0.02];
			\draw (-.6,-.4)--(-.6,1.65);
			
				\draw (1.8+-.2,-.4)--(1.8+-.2,1.65);
				\draw [fill] (1.8+-.3,.25) circle [radius=0.02];
				\draw [fill] (1.8+-.4,0.25) circle [radius=0.02];
				\draw [fill] (1.8+-.5,0.25) circle [radius=0.02];
				\draw [fill] (1.8+-.3,1) circle [radius=0.02];
				\draw [fill] (1.8+-.4,1) circle [radius=0.02];
				\draw [fill] (1.8+-.5,1) circle [radius=0.02];
				\draw (1.8+-.6,-.4)--(1.8+-.6,1.65);
			
				\draw (3.6+-.2,-.4)--(3.6+-.2,1.65);
				\draw [fill] (3.6+-.3,.25) circle [radius=0.02];
				\draw [fill] (3.6+-.4,0.25) circle [radius=0.02];
				\draw [fill] (3.6+-.5,0.25) circle [radius=0.02];
				\draw [fill] (3.6+-.3,1) circle [radius=0.02];
				\draw [fill] (3.6+-.4,1) circle [radius=0.02];
				\draw [fill] (3.6+-.5,1) circle [radius=0.02];
				\draw (3.6+-.6,-.4)--(3.6+-.6,1.65);
				
					\draw (.2,0)--(.2,-.4);
					\draw (.4,0)--(.4,-.4);
					\draw (.8,0)--(.8,-.4);
						\draw [fill] (.5,-0.2) circle [radius=0.02];
						\draw [fill] (.6,-0.2) circle [radius=0.02];
						\draw [fill] (.7,-0.2) circle [radius=0.02];
						\draw (.5,.5)--(.5,1.65);
						
					\draw (1.8+.2,.75)--(1.8+.2,-.4);
					\draw (1.8+.4,.75)--(1.8+.4,-.4);
					\draw (1.8+.8,.75)--(1.8+.8,-.4);
					\draw [fill] (1.8+.5,-0.2) circle [radius=0.02];
					\draw [fill] (1.8+.6,-0.2) circle [radius=0.02];
					\draw [fill] (1.8+.7,-0.2) circle [radius=0.02];
					\draw (1.8+.5,1.25)--(1.8+.5,1.65);
				
				\draw [dashed, blue, thick] (-.8,-.4)--(3.6,-.4);
				\draw [dashed, blue, thick] (-.8,.6)--(3.6,.6);
				\draw [dashed, blue, thick] (-.8,1.65)--(3.6,1.65);
				
				\node [right] at (3.6,.625) {$=$};
			
			\draw (5+0,0+.75)--(5+1,0+.75);
			\draw (5+0,0+.75)--(5+0,.5+.75);
			\draw (5+0,.5+.75)--(5+1,.5+.75);
			\draw (5+1,.5+.75)--(5+1,0+.75);
			\node at (5+.5,.25+.75) {$X$};
			
			\draw (5+1.8+0,.75-.75)--(5+1.8+1,.75-.75);
			\draw (5+1.8+0,.75-.75)--(5+1.8+0,1.25-.75);
			\draw (5+1.8+0,1.25-.75)--(5+1.8+1,1.25-.75);
			\draw (5+1.8+1,1.25-.75)--(5+1.8+1,.75-.75);
			\node at (5+1.8+.5,1-.75) {$X$};
			
			\draw (5+-.2,-.4)--(5+-.2,1.65);
			\draw [fill] (5+-.3,.25) circle [radius=0.02];
			\draw [fill] (5+-.4,0.25) circle [radius=0.02];
			\draw [fill] (5+-.5,0.25) circle [radius=0.02];
			\draw [fill] (5+-.3,1) circle [radius=0.02];
			\draw [fill] (5+-.4,1) circle [radius=0.02];
			\draw [fill] (5+-.5,1) circle [radius=0.02];
			\draw (5+-.6,-.4)--(5+-.6,1.65);
			
			\draw (5+1.8+-.2,-.4)--(5+1.8+-.2,1.65);
			\draw [fill] (5+1.8+-.3,.25) circle [radius=0.02];
			\draw [fill] (5+1.8+-.4,0.25) circle [radius=0.02];
			\draw [fill] (5+1.8+-.5,0.25) circle [radius=0.02];
			\draw [fill] (5+1.8+-.3,1) circle [radius=0.02];
			\draw [fill] (5+1.8+-.4,1) circle [radius=0.02];
			\draw [fill] (5+1.8+-.5,1) circle [radius=0.02];
			\draw (5+1.8+-.6,-.4)--(5+1.8+-.6,1.65);
			
			\draw (5+3.6+-.2,-.4)--(5+3.6+-.2,1.65);
			\draw [fill] (5+3.6+-.3,.25) circle [radius=0.02];
			\draw [fill] (5+3.6+-.4,0.25) circle [radius=0.02];
			\draw [fill] (5+3.6+-.5,0.25) circle [radius=0.02];
			\draw [fill] (5+3.6+-.3,1) circle [radius=0.02];
			\draw [fill] (5+3.6+-.4,1) circle [radius=0.02];
			\draw [fill] (5+3.6+-.5,1) circle [radius=0.02];
			\draw (5+3.6+-.6,-.4)--(5+3.6+-.6,1.65);
			
			\draw (5+.2,.75)--(5+.2,-.4);
			\draw (5+.4,.75)--(5+.4,-.4);
			\draw (5+.8,.75)--(5+.8,-.4);
			\draw [fill] (5+.5,-0.2) circle [radius=0.02];
			\draw [fill] (5+.6,-0.2) circle [radius=0.02];
			\draw [fill] (5+.7,-0.2) circle [radius=0.02];
			\draw (5+.5,1.25)--(5+.5,1.65);
			
			\draw (5+1.8+.2,0)--(5+1.8+.2,-.4);
			\draw (5+1.8+.4,0)--(5+1.8+.4,-.4);
			\draw (5+1.8+.8,0)--(5+1.8+.8,-.4);
			\draw [fill] (5+1.8+.5,-0.2) circle [radius=0.02];
			\draw [fill] (5+1.8+.6,-0.2) circle [radius=0.02];
			\draw [fill] (5+1.8+.7,-0.2) circle [radius=0.02];
			\draw (5+1.8+.5,.5)--(5+1.8+.5,1.65);
			
			\draw [dashed, blue, thick] (5+-.8,-.4)--(5+3.6,-.4);
			\draw [dashed, blue, thick] (5+-.8,.6)--(5+3.6,.6);
			\draw [dashed, blue, thick] (5+-.8,1.65)--(5+3.6,1.65);
			
			\end{tikzpicture}
		\end{equation*}
	\end{proof}
	
	Motivated by the pair of binary trees representation of Thompson group $F$, we consider the group consisting of the pairs of elements of a certain type from $Alg(X)$. 
	\begin{definition}\label{def:equivalence}
		Let $\widehat{G_X}=\{(T_+,T_-): T_\pm\in Alg(X)_n~\forall n\in\mathbb{N}\}$. We define a relation on $\widehat{G_X}$ by 
		\begin{align*}
		&(T_+,T_-)\sim(S_+,S_-)\Leftrightarrow\\
		\exists R\in Alg(X) &\text{~such that~} T_\pm=S_\pm\cdot R \text{~or~} S_\pm=T_\pm\cdot R.
		\end{align*}
	\end{definition}
	
	\begin{proposition}\label{pro:stablisation}
		Suppose $T\in Alg(X)_n, S\in Alg(X)_m$ for some $n,m\in\mathbb{N}$. There exists $P,Q\in Alg(X)$ such that $T\cdot P=S\cdot Q$. 
	\end{proposition}
	\begin{proof}
		Set $R_k=\prod_{j=0}^{k} X_{jN}$ an $(k+1,(k+1)N)$-tangle.
		We consider the elements $C_n=\prod_{j=0}^{n} R_j\in Alg(X)_{(n+1)N}$ illustrated as follows:
		\begin{equation*}
		\begin{tikzpicture}
		\node [left] at (-2.2,-1.4) {\large $C_n=$};
			\draw (0,0)--(1,0);
			\draw (0,0)--(0,.5);
			\draw (0,.5)--(1,.5);
			\draw (1,.5)--(1,0);
			\node at (.5,.25) {$X$};
			\draw (.5,.5)--(.5,.7);
			\draw (.2,0)--(.2,-.4);
			\draw (.4,0)--(.4,-.4);
			\draw (.8,0)--(.8,-.4);
			\draw [fill] (.5,-0.2) circle [radius=0.02];
			\draw [fill] (.6,-0.2) circle [radius=0.02];
			\draw [fill] (.7,-0.2) circle [radius=0.02];
				\draw (-2+0,0-1.5)--(-2+1,0-1.5);
				\draw (-2+0,0-1.5)--(-2+0,.5-1.5);
				\draw (-2+0,.5-1.5)--(-2+1,.5-1.5);
				\draw (-2+1,.5-1.5)--(-2+1,0-1.5);
				\node at (-2+.5,.25-1.5) {$X$};
				\draw (-2+.5,.5-1.5)--(-2+.5,.7-1.5);
				\draw (-2+.2,0-1.5)--(-2+.2,-.4-1.5);
				\draw (-2+.4,0-1.5)--(-2+.4,-.4-1.5);
				\draw (-2+.8,0-1.5)--(-2+.8,-.4-1.5);
				\draw [fill] (-2+.5,-0.2-1.5) circle [radius=0.02];
				\draw [fill] (-2+.6,-0.2-1.5) circle [radius=0.02];
				\draw [fill] (-2+.7,-0.2-1.5) circle [radius=0.02];
				
					\draw (0,0-1.5)--(1,0-1.5);
					\draw (0,0-1.5)--(0,.5-1.5);
					\draw (0,.5-1.5)--(1,.5-1.5);
					\draw (1,.5-1.5)--(1,0-1.5);
					\node at (.5,.25-1.5) {$X$};
					\draw (.5,.5-1.5)--(.5,.7-1.5);
					\draw (.2,0-1.5)--(.2,-.4-1.5);
					\draw (.4,0-1.5)--(.4,-.4-1.5);
					\draw (.8,0-1.5)--(.8,-.4-1.5);
					\draw [fill] (.5,-0.2-1.5) circle [radius=0.02];
					\draw [fill] (.6,-0.2-1.5) circle [radius=0.02];
					\draw [fill] (.7,-0.2-1.5) circle [radius=0.02];
					
					\draw (2+0,0-1.5)--(2+1,0-1.5);
					\draw (2+0,0-1.5)--(2+0,.5-1.5);
					\draw (2+0,.5-1.5)--(2+1,.5-1.5);
					\draw (2+1,.5-1.5)--(2+1,0-1.5);
					\node at (2+.5,.25-1.5) {$X$};
					\draw (2+.5,.5-1.5)--(2+.5,.7-1.5);
					\draw (2+.2,0-1.5)--(2+.2,-.4-1.5);
					\draw (2+.4,0-1.5)--(2+.4,-.4-1.5);
					\draw (2+.8,0-1.5)--(2+.8,-.4-1.5);
					\draw [fill] (2+.5,-0.2-1.5) circle [radius=0.02];
					\draw [fill] (2+.6,-0.2-1.5) circle [radius=0.02];
					\draw [fill] (2+.7,-0.2-1.5) circle [radius=0.02];
					
					\draw [fill] (1.2,.25-1.5) circle [radius=0.02];
					\draw [fill] (1.3,.25-1.5) circle [radius=0.02];
					\draw [fill] (1.4,.25-1.5) circle [radius=0.02];
					\draw [fill] (1.5,.25-1.5) circle [radius=0.02];
					\draw [fill] (1.6,.25-1.5) circle [radius=0.02];
					\draw [fill] (1.7,.25-1.5) circle [radius=0.02];
					
			
			\draw [dashed, blue, thick] (-2,.7-1.5)--(3,.7-1.5);
			\draw (.2,-.4)--(-2+.5,.7-1.5);
			\draw (.4,-.4)--(.5,.7-1.5);
			\draw (.8,-.4)--(2+.5,.7-1.5);
			\draw [dashed, blue, thick] (-2,-.4-1.5)--(3,-.4-1.5);
		    \draw [dashed, blue, thick] (-2,-.4-2)--(3,-.4-2);
		    \draw [fill] (.5,-.4-1.6) circle [radius=.02];
			\draw [fill] (.5,-.4-1.7) circle [radius=.02];
			\draw [fill] (.5,-.4-1.8) circle [radius=.02];
			\draw [fill] (.5,-.4-1.9) circle [radius=.02];
			
				\draw (-2+0,0-3.1)--(-2+1,0-3.1);
				\draw (-2+0,0-3.1)--(-2+0,.5-3.1);
				\draw (-2+0,.5-3.1)--(-2+1,.5-3.1);
				\draw (-2+1,.5-3.1)--(-2+1,0-3.1);
				\node at (-2+.5,.25-3.1) {$X$};
				\draw (-2+.5,.5-3.1)--(-2+.5,.7-3.1);
				\draw (-2+.2,0-3.1)--(-2+.2,-.4-3.1);
				\draw (-2+.4,0-3.1)--(-2+.4,-.4-3.1);
				\draw (-2+.8,0-3.1)--(-2+.8,-.4-3.1);
				\draw [fill] (-2+.5,-0.2-3.1) circle [radius=0.02];
				\draw [fill] (-2+.6,-0.2-3.1) circle [radius=0.02];
				\draw [fill] (-2+.7,-0.2-3.1) circle [radius=0.02];
				
				\draw (-.7+0,0-3.1)--(-.7+1,0-3.1);
				\draw (-.7+0,0-3.1)--(-.7+0,.5-3.1);
				\draw (-.7+0,.5-3.1)--(-.7+1,.5-3.1);
				\draw (-.7+1,.5-3.1)--(-.7+1,0-3.1);
				\node at (-.7+.5,.25-3.1) {$X$};
				\draw (-.7+.5,.5-3.1)--(-.7+.5,.7-3.1);
				\draw (-.7+.2,0-3.1)--(-.7+.2,-.4-3.1);
				\draw (-.7+.4,0-3.1)--(-.7+.4,-.4-3.1);
				\draw (-.7+.8,0-3.1)--(-.7+.8,-.4-3.1);
				\draw [fill] (-.7+.5,-0.2-3.1) circle [radius=0.02];
				\draw [fill] (-.7+.6,-0.2-3.1) circle [radius=0.02];
				\draw [fill] (-.7+.7,-0.2-3.1) circle [radius=0.02];
				
				\draw (.5+0,0-3.1)--(.5+1,0-3.1);
				\draw (.5+0,0-3.1)--(.5+0,.5-3.1);
				\draw (.5+0,.5-3.1)--(.5+1,.5-3.1);
				\draw (.5+1,.5-3.1)--(.5+1,0-3.1);
				\node at (.5+.5,.25-3.1) {$X$};
				\draw (.5+.5,.5-3.1)--(.5+.5,.7-3.1);
				\draw (.5+.2,0-3.1)--(.5+.2,-.4-3.1);
				\draw (.5+.4,0-3.1)--(.5+.4,-.4-3.1);
				\draw (.5+.8,0-3.1)--(.5+.8,-.4-3.1);
				\draw [fill] (.5+.5,-0.2-3.1) circle [radius=0.02];
				\draw [fill] (.5+.6,-0.2-3.1) circle [radius=0.02];
				\draw [fill] (.5+.7,-0.2-3.1) circle [radius=0.02];

				\draw (2+0,0-3.1)--(2+1,0-3.1);
				\draw (2+0,0-3.1)--(2+0,.5-3.1);
				\draw (2+0,.5-3.1)--(2+1,.5-3.1);
				\draw (2+1,.5-3.1)--(2+1,0-3.1);
				\node at (2+.5,.25-3.1) {$X$};
				\draw (2+.5,.5-3.1)--(2+.5,.7-3.1);
				\draw (2+.2,0-3.1)--(2+.2,-.4-3.1);
				\draw (2+.4,0-3.1)--(2+.4,-.4-3.1);
				\draw (2+.8,0-3.1)--(2+.8,-.4-3.1);
				\draw [fill] (2+.5,-0.2-3.1) circle [radius=0.02];
				\draw [fill] (2+.6,-0.2-3.1) circle [radius=0.02];
				\draw [fill] (2+.7,-0.2-3.1) circle [radius=0.02];

				\draw [fill] (1.6,.25-3.1) circle [radius=0.02];
				\draw [fill] (1.7,.25-3.1) circle [radius=0.02];
				\draw [fill] (1.8,.25-3.1) circle [radius=0.02];
				\draw [fill] (1.9,.25-3.1) circle [radius=0.02];	
		\end{tikzpicture}
		\end{equation*}
		
		Suppose $T\in Alg(X)_n$, there exists $k_1\in\mathbb{N}$ such that $T$ is a sub diagram of $C_{k_1}$, i.e, there exists $P_1\in Alg(X)$ such that $T\cdot P_1=C_{k_1}$. Similarly we obtain such $k_2\in\mathbb{N}$ and $Q_1\in Alg(X)$ for $S\in Alg(X)_m$.
		
		Set $k=\max(k_1,k_2)$. If $k_1=K_2$, then $T\cdot P_1=S\cdot Q_1$. If $k_1\neq k_2$, then WLOG we assume $k_1>k_2$. Set $Q_2=\prod_{j=k_2+1}^{k_1} R_j$, then 
		\begin{align*}
		T\cdot P_1&=C_{k_1}\\
		S\cdot Q_1\cdot Q_2&=C_{k_2}\cdot \prod_{j=k_2+1}^{k_1} R_j=C_{k_1}
		\end{align*}
		Therefore, $P=P_1$ and $Q=Q_1\cdot Q_2$ satisfies the requirement of the proposition.   
	\end{proof}
	\begin{corollary}\label{cor:multi}
		Suppose $(T_+,T_-), (S_+,S_-)\in\widehat{G_X}$, there exists $(\widetilde{T_+},\widetilde{T_-}),(\widetilde{S_+},\widetilde{S_-})\in\widehat{G_X}$ such that 
		\begin{align}
		(\widetilde{T_+},\widetilde{T_-})&\sim(T_+,T_-)\label{equ:TT}\\
		(\widetilde{S_+},\widetilde{S_-})&\sim(S_+,S_-)\label{equ:SS}\\
		\widetilde{T_-}&=\widetilde{S_+}\label{equ:TS}    
		\end{align}
	\end{corollary}
	\begin{proof}
		By Proposition \ref{pro:stablisation}, there exists $P,Q\in Alg(X)$ such that $T_-\cdot P=S_+\cdot Q$. We define $\widetilde{T_\pm}=T_\pm\cdot P$ and $\widetilde{S_\pm}=S_\pm\cdot Q$. It follows from definitions that $(\widetilde{T_+},\widetilde{T_-}),(\widetilde{S_+},\widetilde{S_-})\in\widehat{G_X}$ satisfy Relations \eqref{equ:TT}, \eqref{equ:SS} and \eqref{equ:TS}.
	\end{proof}
	
	Let $G_X=\{(T_+,T_-):T_\pm\in Alg(X)_n~\forall n\in\mathbb{N}\}/\sim$, where $\sim$ is the equivalence relation in Definition \ref{def:equivalence}. Suppose $g,h\in G_X$, then there exists $(T,S),(S,R)\in\widehat{G_X}$ such that $g=[(T,S)], h=[(S,R)]$ by Corollary \ref{cor:multi}. Thus we define a binary operation $\circ$ as
	\begin{equation*}
	g\circ h=[(T,R)]
	\end{equation*}
	\begin{theorem}
		The binary operation $\circ$ is well defined on $G_X$ and $G_X$ is a group with the binary operation $\circ$.
	\end{theorem}
	\begin{proof}
		Suppose $(T,S)\sim(\widetilde{T},\widetilde{S})$ and $(S,R)\sim(\widetilde{S},\widetilde{R})$, then there exists $P\in Alg(X)$ such that 
		\begin{align*}
		\widetilde{T}=T\cdot P\\
		\widetilde{S}=S\cdot P\\
		\widetilde{R}=R\cdot P
		\end{align*}
		Therefore $(T,R)\sim(\widetilde{T},\widetilde{R})$, i.e, $\circ$ is well defined on $G_X$.
		
		Suppose $T,S\in Alg(X)_n$ for some $n\in\mathbb{N}$. It follows that $[(T,T)]$ is the identity element with respect to $\circ$. For the element $[(T,S)]$, $[(S,T)]$ is the inverse element. By definition, $G_X$ is closed under the binary operation $\circ$. Therefore $G_X$ is a group with $\circ$.
		
	\end{proof}
	\begin{remark}
		We omit the $\circ$ when there is no confusion. Since the algebra $Alg(X)$ is singly generated by $X$, therefore we denote such groups $G_X$ by singly generated groups. Furthermore, $[(T,R)]$ equals to the identity if and only if $T$ is isotopically equivalent of $R$ for $T,R\in Alg(X)_n$ with $n\in\mathbb{N}$.
	\end{remark}
	\begin{notation}
		In the following sections, we denote $(T,R)$ for $T,R\in Alg(X)_n$ for the equivalence class of $[(T,R)]$ for elements in $G_X$.
	\end{notation}

	\section{The Classical presentation}\label{sec:presentation}
	In this section we discuss the structure of $G_X$ and its classical presentation derived from the vertical isotopy.
	\begin{definition}\label{def:basicform}
		Suppose $X$ is an $(1,N)$-tangle. Set $S_n=X_0\cdot X_{N-1}\cdot X_{2(N-1)} \cdots X_{n(N-1)}$. We call these $S_n$'s basic forms and illustrate them as following:
		\begin{equation*}
		\begin{tikzpicture}
		\draw (0,0)--(1,0);
		\draw (0,0)--(0,.5);
		\draw (0,.5)--(1,.5);
		\draw (1,.5)--(1,0);
		\node at (.5,.25) {$X$};
		\draw (.5,.5)--(.5,.7);
		\draw (.2,0)--(.2,-.4);
		\draw (.4,0)--(.4,-.4);
		\draw (.8,0)--(.8,-.4);
		\draw [fill] (.5,-0.2) circle [radius=0.02];
		\draw [fill] (.6,-0.2) circle [radius=0.02];
		\draw [fill] (.7,-0.2) circle [radius=0.02];
		\draw (.3+0,0-1.5)--(.3+1,0-1.5);
		\draw (.3+0,0-1.5)--(.3+0,.5-1.5);
		\draw (.3+0,.5-1.5)--(.3+1,.5-1.5);
		\draw (.3+1,.5-1.5)--(.3+1,0-1.5);
		\node at (.3+.5,.25-1.5) {$X$};
		\draw (.3+.5,.5-1.5)--(.3+.5,.7-1.5);
		\draw (.3+.2,0-1.5)--(.3+.2,-.4-1.5);
		\draw (.3+.4,0-1.5)--(.3+.4,-.4-1.5);
		\draw (.3+.8,0-1.5)--(.3+.8,-.4-1.5);
		\draw [fill] (.3+.5,-0.2-1.5) circle [radius=0.02];
		\draw [fill] (.3+.6,-0.2-1.5) circle [radius=0.02];
		\draw [fill] (.3+.7,-0.2-1.5) circle [radius=0.02];
		\draw (.8,-.4)--(.8,.7-1.5);
		
		\draw (.2,-.4)--(-.6,.7-1.5);
		\draw (.4,-.4)--(-.4,.7-1.5);
		
		\draw (-.4,.7-1.5)--(-.4,-.4-1.5);
		\draw (-.6,.7-1.5)--(-.6,-.4-1.5);
		\draw (-.2+.3,.7-1.5)--(-.2+.3,-.4-1.5);
		\draw [dashed, blue, thick] (-1,.7-1.5)--(1.5,.7-1.5);		
		\draw [dashed, blue, thick] (-1,-.4-1.5)--(1.5,.-.4-1.5);
		\draw [dashed, blue, thick] (-1,-.4-2)--(1.5,.-.4-2);
		\draw [fill] (.25,-.4-1.6) circle [radius=.02];
		\draw [fill] (.25,-.4-1.7) circle [radius=.02];
		\draw [fill] (.25,-.4-1.8) circle [radius=.02];
		\draw [fill] (.25,-.4-1.9) circle [radius=.02];
		
		\draw [fill] (-.7+.5,.25-1.5) circle [radius=0.02];
		\draw [fill] (-.7+.6,.25-1.5) circle [radius=0.02];
		\draw [fill] (-.7+.7,.25-1.5) circle [radius=0.02];
		\draw [fill] (-.7+.4,.25-1.5) circle [radius=0.02];
		
			\draw (.6+0,0-3.1)--(.6+1,0-3.1);
			\draw (.6+0,0-3.1)--(.6+0,.5-3.1);
			\draw (.6+0,.5-3.1)--(.6+1,.5-3.1);
			\draw (.6+1,.5-3.1)--(.6+1,0-3.1);
			\node at (.6+.5,.25-3.1) {$X$};
			\draw (.6+.5,.5-3.1)--(.6+.5,.7-3.1);
			\draw (.6+.2,0-3.1)--(.6+.2,-.4-3.1);
			\draw (.6+.4,0-3.1)--(.6+.4,-.4-3.1);
			\draw (.6+.8,0-3.1)--(.6+.8,-.4-3.1);
			\draw [fill] (.6+.5,-0.2-3.1) circle [radius=0.02];
			\draw [fill] (.6+.6,-0.2-3.1) circle [radius=0.02];
			\draw [fill] (.6+.7,-0.2-3.1) circle [radius=0.02];
			\draw (.6-.2,.7-3.1)--(.6-.2,-.4-3.1);
			\draw (-.4-.2,.7-3.1)--(-.4-.2,-.4-3.1);
			\draw (-.2-.2,.7-3.1)--(-.2-.2,-.4-3.1);
			
			\draw [fill] (-.5+.5,.25-3.1) circle [radius=0.02];
			\draw [fill] (-.5+.6,.25-3.1) circle [radius=0.02];
			\draw [fill] (-.5+.7,.25-3.1) circle [radius=0.02];
			\draw [fill] (-.5+.4,.25-3.1) circle [radius=0.02];
		\end{tikzpicture}
		\end{equation*}    
	\end{definition}
	\begin{proposition}\label{pro:form of T}
		Suppose $T\in Alg(X)_n$, there exists $\alpha(T)\in\mathbb{N}$ and $X_T\in Alg(X)$ which is $(\alpha(T)(N-1)+1,n)$-tangle such that 
		\begin{equation*}
		\begin{tikzpicture}
		\node [left] (-.2,-.3) {$T=$};
		\draw (0,0)--(2,0);
		\draw (2,0)--(2,.6);
		\draw (2,.6)--(0,.6);
		\draw (0,.6)--(0,0);
		\node at (1,.3) {$S_{\alpha(T)}$};
		\draw (1,.6)--(1,.8);
		\draw (.2,0)--(.2,-.4);
		\draw (.4,0)--(.4,-.4);
		\draw (.6,0)--(.6,-.4);	
		\draw (1.4,0)--(1.4,-.4);
		\draw (1.8,0)--(1.8,-1.4);
		\draw (0,-.4)--(1.6,-.4);
		\draw (0,-1)--(1.6,-1);
		\draw (0,-1)--(0,-.4);
		\draw (1.6,-1)--(1.6,-.4);
		\draw (.2,-1)--(.2,-1.4);
		\draw (.4,-1)--(.4,-1.4);
		\draw (.6,-1)--(.6,-1.4);	
		\draw (1.4,-1)--(1.4,-1.4);
		\node at (.8,-.7) {$X_T$};
		\draw [fill] (.8,-.2) circle [radius=.02];
		\draw [fill] (.9,-.2) circle [radius=.02];
		\draw [fill] (1,-.2) circle [radius=.02];
		\draw [fill] (1.1,-.2) circle [radius=.02];
		\draw [fill] (1.2,-.2) circle [radius=.02];	
		\draw [fill] (.8,-.2-1) circle [radius=.02];
		\draw [fill] (.9,-.2-1) circle [radius=.02];
		\draw [fill] (1,-.2-1) circle [radius=.02];
		\draw [fill] (1.1,-.2-1) circle [radius=.02];
		\draw [fill] (1.2,-.2-1) circle [radius=.02];	
		\end{tikzpicture}
		\end{equation*}	
	\end{proposition}
	\begin{proof}
		Since $T\in Alg(X)_n$, the first word in $T$ is $X_0$. If there exists an $X$ such that it is attached the rightmost string on the bottom of $X_0$, then apply the vertical isotopy to obtain a diagram starting with $X_0\cdot X_{N-1}$. Repeating this procedure and let $\alpha(T)$ be the number of the steps. Therefore $T=S_{\alpha(T)}\cdot X_T$, where $X_T$ the rest of the word. 
	\end{proof}
	\begin{lemma}\label{lem:positive}
		Let $P_X=\{(T,S_n):T\in Alg(X)_n,~\forall n\in\mathbb{N}\}$. Then $P_X$ is a semigroup under $\circ$. Furthermore, it generates the group $G_X$. 
	\end{lemma}
	\begin{proof}
		Suppose $g,h\in P_X$ and $g=(T,S_n), h=(R,S_m)$. 
		\begin{align*}
		gh=&(S_{\alpha(T)}\cdot X_T,S_n)\circ (S_{\alpha(R)}\cdot X_R, S_m)\\
		=&(S_{\alpha(T)+\max(n,\alpha(R))-n}\cdot X_T, S_{\max(n,\alpha(R))})\circ\\
		&(S_{\max(n,\alpha(R))}\cdot X_R, S_{m+\max(n,\alpha(R))-\alpha(R)})\\
		=&(S_{\alpha(T)+\max(n,\alpha(R))-n}\cdot X_T\cdot X_R, S_{m+\max(n,\alpha(R))-\alpha(R)})\in P_X
		\end{align*}
		Therefore $P_X$ is a semigroup under the binary operation $\circ$. 
		
		Note that for $(T,R)\in G_X$, $(T,R)=(T,S_n)(R,S_n)^{-1}$ for some $n\in\mathbb{N}$. Hence $P_X$ is a generating set for $P_X$.     
	\end{proof}
	
	Now we give a description of classical generators for the group $G_X$. 
	\begin{definition}\label{def:xk}
		Suppose $n\in\mathbb{N}$, there exists $a(n),b(n)\in\mathbb{N}$ such that $a(n)$ is the largest integer satisfying $(N-1)a(n)+b(n)=n$ with $0\leq b(n)<N$. We define 
		\begin{equation*}
		x_n=(S_{a(n)}\cdot X_n,S_{a(n)+1})
		\end{equation*}
	\end{definition}
	\begin{lemma}\label{lem:Relation} The set $\{x_n:n\in\mathbb{N}\}$ is a generating set for $G_X$ and satisfies the relation 
		\begin{equation}\label{equ:Relation}
		x_k^{-1}x_nx_k=x_{n+N-1}
		\end{equation}
	\end{lemma}
	\begin{proof}
		By Lemma \ref{lem:positive}, we only need to show every element $g\in P_X$ belongs to the subgroup generated by $\{x_n:n\in\mathbb{N}\}$. Suppose $g=(T,S_n)\in P_X$. By Proposition \ref{pro:form of T}, $X_T=X_{T^\prime}\cdot X_k$ for some $k\in\mathbb{N}$.
		\begin{equation*}
		\begin{tikzpicture}
		\draw (0,0)--(1.2,0);
		\draw (1.2,0)--(1.2,.6);
		\draw (1.2,.6)--(0,.6);
		\draw (0,.6)--(0,0);
		\draw (.6,.6)--(.6,.8);
		\node at (.6,.3) {$S_{\alpha(T)}$};
		\draw (.1,0)--(.1,-.4);
		\draw (.3,0)--(.3,-.4);
		\draw (.5,0)--(.5,-.4);
		\draw (.9,0)--(.9,-.4);
		\draw [fill] (.6,-.2) circle[radius=.01];
		\draw [fill] (.7,-.2) circle[radius=.01];
		\draw [fill] (.8,-.2) circle[radius=.01];
		\draw (1.1,0)--(1.1,-1.75);
		\draw (0,-.4)--(1.0,-.4);
		\draw (1.0,-.4)--(1.0,-.8);
		\draw (1,-.8)--(0,-.8);
		\draw (0,-.8)--(0,-.4);
		\node at (.5,-.6) {$X_{T^\prime}$};
		\draw (.5,-.8)--(.5,-1.2);
		\draw (.25,-1.2)--(.75,-1.2);
		\draw (.25,-1.2)--(.25,-1.45);
		\draw (.75,-1.2)--(.75,-1.45);
		\draw (.75,-1.45)--(.25,-1.45);
		\node at (.5,-1.325) {\small $X$};
		\draw (.35,-1.45)--(.35,-1.75);
		\draw (.45,-1.45)--(.45,-1.75);
		\draw (.65,-1.45)--(.65,-1.75);
		\draw [fill] (.5,-1.6) circle [radius=.01];
		\draw [fill] (.55,-1.6) circle [radius=.01];
		\draw [fill] (.6,-1.6) circle [radius=.01];
		\draw (.9,-.8)--(.9,-1.75);
		\draw [fill] (.6,-1) circle[radius=.01];
		\draw [fill] (.7,-1) circle[radius=.01];
		\draw [fill] (.8,-1) circle[radius=.01];
		\draw (.1,-.8)--(.1,-1);
		\draw (.1,-1) arc [radius=.1,start angle=0, end angle=-90];
		\draw (0,-1.1) arc [radius=.1, start angle=90, end angle=180];
		\draw (-.1,-1.2)--(-.1,-1.75);
			\draw (.2+.1,-.8)--(.2+.1,-1);
			\draw (.2+.1,-1) arc [radius=.1,start angle=0, end angle=-90];
			\draw (.2+0,-1.1) arc [radius=.1, start angle=90, end angle=180];
			\draw (.2+-.1,-1.2)--(.2+-.1,-1.75);
			\draw [fill] (-.05,-1.6) circle [radius=.01];
			\draw [fill] (0,-1.6) circle [radius=.01];
			\draw [fill] (.05,-1.6) circle [radius=.01];
			\node [above] at (0,-1.6) {\small $k$};
		
			\node [left] at(-.1,-.6) {$g=($};
		    \node at (1.4,-0.6) {,};
		    \node [right] at (1.4,-.6) {$S_n)=$};
		    
		    \node [left] at(3+-.1,-.6) {$($};
			\draw (3+0,0)--(3+1.2,0);
			\draw (3+1.2,0)--(3+1.2,.6);
			\draw (3+1.2,.6)--(3+0,.6);
			\draw (3+0,.6)--(3+0,0);
			\draw (3+.6,.6)--(3+.6,.8);
			\node at (3+.6,.3) {$S_{\alpha(T)}$};
			\draw (3+.1,0)--(3+.1,-.4);
			\draw (3+.3,0)--(3+.3,-.4);
			\draw (3+.5,0)--(3+.5,-.4);
			\draw (3+.9,0)--(3+.9,-.4);
			\draw [fill] (3+.6,-.2) circle[radius=.01];
			\draw [fill] (3+.7,-.2) circle[radius=.01];
			\draw [fill] (3+.8,-.2) circle[radius=.01];
			\draw (3+1.1,0)--(3+1.1,-1.75);
			\draw (3+0,-.4)--(3+1.0,-.4);
			\draw (3+1.0,-.4)--(3+1.0,-.8);
			\draw (3+1,-.8)--(3+0,-.8);
			\draw (3+0,-.8)--(3+0,-.4);
			\node at (3+.5,-.6) {$X_{T^\prime}$};
			\draw (3+.5,-.8)--(3+.5,-1.2);
			\draw (3+.25,-1.2)--(3+.75,-1.2);
			\draw (3+.25,-1.2)--(3+.25,-1.45);
			\draw (3+.75,-1.2)--(3+.75,-1.45);
			\draw (3+.75,-1.45)--(3+.25,-1.45);
			\node at (3+.5,-1.325) {\small $X$};
			\draw (3+.35,-1.45)--(3+.35,-1.75);
			\draw (3+.45,-1.45)--(3+.45,-1.75);
			\draw (3+.65,-1.45)--(3+.65,-1.75);
			\draw [fill] (3+.5,-1.6) circle [radius=.01];
			\draw [fill] (3+.55,-1.6) circle [radius=.01];
			\draw [fill] (3+.6,-1.6) circle [radius=.01];
			\draw (3+.9,-.8)--(3+.9,-1.75);
			\draw [fill] (3+.6,-1) circle[radius=.01];
			\draw [fill] (3+.7,-1) circle[radius=.01];
			\draw [fill] (3+.8,-1) circle[radius=.01];
			\draw (3+.1,-.8)--(3+.1,-1);
			\draw (3+.1,-1) arc [radius=.1,start angle=0, end angle=-90];
			\draw (3+0,-1.1) arc [radius=.1, start angle=90, end angle=180];
			\draw (3+-.1,-1.2)--(3+-.1,-1.75);
			\draw (3+.2+.1,-.8)--(3+.2+.1,-1);
			\draw (3+.2+.1,-1) arc [radius=.1,start angle=0, end angle=-90];
			\draw (3+.2+0,-1.1) arc [radius=.1, start angle=90, end angle=180];
			\draw (3+.2+-.1,-1.2)--(3+.2+-.1,-1.75);
			\draw [fill] (3+-.05,-1.6) circle [radius=.01];
			\draw [fill] (3+0,-1.6) circle [radius=.01];
			\draw [fill] (3+.05,-1.6) circle [radius=.01];
			\node [above] at (3+0,-1.6) {\small $k$};
		
		\node at (1.4+3,-0.6) {,};
		\draw (5+0,0)--(5+1.2,0);
		\draw (5+1.2,0)--(5+1.2,.6);
		\draw (5+1.2,.6)--(5+0,.6);
		\draw (5+0,.6)--(5+0,0);
		\draw (5+.6,.6)--(5+.6,.8);
		\node at (5+.6,.3) {$S_{n-1}$};
		\draw (5+.1,0)--(5+.1,-.4);
		\draw (5+.3,0)--(5+.3,-.4);
		\draw (5+.5,0)--(5+.5,-1.2);
		\draw (5+.9,0)--(5+.9,-1.75);
		\draw (5+1.1,0)--(5+1.1,-1.75);
		\draw (5+.25,-1.2)--(5+.75,-1.2);
		\draw (5+.25,-1.2)--(5+.25,-1.45);
		\draw (5+.75,-1.2)--(5+.75,-1.45);
		\draw (5+.75,-1.45)--(5+.25,-1.45);
		\node at (5+.5,-1.325) {\small $X$};
		\draw (5+.35,-1.45)--(5+.35,-1.75);
		\draw (5+.45,-1.45)--(5+.45,-1.75);
		\draw (5+.65,-1.45)--(5+.65,-1.75);
			\draw [fill] (5.5,-1.6) circle [radius=.01];
			\draw [fill] (5.55,-1.6) circle [radius=.01];
			\draw [fill] (5.6,-1.6) circle [radius=.01];
		\draw (5+.1,-.4) arc [radius=.1, start angle=0, end angle=-90];
		\draw (5,-.5) arc [radius=.1, start angle=90, end angle=180];
		\draw (5+.3,-.4) arc [radius=.1, start angle=0, end angle=-90];
		\draw (5+.2,-.5) arc [radius=.1, start angle=90, end angle=180];
		\draw (5-.1,-.6)--(5-.1,-1.75);
		\draw (5+.1,-.6)--(5+.1,-1.75);
			\draw [fill] (5+-.05,-1.6) circle [radius=.01];
			\draw [fill] (5+0,-1.6) circle [radius=.01];
			\draw [fill] (5+.05,-1.6) circle [radius=.01];
			\node [above] at (5+0,-1.6) {\small $k$};

		\draw [fill] (5+.6,-.2) circle[radius=.01];
		\draw [fill] (5+.7,-.2) circle[radius=.01];
		\draw [fill] (5+.8,-.2) circle[radius=.01];
				\node [right] at (5+1.2,-.6) {$)\circ($};
				
			\draw (7.3+0,0)--(7.3+1.2,0);
			\draw (7.3+1.2,0)--(7.3+1.2,.6);
			\draw (7.3+1.2,.6)--(7.3+0,.6);
			\draw (7.3+0,.6)--(7.3+0,0);
			\draw (7.3+.6,.6)--(7.3+.6,.8);
			\node at (7.3+.6,.3) {$S_{n-1}$};
			\draw (7.3+.1,0)--(7.3+.1,-.4);
			\draw (7.3+.3,0)--(7.3+.3,-.4);
			\draw (7.3+.5,0)--(7.3+.5,-1.2);
			\draw (7.3+.9,0)--(7.3+.9,-1.75);
			\draw (7.3+1.1,0)--(7.3+1.1,-1.75);
			\draw (7.3+.25,-1.2)--(7.3+.75,-1.2);
			\draw (7.3+.25,-1.2)--(7.3+.25,-1.45);
			\draw (7.3+.75,-1.2)--(7.3+.75,-1.45);
			\draw (7.3+.75,-1.45)--(7.3+.25,-1.45);
			\node at (7.3+.5,-1.325) {\small $X$};
			\draw (7.3+.35,-1.45)--(7.3+.35,-1.75);
			\draw (7.3+.45,-1.45)--(7.3+.45,-1.75);
			\draw (7.3+.65,-1.45)--(7.3+.65,-1.75);
			\draw [fill] (7.3+.5,-1.6) circle [radius=.01];
			\draw [fill] (7.3+.55,-1.6) circle [radius=.01];
			\draw [fill] (7.3+.6,-1.6) circle [radius=.01];
			\draw (7.3+.1,-.4) arc [radius=.1, start angle=0, end angle=-90];
			\draw (7.3,-.5) arc [radius=.1, start angle=90, end angle=180];
			\draw (7.3+.3,-.4) arc [radius=.1, start angle=0, end angle=-90];
			\draw (7.3+.2,-.5) arc [radius=.1, start angle=90, end angle=180];
			\draw (7.3-.1,-.6)--(7.3-.1,-1.75);
			\draw (7.3+.1,-.6)--(7.3+.1,-1.75);
			\draw [fill] (7.3+-.05,-1.6) circle [radius=.01];
			\draw [fill] (7.3+0,-1.6) circle [radius=.01];
			\draw [fill] (7.3+.05,-1.6) circle [radius=.01];
			\node [above] at (7.3+0,-1.6) {\small $k$};
			
			\draw [fill] (7.3+.6,-.2) circle[radius=.01];
			\draw [fill] (7.3+.7,-.2) circle[radius=.01];
			\draw [fill] (7.3+.8,-.2) circle[radius=.01];
			
			\node [right] at (7.5+1.2,-.6) {$,S_n)$};
		\end{tikzpicture}
		\end{equation*}
		\begin{equation*}
		\begin{tikzpicture}
		\node [left] at (3,-.6) {$=($};
		\draw (3+0,0)--(3+1.2,0);
		\draw (3+1.2,0)--(3+1.2,.6);
		\draw (3+1.2,.6)--(3+0,.6);
		\draw (3+0,.6)--(3+0,0);
		\draw (3+.6,.6)--(3+.6,.8);
		\node at (3+.6,.3) {$S_{\alpha(T)}$};
		\draw (3+.1,0)--(3+.1,-.4);
		\draw (3+.3,0)--(3+.3,-.4);
		\draw (3+.5,0)--(3+.5,-.4);
		\draw (3+.9,0)--(3+.9,-.4);
		\draw [fill] (3+.6,-.2) circle[radius=.01];
		\draw [fill] (3+.7,-.2) circle[radius=.01];
		\draw [fill] (3+.8,-.2) circle[radius=.01];
		\draw (3+1.1,0)--(3+1.1,-1.75);
		\draw (3+0,-.4)--(3+1.0,-.4);
		\draw (3+1.0,-.4)--(3+1.0,-.8);
		\draw (3+1,-.8)--(3+0,-.8);
		\draw (3+0,-.8)--(3+0,-.4);
		\node at (3+.5,-.6) {$X_{T^\prime}$};
		\draw (3+.5,-.8)--(3+.5,-1.75);
	
		\draw (3+.9,-.8)--(3+.9,-1.75);
		\draw [fill] (3+.6,-1) circle[radius=.01];
		\draw [fill] (3+.7,-1) circle[radius=.01];
		\draw [fill] (3+.8,-1) circle[radius=.01];
		\draw (3+.1,-.8)--(3+.1,-1);
		\draw (3+.1,-1) arc [radius=.1,start angle=0, end angle=-90];
		\draw (3+0,-1.1) arc [radius=.1, start angle=90, end angle=180];
		\draw (3+-.1,-1.2)--(3+-.1,-1.75);
		\draw (3+.2+.1,-.8)--(3+.2+.1,-1);
		\draw (3+.2+.1,-1) arc [radius=.1,start angle=0, end angle=-90];
		\draw (3+.2+0,-1.1) arc [radius=.1, start angle=90, end angle=180];
		\draw (3+.2+-.1,-1.2)--(3+.2+-.1,-1.75);
		\draw [fill] (3+-.05,-1.6) circle [radius=.01];
		\draw [fill] (3+0,-1.6) circle [radius=.01];
		\draw [fill] (3+.05,-1.6) circle [radius=.01];
		\node [above] at (3+0,-1.6) {\small $k$};
		
		\node at (1.4+3,-0.6) {,};
		\draw (5+0,0)--(5+1.2,0);
		\draw (5+1.2,0)--(5+1.2,.6);
		\draw (5+1.2,.6)--(5+0,.6);
		\draw (5+0,.6)--(5+0,0);
		\draw (5+.6,.6)--(5+.6,.8);
		\node at (5+.6,.3) {$S_{n-1}$};
		\draw (5+.1,0)--(5+.1,-.4);
		\draw (5+.3,0)--(5+.3,-.4);
		\draw (5+.5,0)--(5+.5,-1.75);
		\draw (5+.9,0)--(5+.9,-1.75);
		\draw (5+1.1,0)--(5+1.1,-1.75);

		\draw (5+.1,-.4) arc [radius=.1, start angle=0, end angle=-90];
		\draw (5,-.5) arc [radius=.1, start angle=90, end angle=180];
		\draw (5+.3,-.4) arc [radius=.1, start angle=0, end angle=-90];
		\draw (5+.2,-.5) arc [radius=.1, start angle=90, end angle=180];
		\draw (5-.1,-.6)--(5-.1,-1.75);
		\draw (5+.1,-.6)--(5+.1,-1.75);
		\draw [fill] (5+-.05,-1.6) circle [radius=.01];
		\draw [fill] (5+0,-1.6) circle [radius=.01];
		\draw [fill] (5+.05,-1.6) circle [radius=.01];
		\node [above] at (5+0,-1.6) {\small $k$};
		
		\draw [fill] (5+.6,-.2) circle[radius=.01];
		\draw [fill] (5+.7,-.2) circle[radius=.01];
		\draw [fill] (5+.8,-.2) circle[radius=.01];
		\node [right] at (5+1.2,-.6) {$)\circ($};
		
			\draw (7.3+0,0)--(7.3+1.2,0);
			\draw (7.3+1.2,0)--(7.3+1.2,.6);
			\draw (7.3+1.2,.6)--(7.3+0,.6);
			\draw (7.3+0,.6)--(7.3+0,0);
			\draw (7.3+.6,.6)--(7.3+.6,.8);
			\node at (7.3+.6,.3) {$S_{n-1}$};
			\draw (7.3+.1,0)--(7.3+.1,-.4);
			\draw (7.3+.3,0)--(7.3+.3,-.4);
			\draw (7.3+.5,0)--(7.3+.5,-1.2);
			\draw (7.3+.9,0)--(7.3+.9,-1.75);
			\draw (7.3+1.1,0)--(7.3+1.1,-1.75);
			\draw (7.3+.25,-1.2)--(7.3+.75,-1.2);
			\draw (7.3+.25,-1.2)--(7.3+.25,-1.45);
			\draw (7.3+.75,-1.2)--(7.3+.75,-1.45);
			\draw (7.3+.75,-1.45)--(7.3+.25,-1.45);
			\node at (7.3+.5,-1.325) {\small $X$};
			\draw (7.3+.35,-1.45)--(7.3+.35,-1.75);
			\draw (7.3+.45,-1.45)--(7.3+.45,-1.75);
			\draw (7.3+.65,-1.45)--(7.3+.65,-1.75);
			\draw [fill] (7.3+.5,-1.6) circle [radius=.01];
			\draw [fill] (7.3+.55,-1.6) circle [radius=.01];
			\draw [fill] (7.3+.6,-1.6) circle [radius=.01];
			\draw (7.3+.1,-.4) arc [radius=.1, start angle=0, end angle=-90];
			\draw (7.3,-.5) arc [radius=.1, start angle=90, end angle=180];
			\draw (7.3+.3,-.4) arc [radius=.1, start angle=0, end angle=-90];
			\draw (7.3+.2,-.5) arc [radius=.1, start angle=90, end angle=180];
			\draw (7.3-.1,-.6)--(7.3-.1,-1.75);
			\draw (7.3+.1,-.6)--(7.3+.1,-1.75);
			\draw [fill] (7.3+-.05,-1.6) circle [radius=.01];
			\draw [fill] (7.3+0,-1.6) circle [radius=.01];
			\draw [fill] (7.3+.05,-1.6) circle [radius=.01];
			\node [above] at (7.3+0,-1.6) {\small $k$};
			
			\draw [fill] (7.3+.6,-.2) circle[radius=.01];
			\draw [fill] (7.3+.7,-.2) circle[radius=.01];
			\draw [fill] (7.3+.8,-.2) circle[radius=.01];
			
			\node [right] at (7.5+1.2,-.6) {$,S_n)$};
		\end{tikzpicture}
		\end{equation*}
	
		By Definition \ref{def:basicform}, $S_n=X_0\cdot X_{N-1}\cdot X_{2(N-1)} \cdots X_{n(N-1)}$. Let $R=X_{(m+1)(N-1)}\cdot X_{n(N-1)}$. By Definition \ref{def:xk}, $a(k)$ is the smallest integer such that $(N-1)a(k)+b(k)=k$ with $0\leq b(k)<N$. 
		
			\begin{equation*}
			\begin{tikzpicture}
			\draw (0+0,0)--(0+1.2,0);
			\draw (0+1.2,0)--(0+1.2,.6);
			\draw (0+1.2,.6)--(0+0,.6);
			\draw (0+0,.6)--(0+0,0);
			\draw (0+.6,.6)--(0+.6,.8);
			\node at (0+.6,.3) {$S_{n-1}$};
			\draw (0+.1,0)--(0+.1,-.4);
			\draw (0+.3,0)--(0+.3,-.4);
			\draw (0+.5,0)--(0+.5,-1.2);
			\draw (0+.9,0)--(0+.9,-1.75);
			\draw (0+1.1,0)--(0+1.1,-1.75);
			\draw (0+.25,-1.2)--(0+.75,-1.2);
			\draw (0+.25,-1.2)--(0+.25,-1.45);
			\draw (0+.75,-1.2)--(0+.75,-1.45);
			\draw (0+.75,-1.45)--(0+.25,-1.45);
			\node at (0+.5,-1.325) {\small $X$};
			\draw (0+.35,-1.45)--(0+.35,-1.75);
			\draw (0+.45,-1.45)--(0+.45,-1.75);
			\draw (0+.65,-1.45)--(0+.65,-1.75);
			\draw [fill] (0+.5,-1.6) circle [radius=.01];
			\draw [fill] (0+.55,-1.6) circle [radius=.01];
			\draw [fill] (0+.6,-1.6) circle [radius=.01];
			\draw (0+.1,-.4) arc [radius=.1, start angle=0, end angle=-90];
			\draw (0,-.5) arc [radius=.1, start angle=90, end angle=180];
			\draw (0+.3,-.4) arc [radius=.1, start angle=0, end angle=-90];
			\draw (0+.2,-.5) arc [radius=.1, start angle=90, end angle=180];
			\draw (0-.1,-.6)--(0-.1,-1.75);
			\draw (0+.1,-.6)--(0+.1,-1.75);
			\draw [fill] (0+-.05,-1.6) circle [radius=.01];
			\draw [fill] (0+0,-1.6) circle [radius=.01];
			\draw [fill] (0+.05,-1.6) circle [radius=.01];
			\node [above] at (0+0,-1.6) {\small $k$};
			
			\draw [fill] (0+.6,-.2) circle[radius=.01];
			\draw [fill] (0+.7,-.2) circle[radius=.01];
			\draw [fill] (0+.8,-.2) circle[radius=.01];
			
			\node [right] at (1.2,-.6) {$=$};
			
			\draw (2+0,0)--(2+1.2,0);
			\draw (2+1.2,0)--(2+1.2,.6);
			\draw (2+1.2,.6)--(2+0,.6);
			\draw (2+0,.6)--(2+0,0);
			\draw (2+.6,.6)--(2+.6,.8);
			\node at (2+.6,.3) {$S_{\alpha(k)}$};
				\draw (2+.1,0)--(2+.1,-.4);
				\draw (2+.3,0)--(2+.3,-.4);
				\draw (2+.5,0)--(2+.5,-1.2);
				\draw (2+.9,0)--(2+.9,-1.75);
				\draw (2+.25,-1.2)--(2+.75,-1.2);
				\draw (2+.25,-1.2)--(2+.25,-1.45);
				\draw (2+.75,-1.2)--(2+.75,-1.45);
				\draw (2+.75,-1.45)--(2+.25,-1.45);
				\node at (2+.5,-1.325) {\small $X$};
				\draw (2+.35,-1.45)--(2+.35,-1.75);
				\draw (2+.45,-1.45)--(2+.45,-1.75);
				\draw (2+.65,-1.45)--(2+.65,-1.75);
				\draw [fill] (2+.5,-1.6) circle [radius=.01];
				\draw [fill] (2+.55,-1.6) circle [radius=.01];
				\draw [fill] (2+.6,-1.6) circle [radius=.01];
				\draw (2+.1,-.4) arc [radius=.1, start angle=0, end angle=-90];
				\draw (2,-.5) arc [radius=.1, start angle=90, end angle=180];
				\draw (2+.3,-.4) arc [radius=.1, start angle=0, end angle=-90];
				\draw (2+.2,-.5) arc [radius=.1, start angle=90, end angle=180];
				\draw (2-.1,-.6)--(2-.1,-1.75);
				\draw (2+.1,-.6)--(2+.1,-1.75);
				\draw [fill] (2+-.05,-1.6) circle [radius=.01];
				\draw [fill] (2+0,-1.6) circle [radius=.01];
				\draw [fill] (2+.05,-1.6) circle [radius=.01];
				\node [above] at (2+0,-1.6) {\small $k$};
				
				\draw (2+1.1,0)--(2+1.1,-.4);
				\draw (2+1.1,-.4)--(2+1.5,-.6);
				\draw (2+1.1,-.6)--(2+1.9,-.6);
				\draw (2+1.9,-.6)--(2+1.9,-1);
				\draw (2+1.9,-1)--(2+1.1,-1);
				\draw (2+1.1,-1)--(2+1.1,-.6);
				\node at (2+1.5,-.8) {\small $R$};
				\draw (2+1.2,-1)--(2+1.2,-1.75);
				\draw (2+1.4,-1)--(2+1.4,-1.75);
				\draw (2+1.8,-1)--(2+1.8,-1.75);
			\draw [fill] (2+1.5,-1.375) circle [radius=.01];
			\draw [fill] (2+1.6,-1.375) circle [radius=.01];
			\draw [fill] (2+1.7,-1.375) circle [radius=.01];
			\draw [dashed, blue, thick] (2-.2,-1.1)--(2+2.1,-1.1);
			\draw [dashed, blue, thick] (2-.2,-.4)--(2+2.1,-.4);
			
			\node [right] at (2+2.1,-.6) {$=$};
			
					\draw (5+0,0)--(5+1.2,0);
					\draw (5+1.2,0)--(5+1.2,.6);
					\draw (5+1.2,.6)--(5+0,.6);
					\draw (5+0,.6)--(5+0,0);
					\draw (5+.6,.6)--(5+.6,.8);
					\node at (5+.6,.3) {$S_{\alpha(k)}$};
					\draw (5+.1,0)--(5+.1,-.4);
					\draw (5+.3,0)--(5+.3,-.4);
					\draw (5+.5,0)--(5+.5,-1.2+.6);
					\draw (5+.9,0)--(5+.9,-1.75);
					\draw (5+.25,-1.2+.6)--(5+.75,-1.2+.6);
					\draw (5+.25,-1.2+.6)--(5+.25,-1.45+.6);
					\draw (5+.75,-1.2+.6)--(5+.75,-1.45+.6);
					\draw (5+.75,-1.45+.6)--(5+.25,-1.45+.6);
					\node at (5+.5,-1.325+.6) {\small $X$};
					\draw (5+.35,-1.45+.6)--(5+.35,-1.75);
					\draw (5+.45,-1.45+.6)--(5+.45,-1.75);
					\draw (5+.65,-1.45+.6)--(5+.65,-1.75);
					\draw [fill] (5+.5,-1.6) circle [radius=.01];
					\draw [fill] (5+.55,-1.6) circle [radius=.01];
					\draw [fill] (5+.6,-1.6) circle [radius=.01];
					\draw (5+.1,-.4) arc [radius=.1, start angle=0, end angle=-90];
					\draw (5,-.5) arc [radius=.1, start angle=90, end angle=180];
					\draw (5+.3,-.4) arc [radius=.1, start angle=0, end angle=-90];
					\draw (5+.2,-.5) arc [radius=.1, start angle=90, end angle=180];
					\draw (5-.1,-.6)--(5-.1,-1.75);
					\draw (5+.1,-.6)--(5+.1,-1.75);
					\draw [fill] (5+-.05,-1.6) circle [radius=.01];
					\draw [fill] (5+0,-1.6) circle [radius=.01];
					\draw [fill] (5+.05,-1.6) circle [radius=.01];
					\node [above] at (5+0,-1.6) {\small $k$};
					
					\draw (5+1.1,0)--(5+1.1,-.4-.6);
					\draw (5+1.1,-.4-.6)--(5+1.5,-.6-.6);
					\draw (5+1.1,-.6-.6)--(5+1.9,-.6-.6);
					\draw (5+1.9,-.6-.6)--(5+1.9,-1-.6);
					\draw (5+1.9,-1-.6)--(5+1.1,-1-.6);
					\draw (5+1.1,-1-.6)--(5+1.1,-.6-.6);
					\node at (5+1.5,-.8-.6) {\small $R$};
					\draw (5+1.2,-1-.6)--(5+1.2,-1.75);
					\draw (5+1.4,-1-.6)--(5+1.4,-1.75);
					\draw (5+1.8,-1-.6)--(5+1.8,-1.75);
					\draw [fill] (5+1.5,-1.675) circle [radius=.01];
					\draw [fill] (5+1.6,-1.675) circle [radius=.01];
					\draw [fill] (5+1.7,-1.675) circle [radius=.01];
					\draw [dashed, blue, thick] (5-.2,-1.1)--(5+2.1,-1.1);
					\draw [dashed, blue, thick] (5-.2,-.4)--(5+2.1,-.4);
			\end{tikzpicture}
			\end{equation*}
			\begin{equation*}
			\begin{tikzpicture}
			\node [left] at (-.2,-.6) {$S_n=$};
				\draw (0+0,0)--(0+1.2,0);
				\draw (0+1.2,0)--(0+1.2,.6);
				\draw (0+1.2,.6)--(0+0,.6);
				\draw (0+0,.6)--(0+0,0);
				\draw (0+.6,.6)--(0+.6,.8);
				\node at (0+.6,.3) {$S_{\alpha(k)}$};
				\draw (0+.1,0)--(0+.1,-.4);
				\draw (0+.3,0)--(0+.3,-.4);
				\draw (0+.5,0)--(0+.5,-1.75);
				\draw (0+.9,0)--(0+.9,-.4);
				\draw (.6+.25,-1.2+.6)--(.6+.75,-1.2+.6);
				\draw (.6+.25,-1.2+.6)--(.6+.25,-1.45+.6);
				\draw (.6+.75,-1.2+.6)--(.6+.75,-1.45+.6);
				\draw (.6+.75,-1.45+.6)--(.6+.25,-1.45+.6);
				\draw (0+.7,-.6)--(0+.7,-1.75);
				\node at (.6+.5,-1.325+.6) {\small $X$};
					\draw (.6+.35,-1.45+.6)--(.6+.35,-1.75);
					\draw (.6+.45,-1.45+.6)--(.6+.45,-1.75);
					\draw (.6+.65,-1.45+.6)--(.6+.65,-1.05);
					\draw [fill] (.6+.5,-.95) circle [radius=.01];
					\draw [fill] (.6+.55,-.95) circle [radius=.01];
					\draw [fill] (.6+.6,-.95) circle [radius=.01];

				\draw (0+.9,-.4) arc [radius=.1, start angle=0, end angle=-90];
				\draw (0+.8,-.5) arc [radius=.1, start angle=90, end angle=180];
				\draw (0+.1,-.4) arc [radius=.1, start angle=0, end angle=-90];
				\draw (0,-.5) arc [radius=.1, start angle=90, end angle=180];
				\draw (0+.3,-.4) arc [radius=.1, start angle=0, end angle=-90];
				\draw (0+.2,-.5) arc [radius=.1, start angle=90, end angle=180];
				\draw (0-.1,-.6)--(0-.1,-1.75);
				\draw (0+.1,-.6)--(0+.1,-1.75);
				\draw [fill] (0+-.05,-1.6) circle [radius=.01];
				\draw [fill] (0+0,-1.6) circle [radius=.01];
				\draw [fill] (0+.05,-1.6) circle [radius=.01];
				\node [above] at (0+0,-1.6) {\small $k$};
				
				\draw (0+1.1,0)--(0+1.1,-.6);
				\draw (.6+.65,-1.05)--(0+1.5,-.6-.6);
				\draw (0+1.1,-.6-.6)--(0+1.9,-.6-.6);
				\draw (0+1.9,-.6-.6)--(0+1.9,-1-.6);
				\draw (0+1.9,-1-.6)--(0+1.1,-1-.6);
				\draw (0+1.1,-1-.6)--(0+1.1,-.6-.6);
				\node at (0+1.5,-.8-.6) {\small $R$};
				\draw (0+1.2,-1-.6)--(0+1.2,-1.75);
				\draw (0+1.4,-1-.6)--(0+1.4,-1.75);
				\draw (0+1.8,-1-.6)--(0+1.8,-1.75);
				\draw [fill] (0+1.5,-1.675) circle [radius=.01];
				\draw [fill] (0+1.6,-1.675) circle [radius=.01];
				\draw [fill] (0+1.7,-1.675) circle [radius=.01];
				\draw [dashed, blue, thick] (0-.2,-1.1)--(0+2.1,-1.1);
				\draw [dashed, blue, thick] (0-.2,-.4)--(0+2.1,-.4);
			\end{tikzpicture}
			\end{equation*}
	
		Therefore, by Definition \ref{def:equivalence},
		\begin{equation*}
		\begin{tikzpicture}
		\node [left] at (-.2,-.6) {$($};
			\draw (0+0,0)--(0+1.2,0);
			\draw (0+1.2,0)--(0+1.2,.6);
			\draw (0+1.2,.6)--(0+0,.6);
			\draw (0+0,.6)--(0+0,0);
			\draw (0+.6,.6)--(0+.6,.8);
			\node at (0+.6,.3) {$S_{\alpha(k)}$};
			\draw (0+.1,0)--(0+.1,-.4);
			\draw (0+.3,0)--(0+.3,-.4);
			\draw (0+.5,0)--(0+.5,-1.2);
			\draw (0+.9,0)--(0+.9,-1.75);
			\draw (0+.25,-1.2)--(0+.75,-1.2);
			\draw (0+.25,-1.2)--(0+.25,-1.45);
			\draw (0+.75,-1.2)--(0+.75,-1.45);
			\draw (0+.75,-1.45)--(0+.25,-1.45);
			\node at (0+.5,-1.325) {\small $X$};
			\draw (0+.35,-1.45)--(0+.35,-1.75);
			\draw (0+.45,-1.45)--(0+.45,-1.75);
			\draw (0+.65,-1.45)--(0+.65,-1.75);
			\draw [fill] (0+.5,-1.6) circle [radius=.01];
			\draw [fill] (0+.55,-1.6) circle [radius=.01];
			\draw [fill] (0+.6,-1.6) circle [radius=.01];
			\draw (0+.1,-.4) arc [radius=.1, start angle=0, end angle=-90];
			\draw (0,-.5) arc [radius=.1, start angle=90, end angle=180];
			\draw (0+.3,-.4) arc [radius=.1, start angle=0, end angle=-90];
			\draw (0+.2,-.5) arc [radius=.1, start angle=90, end angle=180];
			\draw (0-.1,-.6)--(0-.1,-1.75);
			\draw (0+.1,-.6)--(0+.1,-1.75);
			\draw [fill] (0+-.05,-1.6) circle [radius=.01];
			\draw [fill] (0+0,-1.6) circle [radius=.01];
			\draw [fill] (0+.05,-1.6) circle [radius=.01];
			\node [above] at (0+0,-1.6) {\small $k$};
			
			\draw (0+1.1,0)--(0+1.1,-.4);
			\draw (0+1.1,-.4)--(0+1.5,-.6);
			\draw (0+1.1,-.6)--(0+1.9,-.6);
			\draw (0+1.9,-.6)--(0+1.9,-1);
			\draw (0+1.9,-1)--(0+1.1,-1);
			\draw (0+1.1,-1)--(0+1.1,-.6);
			\node at (0+1.5,-.8) {\small $R$};
			\draw (0+1.2,-1)--(0+1.2,-1.75);
			\draw (0+1.4,-1)--(0+1.4,-1.75);
			\draw (0+1.8,-1)--(0+1.8,-1.75);
			\draw [fill] (0+1.5,-1.375) circle [radius=.01];
			\draw [fill] (0+1.6,-1.375) circle [radius=.01];
			\draw [fill] (0+1.7,-1.375) circle [radius=.01];
			\draw [dashed, blue, thick] (0-.2,-1.1)--(0+2.1,-1.1);
			\draw [dashed, blue, thick] (0-.2,-.4)--(0+2.1,-.4);
			
			\node [right] at (0+2.1,-.6) {$,S_n)=($};

				\draw (4+0,0)--(4+1.2,0);
				\draw (4+1.2,0)--(4+1.2,.6);
				\draw (4+1.2,.6)--(4+0,.6);
				\draw (4+0,.6)--(4+0,0);
				\draw (4+.6,.6)--(4+.6,.8);
				\node at (4+.6,.3) {$S_{\alpha(k)}$};
				\draw (4+.1,0)--(4+.1,-.4);
				\draw (4+.3,0)--(4+.3,-.4);
				\draw (4+.5,0)--(4+.5,-1.2+.6);
				\draw (4+.9,0)--(4+.9,-1.75);
				\draw (4+.25,-1.2+.6)--(4+.75,-1.2+.6);
				\draw (4+.25,-1.2+.6)--(4+.25,-1.45+.6);
				\draw (4+.75,-1.2+.6)--(4+.75,-1.45+.6);
				\draw (4+.75,-1.45+.6)--(4+.25,-1.45+.6);
				\node at (4+.5,-1.325+.6) {\small $X$};
				\draw (4+.35,-1.45+.6)--(4+.35,-1.75);
				\draw (4+.45,-1.45+.6)--(4+.45,-1.75);
				\draw (4+.65,-1.45+.6)--(4+.65,-1.75);
				\draw [fill] (4+.5,-1.6) circle [radius=.01];
				\draw [fill] (4+.55,-1.6) circle [radius=.01];
				\draw [fill] (4+.6,-1.6) circle [radius=.01];
				\draw (4+.1,-.4) arc [radius=.1, start angle=0, end angle=-90];
				\draw (4,-.5) arc [radius=.1, start angle=90, end angle=180];
				\draw (4+.3,-.4) arc [radius=.1, start angle=0, end angle=-90];
				\draw (4+.2,-.5) arc [radius=.1, start angle=90, end angle=180];
				\draw (4-.1,-.6)--(4-.1,-1.75);
				\draw (4+.1,-.6)--(4+.1,-1.75);
				\draw [fill] (4+-.05,-1.6) circle [radius=.01];
				\draw [fill] (4+0,-1.6) circle [radius=.01];
				\draw [fill] (4+.05,-1.6) circle [radius=.01];
				\node [above] at (4+0,-1.6) {\small $k$};
				
				\draw (4+1.1,0)--(4+1.1,-.4-.6);
				\draw (4+1.1,-.4-.6)--(4+1.5,-.6-.6);
				\draw (4+1.1,-.6-.6)--(4+1.9,-.6-.6);
				\draw (4+1.9,-.6-.6)--(4+1.9,-1-.6);
				\draw (4+1.9,-1-.6)--(4+1.1,-1-.6);
				\draw (4+1.1,-1-.6)--(4+1.1,-.6-.6);
				\node at (4+1.5,-.8-.6) {\small $R$};
				\draw (4+1.2,-1-.6)--(4+1.2,-1.75);
				\draw (4+1.4,-1-.6)--(4+1.4,-1.75);
				\draw (4+1.8,-1-.6)--(4+1.8,-1.75);
				\draw [fill] (4+1.5,-1.675) circle [radius=.01];
				\draw [fill] (4+1.6,-1.675) circle [radius=.01];
				\draw [fill] (4+1.7,-1.675) circle [radius=.01];
				\draw [dashed, blue, thick] (4-.2,-1.1)--(4+2.1,-1.1);
				\draw [dashed, blue, thick] (4-.2,-.4)--(4+2.1,-.4);
			
			\node [right] at (4+2,-.6) {$,$};
			
			\draw (6.5+0+0,0)--(6.5+0+1.2,0);
			\draw (6.5+0+1.2,0)--(6.5+0+1.2,.6);
			\draw (6.5+0+1.2,.6)--(6.5+0+0,.6);
			\draw (6.5+0+0,.6)--(6.5+0+0,0);
			\draw (6.5+0+.6,.6)--(6.5+0+.6,.8);
			\node at (6.5+0+.6,.3) {$S_{\alpha(k)}$};
			\draw (6.5+0+.1,0)--(6.5+0+.1,-.4);
			\draw (6.5+0+.3,0)--(6.5+0+.3,-.4);
			\draw (6.5+0+.5,0)--(6.5+0+.5,-1.75);
			\draw (6.5+0+.9,0)--(6.5+0+.9,-.4);
			\draw (6.5+.6+.25,-1.2+.6)--(6.5+.6+.75,-1.2+.6);
			\draw (6.5+.6+.25,-1.2+.6)--(6.5+.6+.25,-1.45+.6);
			\draw (6.5+.6+.75,-1.2+.6)--(6.5+.6+.75,-1.45+.6);
			\draw (6.5+.6+.75,-1.45+.6)--(6.5+.6+.25,-1.45+.6);
			\draw (6.5+0+.7,-.6)--(6.5+0+.7,-1.75);
			\node at (6.5+.6+.5,-1.325+.6) {\small $X$};
			\draw (6.5+.6+.35,-1.45+.6)--(6.5+.6+.35,-1.75);
			\draw (6.5+.6+.45,-1.45+.6)--(6.5+.6+.45,-1.75);
			\draw (6.5+.6+.65,-1.45+.6)--(6.5+.6+.65,-1.05);
			\draw [fill] (6.5+.6+.5,-.95) circle [radius=.01];
			\draw [fill] (6.5+.6+.55,-.95) circle [radius=.01];
			\draw [fill] (6.5+.6+.6,-.95) circle [radius=.01];

			\draw (6.5+0+.9,-.4) arc [radius=.1, start angle=0, end angle=-90];
			\draw (6.5+0+.8,-.5) arc [radius=.1, start angle=90, end angle=180];
			\draw (6.5+0+.1,-.4) arc [radius=.1, start angle=0, end angle=-90];
			\draw (6.5+0,-.5) arc [radius=.1, start angle=90, end angle=180];
			\draw (6.5+0+.3,-.4) arc [radius=.1, start angle=0, end angle=-90];
			\draw (6.5+0+.2,-.5) arc [radius=.1, start angle=90, end angle=180];
			\draw (6.5+0-.1,-.6)--(6.5+0-.1,-1.75);
			\draw (6.5+0+.1,-.6)--(6.5+0+.1,-1.75);
			\draw [fill] (6.5+0+-.05,-1.6) circle [radius=.01];
			\draw [fill] (6.5+0+0,-1.6) circle [radius=.01];
			\draw [fill] (6.5+0+.05,-1.6) circle [radius=.01];
			\node [above] at (6.5+0+0,-1.6) {\small $k$};
			
			\draw (6.5+0+1.1,0)--(6.5+0+1.1,-.6);
			\draw (6.5+.6+.65,-1.05)--(6.5+0+1.5,-.6-.6);
			\draw (6.5+0+1.1,-.6-.6)--(6.5+0+1.9,-.6-.6);
			\draw (6.5+0+1.9,-.6-.6)--(6.5+0+1.9,-1-.6);
			\draw (6.5+0+1.9,-1-.6)--(6.5+0+1.1,-1-.6);
			\draw (6.5+0+1.1,-1-.6)--(6.5+0+1.1,-.6-.6);
			\node at (6.5+0+1.5,-.8-.6) {\small $R$};
			\draw (6.5+0+1.2,-1-.6)--(6.5+0+1.2,-1.75);
			\draw (6.5+0+1.4,-1-.6)--(6.5+0+1.4,-1.75);
			\draw (6.5+0+1.8,-1-.6)--(6.5+0+1.8,-1.75);
			\draw [fill] (6.5+0+1.5,-1.675) circle [radius=.01];
			\draw [fill] (6.5+0+1.6,-1.675) circle [radius=.01];
			\draw [fill] (6.5+0+1.7,-1.675) circle [radius=.01];
			\draw [dashed, blue, thick] (6.5+0-.2,-1.1)--(6.5+0+2.1,-1.1);
			\draw [dashed, blue, thick] (6.5+0-.2,-.4)--(6.5+0+2.1,-.4);
			
			\node [right] at (6.5+2,-.6) {$)=x_k$};
		\end{tikzpicture}
		\end{equation*}
	
		Hence $g$ is written as the product of a word in $P_X$ with less length and $x_k$ for some $k\in\mathbb{N}$. Then by a induction on the length, $G_X$ is generated by $\{x_n:n\in\mathbb{N}\}$.
		
		Now we prove Relation \eqref{equ:Relation}. Suppose $k,n\in\mathbb{N}$ with $k<n$.
		\begin{align*}
		x_nx_k&=(S_{a(n)}\cdot X_n,S_{a(n)+1})\circ(S_{a(k)}\cdot X_k,S_{a(k)+1})\\
		&=(S_{a(n)}\cdot X_n,S_{a(n)+1})\circ(S_{a(n)+1}\cdot X_k,S_{a(n)+2})\\
		&=(S_{a(n)}\cdot X_n\cdot X_k,S_{a(n)+1}\cdot X_k)\circ(S_{a(n+N)}\cdot X_k,S_{a(n)+2})\\
		&=((S_{a(n)}\cdot X_n\cdot X_k,S_{a(n)+2})\\\
		\end{align*}
		By Definition \ref{def:xk}, we know that $a(n+N-1)=a(n)+1$. Therefore,
		\begin{align*}
		x_kx_{n+N-1}&=(S_{a(k)}\cdot X_k,S_{a(k)+1})\circ(S_{a(n+N-1)}\cdot X_{n+N-1},S_{a(n+N-1)+1})\\
		&=(S_{a(n)}\cdot X_k,S_{a(n)+1})\circ(S_{a(n)+1}\cdot X_{n+N-1},S_{a(n)+2})\\
		&=(S_{a(n)}\cdot X_k\cdot X_{n+N-1},S_{a(n)+1}\cdot X_{n+N-1})\circ(S_{a(n)+1}\cdot X_n,S_{a(n)+2})\\
		&=(S_{a(n)}\cdot X_k\cdot X_{n+N-1},S_{(a(n)+2)})    
		\end{align*}
		Hence $x_nx_k=x_kx_{n+N-1}$ by Proposition \ref{pro:verticalisotopy}.
	\end{proof}
	
	From the proof of Lemma \ref{lem:Relation}, we have the following proposition,
	\begin{corollary}
		Suppose $g\in P_X$, then there exists $n\in\mathbb{N};i_1,i_2,\cdots,i_n\in\mathbb{N};k_1,k_2,\cdots,k_n\in\mathbb{N}$ such that 
		\begin{equation*}
		g=x_{i_1}^{k_1}x_{i_2}^{k_2}\cdots x_{i_n}^{k_n}.
		\end{equation*}
		Furthermore, 
		\begin{equation*}
		g=1\Leftrightarrow k_1=k_2\cdots=k_n=0.
		\end{equation*}
	\end{corollary}

	\begin{theorem}\label{thm:main}
		The group $G_X$ has a classical presentation 
		\begin{equation}
		G_X\cong\langle t_n,n\in\mathbb{N}\vert t_k^{-1}t_nt_k=t_{n+N-1},~\forall k<N\rangle.
		\end{equation}
		i.e, the group $G_X$ is isomorphic to $F_N$.
	\end{theorem}
	\begin{proof}
		We fist denote $R$, the normal subgroup generated by $\{t^{-1}_kt_n^{-1}t_kt_{n+N-1},~\forall k<n\}$. Define the map $\Phi$ by sending $t_k$ to $x_k$. It follows that $\Phi$ extends to a surjective group homomorphism from $\langle t_n,n\in\mathbb{N}\vert t_k^{-1}t_nt_k=t_{n+N-1},~\forall k<N\rangle$ to $G_X$. Thus we only need to show $\Phi$ is injective. Note that for $k,n\in\mathbb{N},~k<n$,
		\begin{align*}
		t_k^{-1}t_n&=t_{n+N-1}t_k^{-1}\\
		t_n^{-1}t_k&=t_kt_{n+N-1}^{-1}
		\end{align*}
		Therefore, every element $g\in\langle t_n,n\in\mathbb{N}\vert t_k^{-1}t_nt_k=t_{n+N-1},~\forall k<N\rangle$, there exists $g_+,g_-$ such that $g_+g_-^{-1}$ where $g_\pm$ is a word on $\{t_1,t_2,\cdots\}$ with positive powers. These are called positive elements of Thompson group $F_N$. Thus to show $\Phi$ is injective, we only need to show that for every two positive elements $g,h\in F_N$, 
		\begin{equation*}
		\Phi(g)=\Phi(h)\Leftrightarrow g=h
		\end{equation*}
		Since $g,h$ are positive elements and $\Phi$ is a group homomorphism, $\Phi(g),\Phi(h)\in P_X$. Hence there exists $T_g,T_h\in Alg(X)_{n(N-1)+1}$ such that $\Phi(g)=(T_g,S_n)$ and $\Phi(h)=(T_h,S_n)$ where $T_g$ and $T_h$ are isotopically equivalent.
		
		Assume $T_g=T_g^\prime \cdot X_k$, i.e,
		\begin{equation*}
		\begin{tikzpicture}
		\node [left] at (-.2,-.1) {$T_g=$};
		\draw (0,0)--(1.6,0);
		\draw (1.6,0)--(1.6,.8);
		\draw (1.6,.8)--(0,.8);
		\draw (0,.8)--(0,0);
		\node at (.8,.4) {$T_g^\prime$};
		\draw (.8,.8)--(.8,1.2);
		\draw (1.5,0)--(1.5,-1);
		\draw (.1,0)--(.1,-1);
		\draw (.8,0)--(.8,-.3);
		\draw (.5,-.3)--(1.1,-.3);
		\draw (.5,-.3)--(.5,-.7);
		\draw (.5,-.7)--(1.1,-.7);
		\draw (1.1,-.7)--(1.1,-.3);
		\node at (.8,-.5) {$X$};
		\draw (.6,-.7)--(.6,-1);
		\draw (.7,-.7)--(.7,-1);
		\draw (1,-.7)--(1,-1);
		\draw [fill] (1-.075,-.85) circle [radius=.01];
		\draw [fill] (1-.15,-.85) circle [radius=.01];
		\draw [fill] (1-.225,-.85) circle [radius=.01];
		\draw (1.2,0)--(1.2,-1);
		\draw (.4,0)--(.4,-1);
		\draw [fill] (.25,-.5) circle [radius=.01];
		\draw [fill] (.25-.075,-.5) circle [radius=.01];
		\draw [fill] (.25+.075,-.5) circle [radius=.01];
		\node [above] at (.25,-.5) {$k$};
			\draw [fill] (1.35,-.5) circle [radius=.01];
			\draw [fill] (1.35-.075,-.5) circle [radius=.01];
			\draw [fill] (1.35+.075,-.5) circle [radius=.01];
		\end{tikzpicture}
		\end{equation*}
		Since $T_g$ and $T_h$ are isotopically equivalent, $T_h$ must be of the form:
		\begin{equation*}
		\begin{tikzpicture}
		\draw [dashed, blue, thick] (0,0)--(1.6,0);
		\draw (.8,0)--(.8,.4);
		\draw [dashed,blue, thick] (0,-.5)--(1.6,-.5);
		\draw [fill] (.8,-.1) circle [radius=.01];
		\draw [fill] (.8,-.2) circle [radius=.01];
		\draw [fill] (.8,-.3) circle [radius=.01];
		\draw [fill] (.8,-.4) circle [radius=.01];
		\draw (0.5,-.7)--(1.1,-.7);
		\draw (0.5,-1.1)--(1.1,-1.1);
		\draw (.5,-.7)--(.5,-1.1);
		\draw (1.1,-.7)--(1.1,-1.1);
		\node at (.8,-.9) {$X$};
		\draw (.8,-.7)--(.8,-.5);
		\draw [red] (.6,-1.1)--(.6,-2.4);
		\draw (.7,-1.1)--(.7,-2.4);
		\draw (1,-1.1)--(1,-2.4);
		\draw [fill] (.85,-1.25) circle [radius=.01];
		\draw [fill] (.85-.075,-1.25) circle [radius=.01];
		\draw [fill] (.85+.075,-1.25) circle [radius=.01];
		\draw (.4,-.5)--(.4,-1.4);
		\draw (.1,-.5)--(.1,-1.4);
		\draw (1.6-.4,-.5)--(1.6-.4,-1.4);
		\draw (1.6-.1,-.5)--(1.6-.1,-1.4);
		\draw [fill] (.25,-.85) circle [radius=.01];
		\draw [fill] (.25-.075,-.85) circle [radius=.01];
		\draw [fill] (.25+.075,-.85) circle [radius=.01];
		\draw [fill] (1.6-.25,-.85) circle [radius=.01];
		\draw [fill] (1.6-.25-.075,-.85) circle [radius=.01];
		\draw [fill] (1.6-.25+.075,-.85) circle [radius=.01];
		\draw [dashed, blue, thick] (0,-1.4)--(1.6,-1.4);
		\draw [fill] (.25,-1.5) circle [radius=.01];
		\draw [fill] (.25,-1.6) circle [radius=.01];
		\draw [fill] (.25,-1.7) circle [radius=.01];
		\draw [fill] (.25,-1.8) circle [radius=.01];
		\draw [fill] (1.35,-1.5) circle [radius=.01];
		\draw [fill] (1.35,-1.6) circle [radius=.01];
		\draw [fill] (1.35,-1.7) circle [radius=.01];
		\draw [fill] (1.35,-1.8) circle [radius=.01];
		\draw [dashed, blue, thick] (0,-1.9)--(1.6,-1.9);
			\draw [fill] (.85,-2.15) circle [radius=.01];
			\draw [fill] (.85-.075,-2.15) circle [radius=.01];
			\draw [fill] (.85+.075,-2.15) circle [radius=.01];
			
				\draw (.4,-1.9)--(.4,-2.4);
				\draw (.1,-1.9)--(.1,-2.4);
				\draw (1.6-.4,-1.9)--(1.6-.4,-2.4);
				\draw (1.6-.1,-1.9)--(1.6-.1,-2.4);
				\draw [fill] (.25,-2.15) circle [radius=.01];
				\draw [fill] (.25-.075,-2.15) circle [radius=.01];
				\draw [fill] (.25+.075,-2.15) circle [radius=.01];
				\draw [fill] (1.6-.25,-2.15) circle [radius=.01];
				\draw [fill] (1.6-.25-.075,-2.15) circle [radius=.01];
				\draw [fill] (1.6-.25+.075,-2.15) circle [radius=.01];
		\draw [dashed, blue, thick] (0,-2.4)--(1.6,-2.4);
		
		\draw [fill,red] (.4,-2.4) circle [radius=.05];
		\node [below] at (.4,-2.4) {$k$};
		
		\end{tikzpicture}
		\end{equation*}
		where the red string is the $(k+1)$th string from the left.    Therefore, $T_h=T_{h_1}\cdot X_n\cdot T_{h_2}$, where $X_n$ corresponds to the red layer. Following from the proof of Lemma \ref{lem:Relation}, we have $h=h_1 t_n h_2$.
		
		Suppose $h_2=1$, then $n=k$ and $g=g^\prime t_k$, $h=h^\prime t_k$ for some $g^\prime,h^\prime$ positive elements. Note that 
		\begin{equation*}
		g=h\Leftrightarrow g^\prime=h^\prime.
		\end{equation*}
		
		Suppose $h_2\neq 1$, let $t_m$ be the first word in $h_2$, i.e, $h=h_1 t_n t_m h_3$. 
		
		By the structure of $T_h$, we know that either $m<n$ or $m>n+N-1$. If $m<n$, then $t_nt_m=t_mt_{n+N-1}$; if $m>n+N-1$, then $t_nt_m=t_{m-N+1}t_n$ by Relation \eqref{equ:Relation}. By symmetry we just discuss the first case. Let $\widetilde{h}=h_1 t_m t_{n+N-1} h_3$.
		\begin{equation*}
		\begin{tikzpicture}
		\node [left] at (-1.8,-1.45) {$T_h=$};
			\draw [dashed, blue, thick] (-1.6,0)--(1.6,0);
			\draw (.8,0)--(.8,.4);
			\draw [dashed,blue, thick] (-1.6,-.5)--(1.6,-.5);
			\draw [fill] (.8,-.1) circle [radius=.01];
			\draw [fill] (.8,-.2) circle [radius=.01];
			\draw [fill] (.8,-.3) circle [radius=.01];
			\draw [fill] (.8,-.4) circle [radius=.01];
			\draw (0.5,-.7)--(1.1,-.7);
			\draw (0.5,-1.1)--(1.1,-1.1);
			\draw (.5,-.7)--(.5,-1.1);
			\draw (1.1,-.7)--(1.1,-1.1);
			\node at (.8,-.9) {$X$};
			\draw [dashed,blue, thick] (-1.6,-1.4)--(1.6,-1.4);
				\draw (-1.1+0.5,-.7-.9)--(-1.1+1.1,-.7-.9);
				\draw (-1.1+0.5,-1.1-.9)--(-1.1+1.1,-1.1-.9);
				\draw (-1.1+.5,-.7-.9)--(-1.1+.5,-1.1-.9);
				\draw (-1.1+1.1,-.7-.9)--(-1.1+1.1,-1.1-.9);
				\node at (-1.1+.8,-.9-.9) {$X$};
				\draw (-1.1+.8,-.7-.9)--(-1.1+.8,-1.4);
				\draw (-1.1+.6,-1.1-.9)--(-1.1+.6,-2.3);
				\draw (-1.1+.7,-1.1-.9)--(-1.1+.7,-2.3);
				\draw (-1.1+1,-1.1-.9)--(-1.1+1,-2.3);
					\draw [fill] (-1.1+.85,-1.25-.9) circle [radius=.01];
					\draw [fill] (-1.1+.85-.075,-1.25-.9) circle [radius=.01];
					\draw [fill] (-1.1+.85+.075,-1.25-.9) circle [radius=.01];
					\draw [fill] (-1.1+.25,-.85) circle [radius=.01];
					\draw [fill] (-1.1+.25-.075,-.85) circle [radius=.01];
					\draw [fill] (-1.1+.25+.075,-.85) circle [radius=.01];
						\draw (-1.1+.4,-.5)--(-1.1+.4,-1.4-.9);
						\draw (-1.1+.1,-.5)--(-1.1+.1,-1.4-.9);
						
						\draw (-1.1+.1,-1.9-.9)--(-1.1+.1,-2.4-.9);
				
			\draw (.8,-.7)--(.8,-.5);
			\draw [red] (.6,-1.1)--(.6,-2.4-.9);
			\draw (.7,-1.1)--(.7,-2.4-.9);
			\draw (1,-1.1)--(1,-2.4-.9);
			\draw [fill] (.85,-1.25) circle [radius=.01];
			\draw [fill] (.85-.075,-1.25) circle [radius=.01];
			\draw [fill] (.85+.075,-1.25) circle [radius=.01];
			\draw (.4,-.5)--(.4,-1.4-.9);
			\draw (.1,-.5)--(.1,-1.4-.9);
			\draw (1.6-.4,-.5)--(1.6-.4,-1.4);
			\draw (1.6-.1,-.5)--(1.6-.1,-1.4);
			\draw [fill] (.25,-.85) circle [radius=.01];
			\draw [fill] (.25-.075,-.85) circle [radius=.01];
			\draw [fill] (.25+.075,-.85) circle [radius=.01];
			\draw [fill] (1.6-.25,-.85) circle [radius=.01];
			\draw [fill] (1.6-.25-.075,-.85) circle [radius=.01];
			\draw [fill] (1.6-.25+.075,-.85) circle [radius=.01];
			\draw [dashed, blue, thick] (-1.6,-1.4-.9)--(1.6,-1.4-.9);
			\draw [fill] (.25,-1.5-.9) circle [radius=.01];
			\draw [fill] (.25,-1.6-.9) circle [radius=.01];
			\draw [fill] (.25,-1.7-.9) circle [radius=.01];
			\draw [fill] (.25,-1.8-.9) circle [radius=.01];
			\draw [fill] (1.35,-1.5-.9) circle [radius=.01];
			\draw [fill] (1.35,-1.6-.9) circle [radius=.01];
			\draw [fill] (1.35,-1.7-.9) circle [radius=.01];
			\draw [fill] (1.35,-1.8-.9) circle [radius=.01];
			\draw [dashed, blue, thick] (-1.6,-1.9-.9)--(1.6,-1.9-.9);
			\draw [fill] (.85,-2.15-.9) circle [radius=.01];
			\draw [fill] (.85-.075,-2.15-.9) circle [radius=.01];
			\draw [fill] (.85+.075,-2.15-.9) circle [radius=.01];
			
			\draw (.4,-1.9-.9)--(.4,-2.4-.9);
		
			\draw (1.6-.4,-1.9-.9)--(1.6-.4,-2.4-.9);
			\draw (1.6-.1,-1.9-.9)--(1.6-.1,-2.4-.9);
			\draw [fill] (-.6,-2.15-.9) circle [radius=.01];
			\draw [fill] (-.6-.1,-2.15-.9) circle [radius=.01];
			\draw [fill] (-.6+.1,-2.15-.9) circle [radius=.01];
			\draw [fill] (1.6-.25,-2.15-.9) circle [radius=.01];
			\draw [fill] (1.6-.25-.075,-2.15-.9) circle [radius=.01];
			\draw [fill] (1.6-.25+.075,-2.15-.9) circle [radius=.01];
			\draw [dashed, blue, thick] (-1.6,-2.4-.9)--(1.6,-2.4-.9);
			
			\draw [fill,red] (.4,-2.4-.9) circle [radius=.05];
			\node [below] at (.4,-2.4-.9) {$k$};
			
			\node [right] at (1.7,-1.45) {$=$};
			
				\draw [dashed, blue, thick] (4+-1.6,0)--(4+1.6,0);
				\draw (4+.8,0)--(4+.8,.4);
				\draw [dashed,blue, thick] (4+-1.6,-.5)--(4+1.6,-.5);
				\draw [fill] (4+.8,-.1) circle [radius=.01];
				\draw [fill] (4+.8,-.2) circle [radius=.01];
				\draw [fill] (4+.8,-.3) circle [radius=.01];
				\draw [fill] (4+.8,-.4) circle [radius=.01];
				\draw (4+0.5,-.7-.9)--(4+1.1,-.7-.9);
				\draw (4+0.5,-1.1-.9)--(4+1.1,-1.1-.9);
				\draw (4+.5,-.7-.9)--(4+.5,-1.1-.9);
				\draw (4+1.1,-.7-.9)--(4+1.1,-1.1-.9);
				\node at (4+.8,-.9-.9) {$X$};
				\draw [dashed,blue, thick] (4+-1.6,-1.4)--(4+1.6,-1.4);
				\draw (4+-1.1+0.5,-.7)--(4+-1.1+1.1,-.7);
				\draw (4+-1.1+0.5,-1.1)--(4+-1.1+1.1,-1.1);
				\draw (4+-1.1+.5,-.7)--(4+-1.1+.5,-1.1);
				\draw (4+-1.1+1.1,-.7)--(4+-1.1+1.1,-1.1);
				\node at (4+-1.1+.8,-.9) {$X$};
				\draw (4+-1.1+.8,-.7)--(4+-1.1+.8,-.5);
				\draw (4+-1.1+.6,-1.1)--(4+-1.1+.6,-1.4);
				\draw (4+-1.1+.7,-1.1)--(4+-1.1+.7,-1.4);
				\draw (4+-1.1+1,-1.1)--(4+-1.1+1,-1.4);
				\draw [fill] (4+-1.1+.85,-1.25) circle [radius=.01];
				\draw [fill] (4+-1.1+.85-.075,-1.25) circle [radius=.01];
				\draw [fill] (4+-1.1+.85+.075,-1.25) circle [radius=.01];
				\draw [fill] (4+-1.1+.25,-.85) circle [radius=.01];
				\draw [fill] (4+-1.1+.25-.075,-.85) circle [radius=.01];
				\draw [fill] (4+-1.1+.25+.075,-.85) circle [radius=.01];
				\draw (4+-1.1+.4,-.5)--(4+-1.1+.4,-1.4-.9);
				\draw (4+-1.1+.1,-.5)--(4+-1.1+.1,-1.4-.9);
				
				\draw (4+-1.1+.1,-1.9-.9)--(4+-1.1+.1,-2.4-.9);
				
				\draw (4+.8,-.7-.9)--(4+.8,-.5-.9);
				\draw [red] (4+.6,-1.1-.9)--(4+.6,-2.4-.9);
				\draw (4+.7,-1.1-.9)--(4+.7,-2.4-.9);
				\draw (4+1,-1.1-.9)--(4+1,-2.4-.9);
				\draw [fill] (4+.85,-1.25-.9) circle [radius=.01];
				\draw [fill] (4+.85-.075,-1.25-.9) circle [radius=.01];
				\draw [fill] (4+.85+.075,-1.25-.9) circle [radius=.01];
				\draw (4+.4,-.5)--(4+.4,-1.4-.9);
				\draw (4+.1,-.5)--(4+.1,-1.4-.9);
				\draw (4+1.6-.4,-.5)--(4+1.6-.4,-1.4);
				\draw (4+1.6-.1,-.5)--(4+1.6-.1,-1.4);
				\draw [fill] (4+.25,-.85) circle [radius=.01];
				\draw [fill] (4+.25-.075,-.85) circle [radius=.01];
				\draw [fill] (4+.25+.075,-.85) circle [radius=.01];
				\draw [fill] (4+1.6-.25,-.85) circle [radius=.01];
				\draw [fill] (4+1.6-.25-.075,-.85) circle [radius=.01];
				\draw [fill] (4+1.6-.25+.075,-.85) circle [radius=.01];
				\draw [dashed, blue, thick] (4+-1.6,-1.4-.9)--(4+1.6,-1.4-.9);
				\draw [fill] (4+.25,-1.5-.9) circle [radius=.01];
				\draw [fill] (4+.25,-1.6-.9) circle [radius=.01];
				\draw [fill] (4+.25,-1.7-.9) circle [radius=.01];
				\draw [fill] (4+.25,-1.8-.9) circle [radius=.01];
				\draw [fill] (4+1.35,-1.5-.9) circle [radius=.01];
				\draw [fill] (4+1.35,-1.6-.9) circle [radius=.01];
				\draw [fill] (4+1.35,-1.7-.9) circle [radius=.01];
				\draw [fill] (4+1.35,-1.8-.9) circle [radius=.01];
				\draw [dashed, blue, thick] (4+-1.6,-1.9-.9)--(4+1.6,-1.9-.9);
				\draw [fill] (4+.85,-2.15-.9) circle [radius=.01];
				\draw [fill] (4+.85-.075,-2.15-.9) circle [radius=.01];
				\draw [fill] (4+.85+.075,-2.15-.9) circle [radius=.01];
				
				\draw (4+.4,-1.9-.9)--(4+.4,-2.4-.9);
				
				\draw (4+1.6-.4,-1.9-.9)--(4+1.6-.4,-2.4-.9);
				\draw (4+1.6-.1,-1.9-.9)--(4+1.6-.1,-2.4-.9);
				\draw [fill] (4+-.6,-2.15-.9) circle [radius=.01];
				\draw [fill] (4+-.6-.1,-2.15-.9) circle [radius=.01];
				\draw [fill] (4+-.6+.1,-2.15-.9) circle [radius=.01];
				\draw [fill] (4+1.6-.25,-2.15-.9) circle [radius=.01];
				\draw [fill] (4+1.6-.25-.075,-2.15-.9) circle [radius=.01];
				\draw [fill] (4+1.6-.25+.075,-2.15-.9) circle [radius=.01];
				\draw [dashed, blue, thick] (4+-1.6,-2.4-.9)--(4+1.6,-2.4-.9);
				
				\draw [fill,red] (4+.4,-2.4-.9) circle [radius=.05];
				\node [below] at (4+.4,-2.4-.9) {$k$};
				\node [right] at (5.8,-1.45) {$=T_{\widetilde{h}}$};
		\end{tikzpicture}
		\end{equation*}

		Note that $h^{-1}\widetilde{h}=h_3^{-1}t_m^{-1}t_n^{-1}t_mt_{n+N-1}h_3\in R$. Therefore to show $g^{-1}h\in R$, we only need to show $g^{-1}\widetilde{h}\in R$, i.e, to show that $\Phi(g)=\Phi(\widetilde{h})\Leftrightarrow g=\widetilde{h}$. By repeating this procedure, we obtain a positive element $\widehat{h}$ with $\widehat{h}=\widehat{h}^\prime t_k$. We only need to show 
		\begin{align*}
		\Phi(g)=\Phi(\widehat{h})&\Leftrightarrow g=\widehat{h}\\
		&\Leftrightarrow g^\prime=\widehat{h}^\prime
		\end{align*}
		
		In both cases, we reduce to the case that comparing positive elements of fewer lengths. Eventually it reduces to the case that for $i_1,\cdot, i_m; n_1,\cdot, n_m\in\mathbb{N}$
		\begin{align*}
		t_{i_1}^{n_1}t_{i_2}^{n_2}\cdots t_{i_m}^{n_m}&=1\\
		\Rightarrow x_{i_1}^{n_1}x_{i_2}^{n_2}\cdots x_{i_m}^{n_m}&=1
		\end{align*}
		Therefore $n_1=n_2=\cdots n_m=1$, i.e, $\Phi$ is injective. 
	\end{proof}
	
	\section{Examples}\label{sec:example}
	
	In this section, we will mainly introduce two examples defined by Jones. 
	\subsection{The Jones subgroup $\vec{F}$}\label{sec:vecF}
	Jones introduced the Jones subgroup $\vec{F}$ \cite{JonTP}, and it is showed that $\vec{F}\cong F_3$ by Golan and Sapir \cite{GoSa2015}. We first recall the definition of $\vec{F}$. 
	
	Let $\mathscr{P}_\bullet$ be the subfactor planar algebra, $\mathcal{TL}(2)$. We construct a unitary representation $\pi$ with Approach 1 in \S \ref{sec:pre} by taking the 2-box $R$ to be $\displaystyle2^{1/4}(\gra{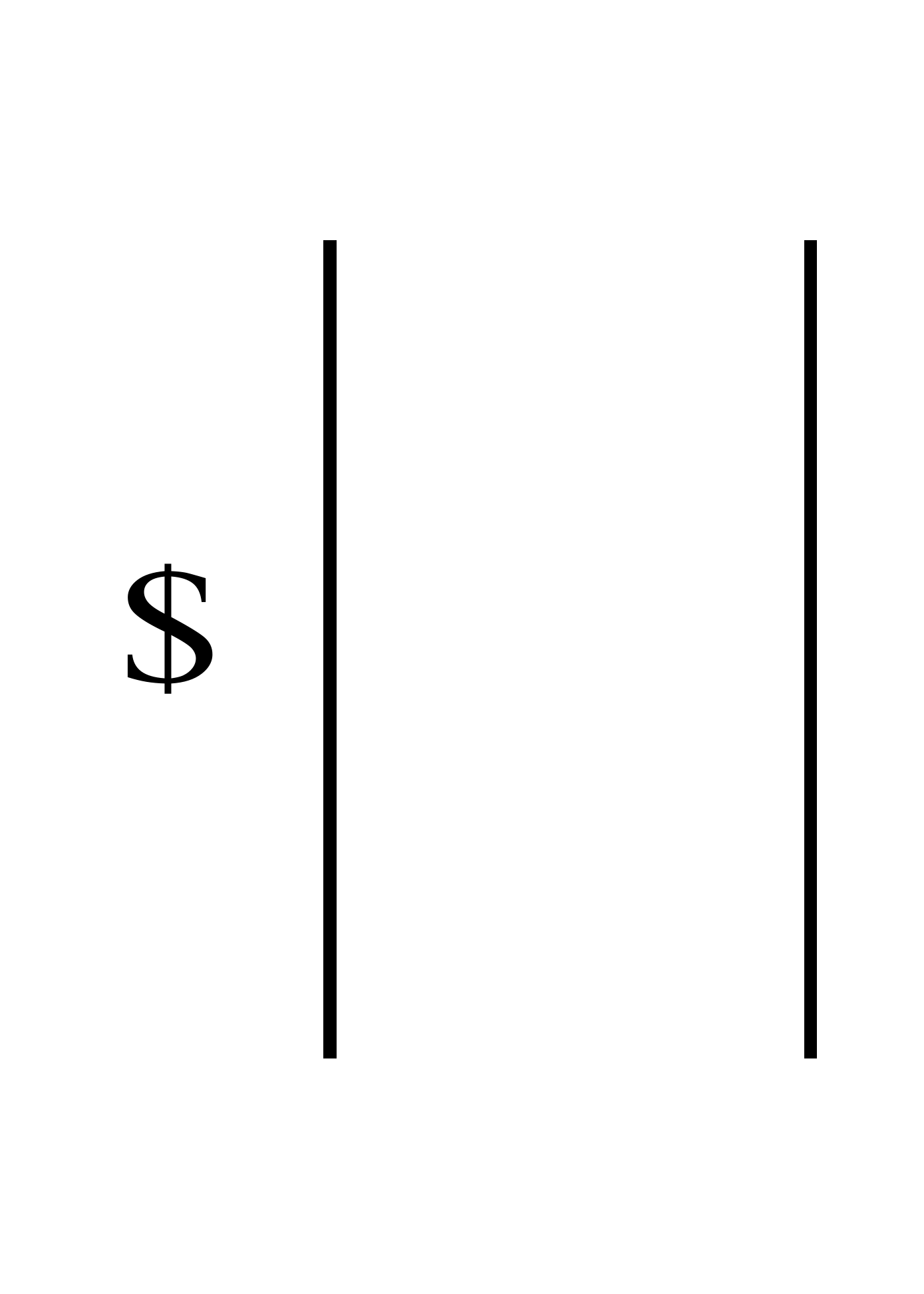}-\displaystyle\frac{1}{\sqrt{2}}\gra{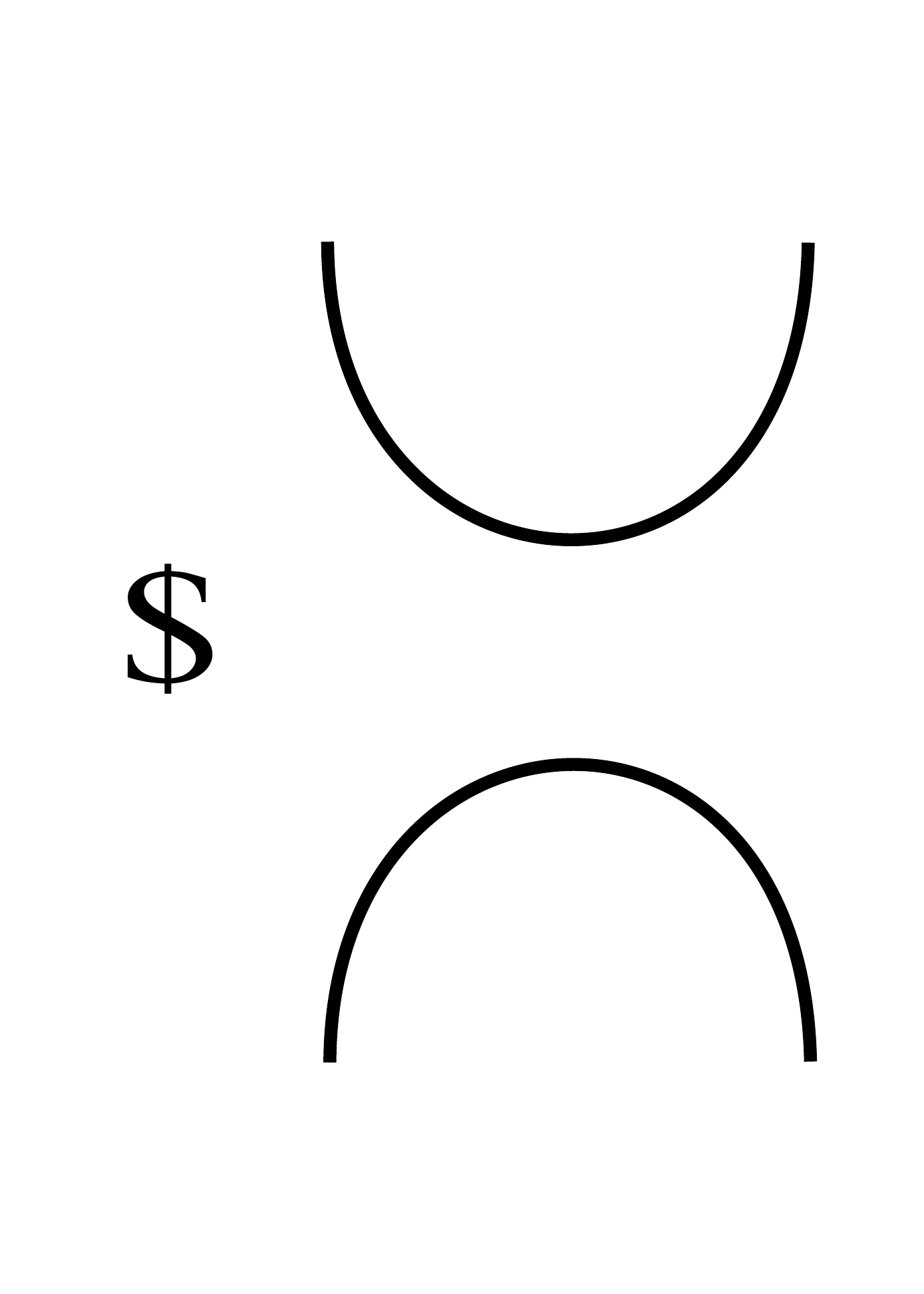})$, a multiple of the $2$nd Jones-Wenzl idempotent. 
	
	For every element $g\in F$, the matrix coefficient 
	$\langle \pi(g)\xi,\xi\rangle$ has a specific chromatic meaning:
	
	Suppose $g$ has a pair of binary tree representation as $(T_+,T_-)$. We arrange the pair of trees in $\mathbb{R}^2$ such that the leaves of $T_\pm$ are the points $(1/2,0), (3/2, 0), \cdots ((2n-1)/2,0)$ with all the edges being the straight line segments sloping either up from left to right or down from left to right. $T_+$ is in the upper half plane and $T_-$ is in the lower half plane.  We construct a simply-laced planar graph $\Gamma(T_+,T_-)$ from the given trees. Let the vertices be $\{(0, 0),(1, 0), ...,(n, 0)\}$ contained in each region between the edges of each tree. The vertices $(k,0)$ and $(j,0)$ is connected by an edge if and only if the corresponding regions are separated by an edge sloping up in $T_+$ or down in $T_-$ from left to right. 
	
	\begin{proposition}[Jones, \cite{JonTP}]\label{pro:gamma}
		$\Gamma(T_+,T_-)$ defined above is consisting of a pair of trees, $\Gamma(T_+)$ in the upper half plane and $\Gamma(T_-)$ in the lower half plane having following properties:
		\begin{itemize}
			\item The vertices are $0,1,2,\cdots, n$. 
			\item Each vertex other than $0$ is connected to exactly one vertex to its left
			\item Each edge can be parameterised as $(x(t),y(t))$ for $0\leq t\leq1$ such that $x^\prime(t)>0$ and $y(t)>0$ on $(0,1)$ or $y(t)<0$ on $(0,1)$.
		\end{itemize}
	\end{proposition}
	
	\begin{proposition}[Jones, \cite{JonTP}] \label{pro:matrixcoef}
		We have the following for the matrix coefficient,
		\begin{equation*}
		\langle\pi(g)\xi,\xi\rangle=\frac{1}{2}Chr_{\Gamma(T_+,T_-)}(2),
		\end{equation*}
		where $Chr_{\Gamma(T_+,T_-)}(2)$ is the value of the chromatic polynomial for $\Gamma(T_+,T_-)$ at $2$, i.e, the number of $2$-coloring of $\Gamma(T_+,T_-)$.
	\end{proposition}
	\begin{remark}
		By definition of the chromatic polynomial, we have
		\begin{align*}
		Chr_{\Gamma(T_+,T_-)}(2)=
		\begin{cases}
		2, & \text{if~} \Gamma(T_+,T_-) \text{~is bipartite}\\
		0, & \text{if~} \Gamma(T_+,T_-) \text{~is not bipartite}
		\end{cases}
		\end{align*}
	\end{remark}
	\begin{definition}[the Jones subgroup]
		We denote the Jones subgroup $\vec{F}$ as the stabilizer of the vacuum vector, i.e,
		\begin{equation*}
		\vec{F}=\{g\in F\vert \pi(g)\xi=\xi\}
		\end{equation*} 
	\end{definition}
	\begin{remark}
		Suppose $g$ has a pair of trees representation as $(T_+,T_-)$. By the remark of Proposition \ref{pro:matrixcoef}, we have that $g\in\vec{F}$ if and only if $\Gamma(T_+,T_-)$ is a bipartite planar graph.
	\end{remark}
	
	Now we show that $\vec{F}$ is isomorphic to $F_3$. 
	\begin{lemma}\label{lem:alternating}
		Suppose $g\in\vec{F}$, then $g$ has a pair of trees representation as $(T_+,T_-)$ such that the coloring of the vertices of $\Gamma(T_+,T_-)$ is $\pm\mp\pm\mp\cdots\pm$ from the left to right. 
	\end{lemma}
	\begin{proof}
		Suppose $(T_+,T_-)$ is a pair of trees representation of $g\in\vec{F}$. Since $g\in\vec{F}$, $\Gamma(T_+,T_-)$ is a bipartite graph, i.e, there exists a coloring of the vertices. Suppose there exists a vertex $(i,0)$ with $i\in\mathbb{N}$ and the coloring of $(i,0)$ and $(i+1,0)$ are different, i.e, 
		\begin{equation*}
		\Gamma(T_+,T_-):\grb{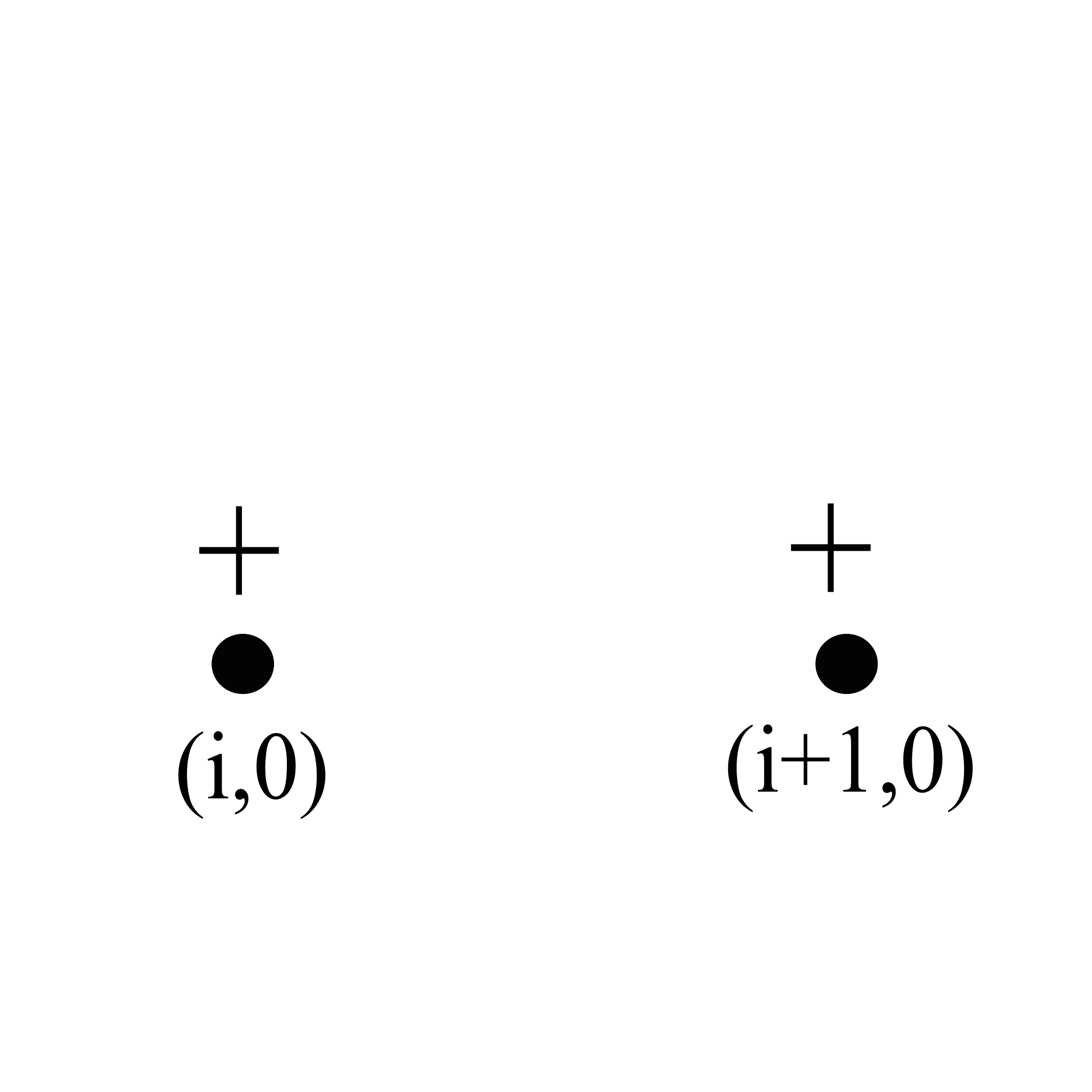}\longleftrightarrow(T_+,T_-):\grb{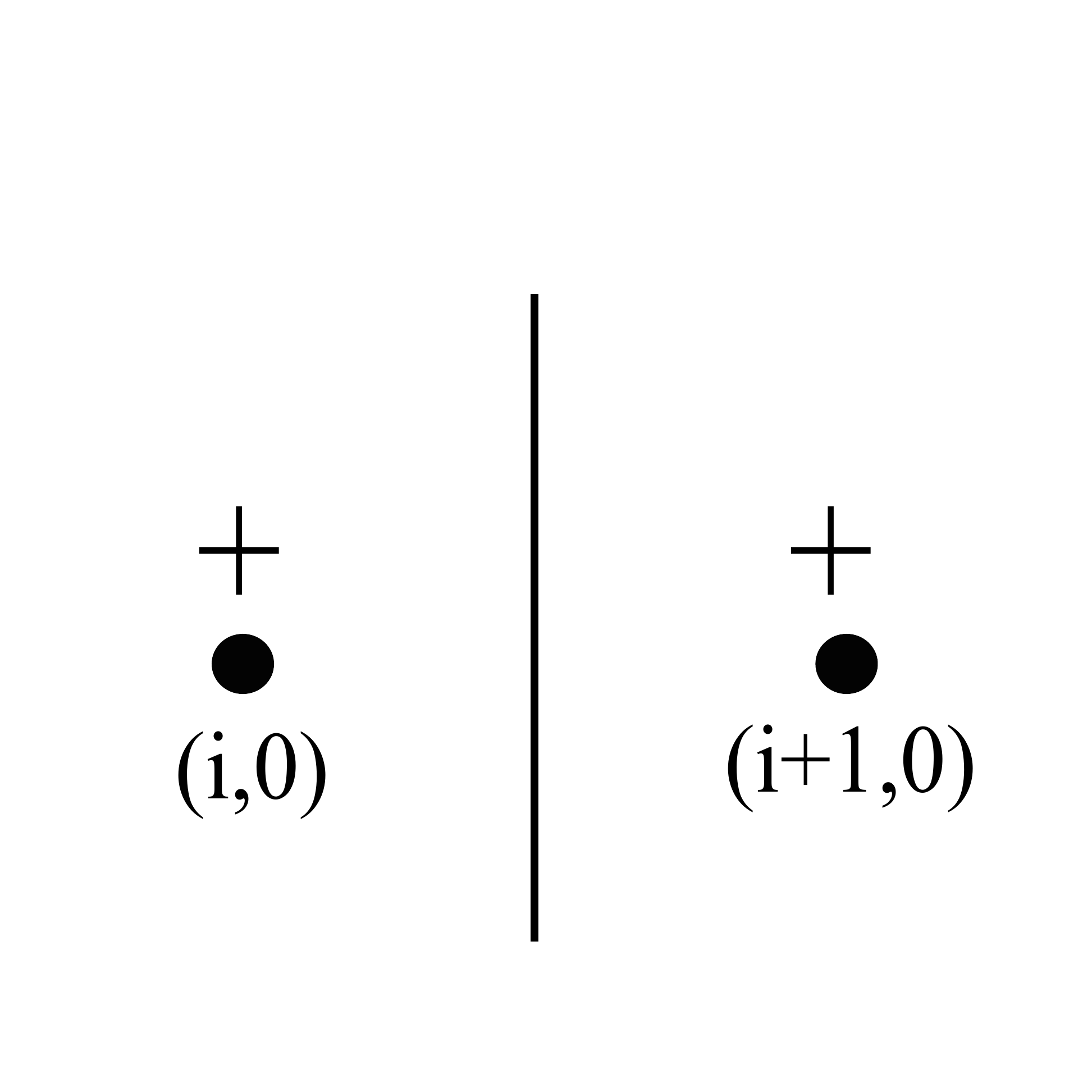}
		\end{equation*}
		We consider the pair of trees with by putting a caret on both $T_\pm$ on the edge with endpoints $(i+1/2,0)$. Therefore, we have
		\begin{equation*}
		\Gamma(T_+,T_-):\grb{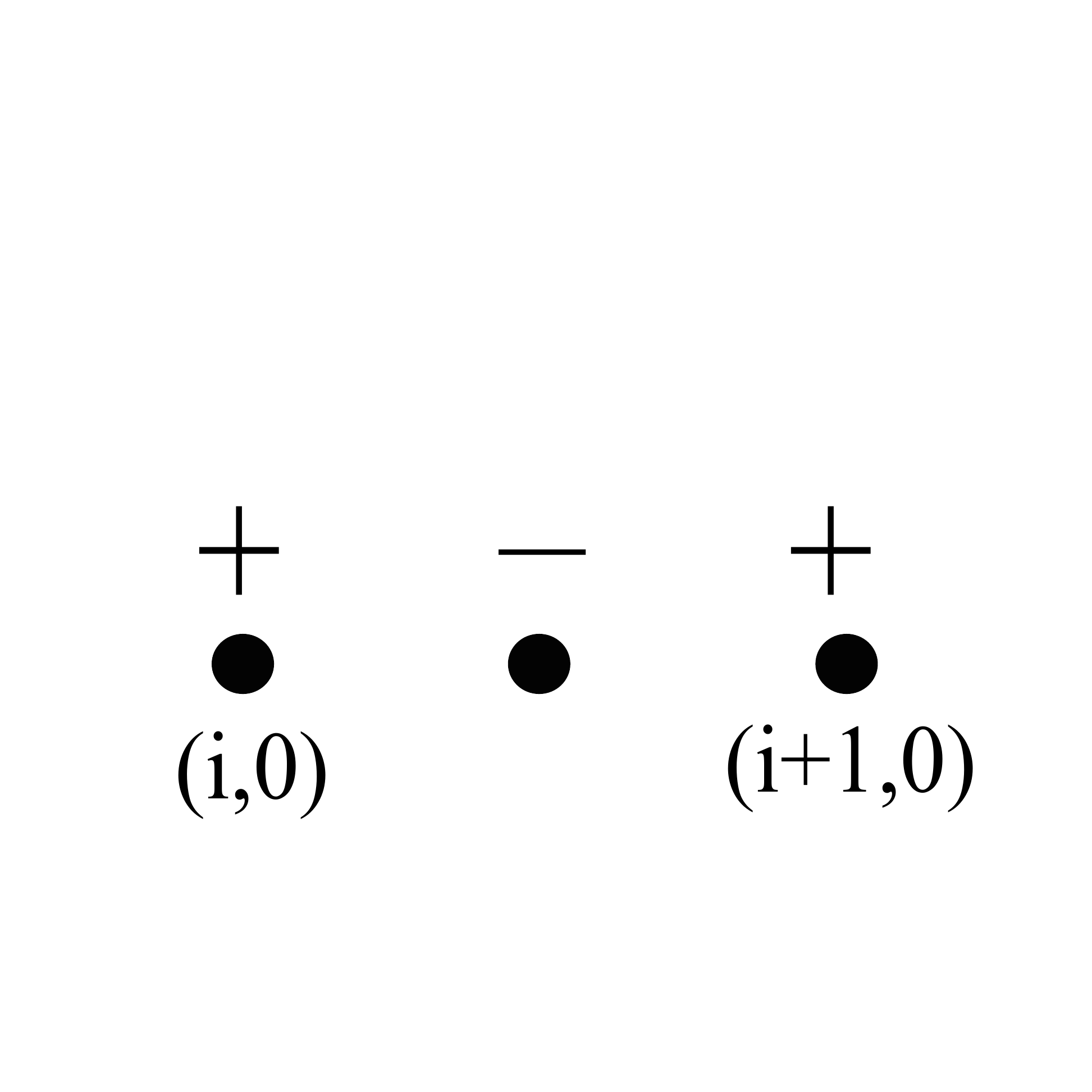}\longleftrightarrow (T_+,T_-):\grb{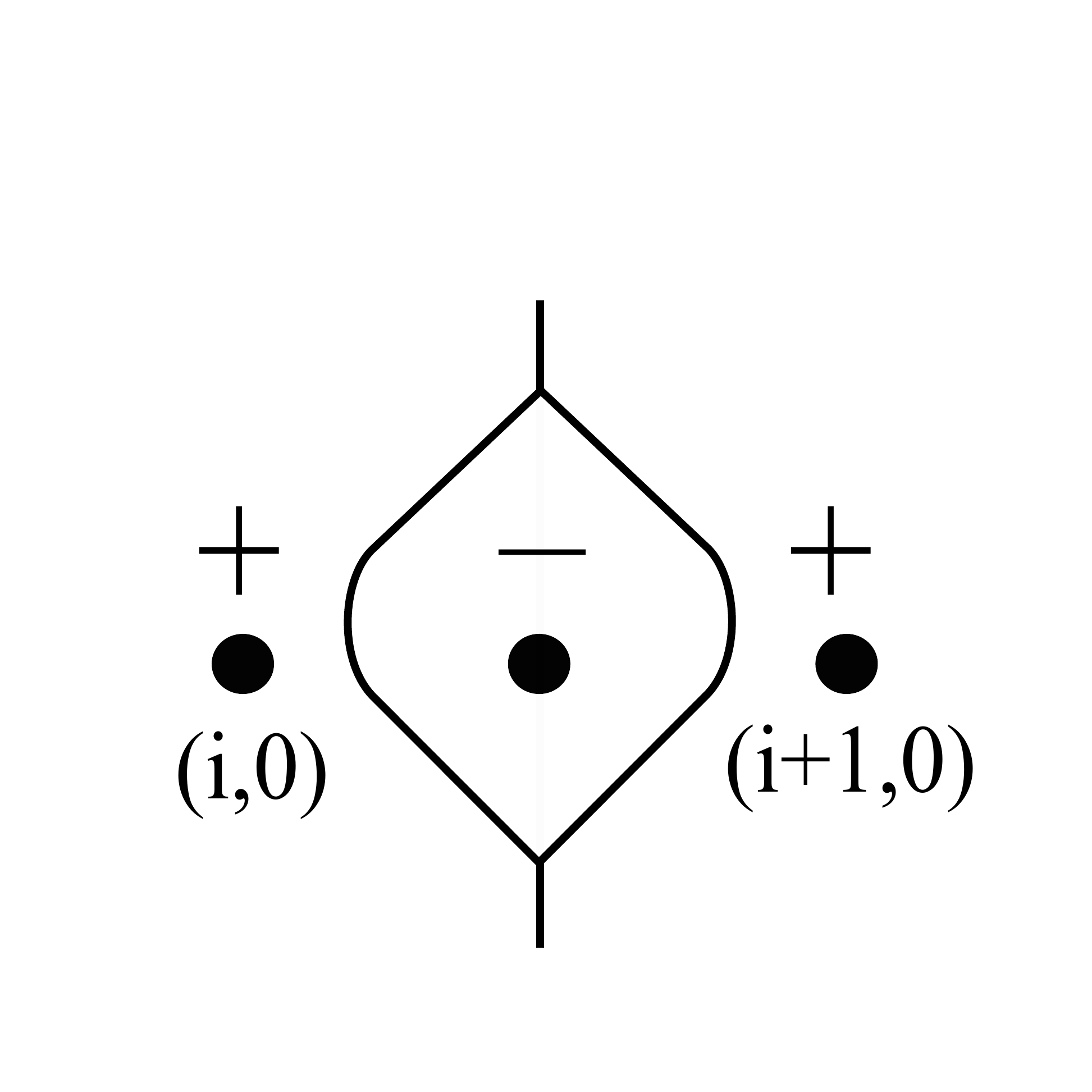}
		\end{equation*}
		Repeating this procedure, we reach a pair of trees representation of $g$ required by the lemma.
	\end{proof}
	\begin{theorem}\label{thm:vecf}
		The Jones subgroup $\vec{F}$ is isomorphic to $F_3$. 
	\end{theorem}
	\begin{proof}
		Let $X=\gra{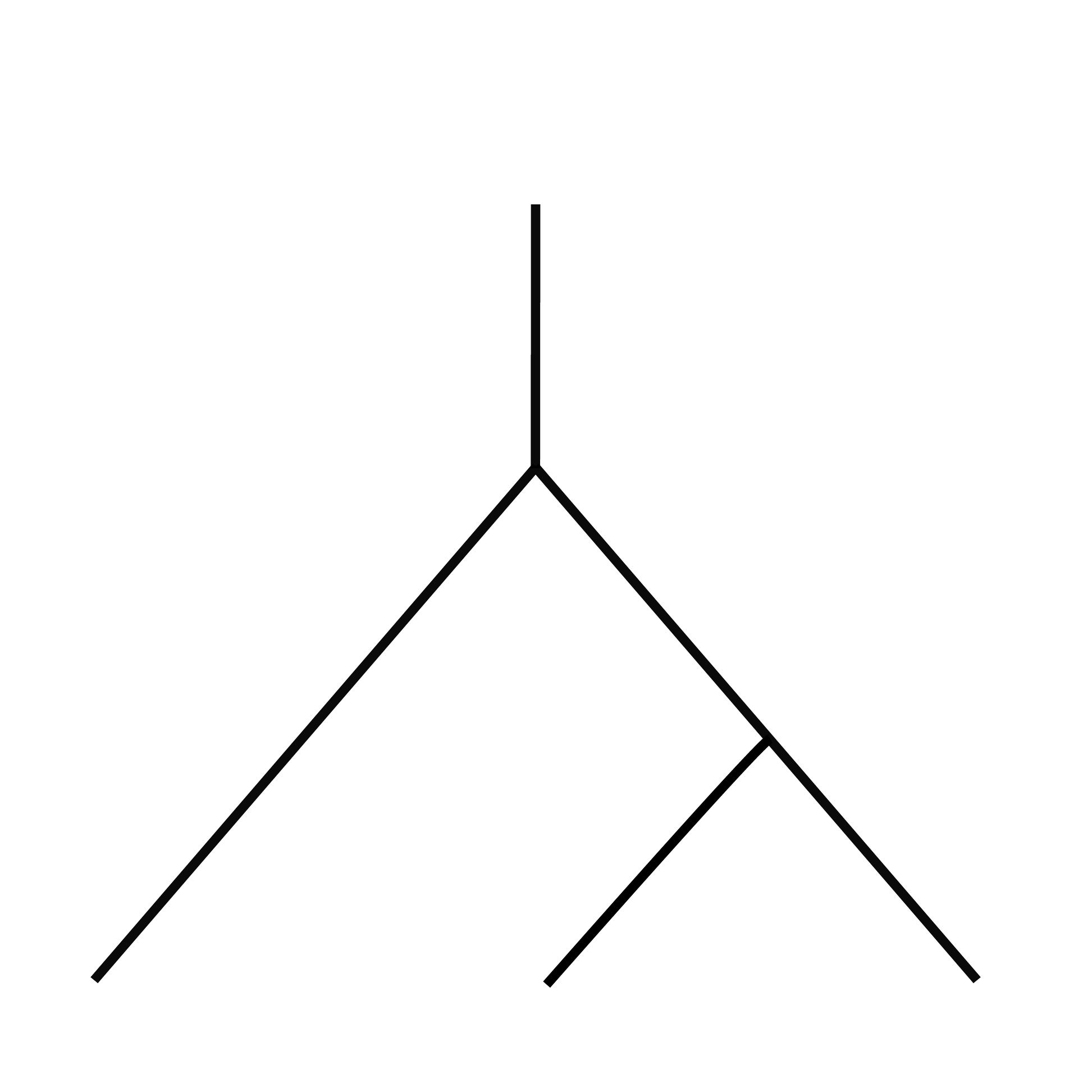}$. Suppose non-trivial $g\in\vec{F}$ with a pair of trees representation $(T_+,T_-)$ as in Lemma \ref{lem:alternating}. Assume $T_+$ has $2n-1$ leaves.
		\begin{claim}
			There exists $i\in\mathbb{N}$ such that in $T_+$
			$$\grb{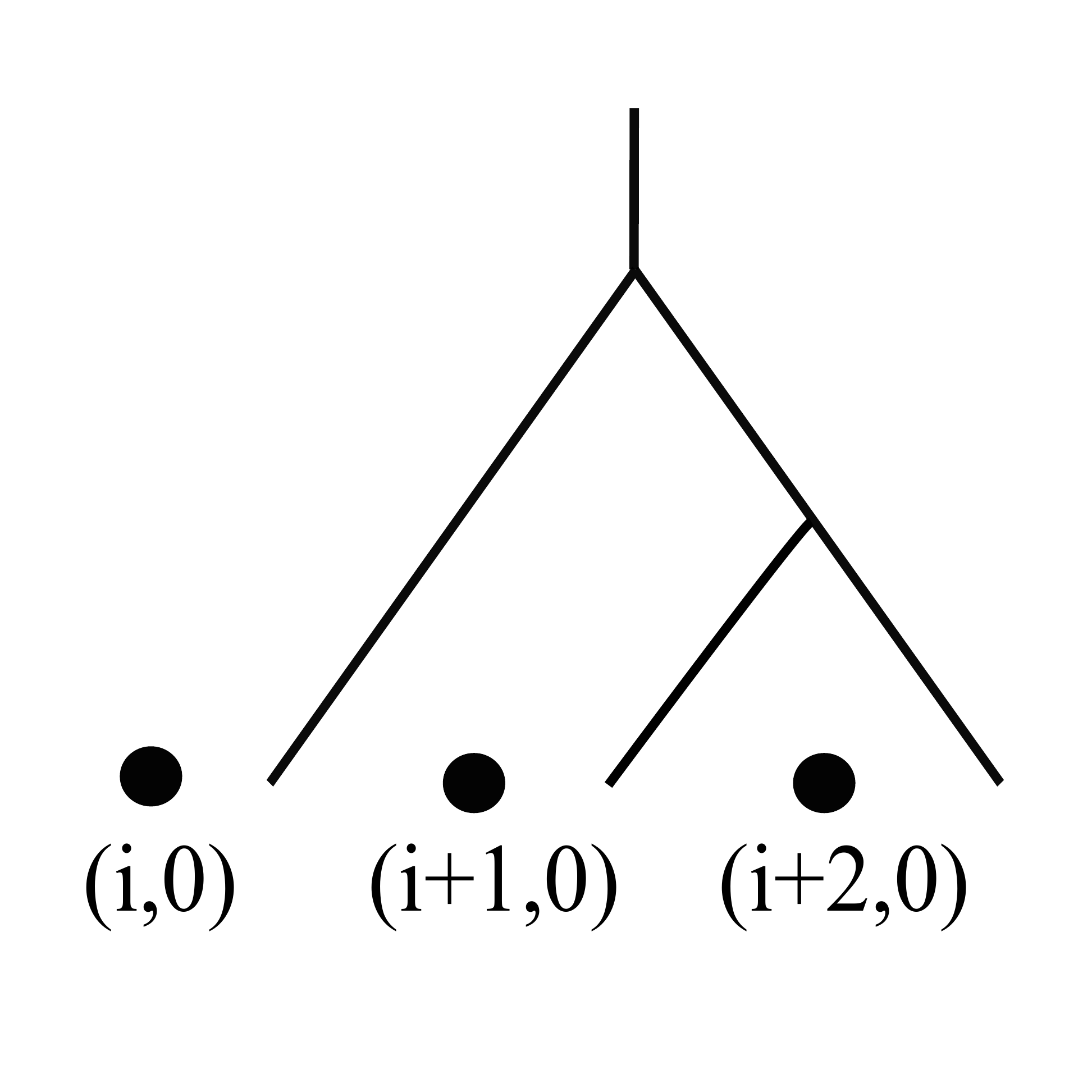}$$
		\end{claim}
		\begin{proof}[\textbf{Proof of the claim}]
			By Proposition \ref{pro:gamma}, we consider the upper half part of $\Gamma(T_+,T_-)$, denoted by $\Gamma(T_+)$. The claim is equivalent to that there exists $i\in\mathbb{N}$ such that $(i+2,0)$ is connected to $(i+1,0)$ and $(i+1,0)$ is connected to $(i,0)$ by edges in $\Gamma(T_+)$. We prove this statement by induction on the number of vertices. The basic step is trivially true since the coloring is $\pm\mp\pm$. 
			
			Now suppose the statement holds for $2n-1$ vertices. Consider $\Gamma(T_+)$ has $2n+1$ vertices. Since $\Gamma(T_+)$, there exist $m\in\mathbb{N}$ such that $(m,0)$ is connected to $(m-1,0)$ by induction. Let $(k,0)$ be the only vertex that $(m-1,0)$ is connected to on its left. If $k=m-2$, then we find such vertex required by the claim. If $k<m-2$, we apply the induction on the induced sub graph of $\Gamma(T_+)$ on vertices $\{(k,0),(k+1,0),\cdots, (m.0)\}$. Therefore we concludes the claim. 
		\end{proof}
		With proving the claim, we obtain that $T_+\in Alg(X)_{2n-1}$. 
		Therefore $\vec{F}\cong F_3$ by Theorem \ref{thm:main}.
	\end{proof}
	
	\begin{remark}
		Jones defined a family of subgroups of $F$ which start with $\vec{F}$ using the $R$-matrix in spin models \cite{JonTP}. These subgroups have been redefined by Golan and Sapir denoted by $\vec{F}_n, n\geq2$ and showed to be isomorphic to $F_{n+1}$ \cite{GoSa2015}. These results can be obtained by using the same idea in the proof of Theorem \ref{thm:vecf}.
	\end{remark}
	
	\subsection{The 3-colorable subgroup}\label{sec:3color}
	Jones introduce the 3-colorable subgroup $\mathcal{F}$ and we now recall the definition.
	
	Let $\mathscr{P}_\bullet$ is the subfactor planar algebra $\mathcal{TL}(2)$. we set $S$ to be $3^{\frac{1}{4}}\graa{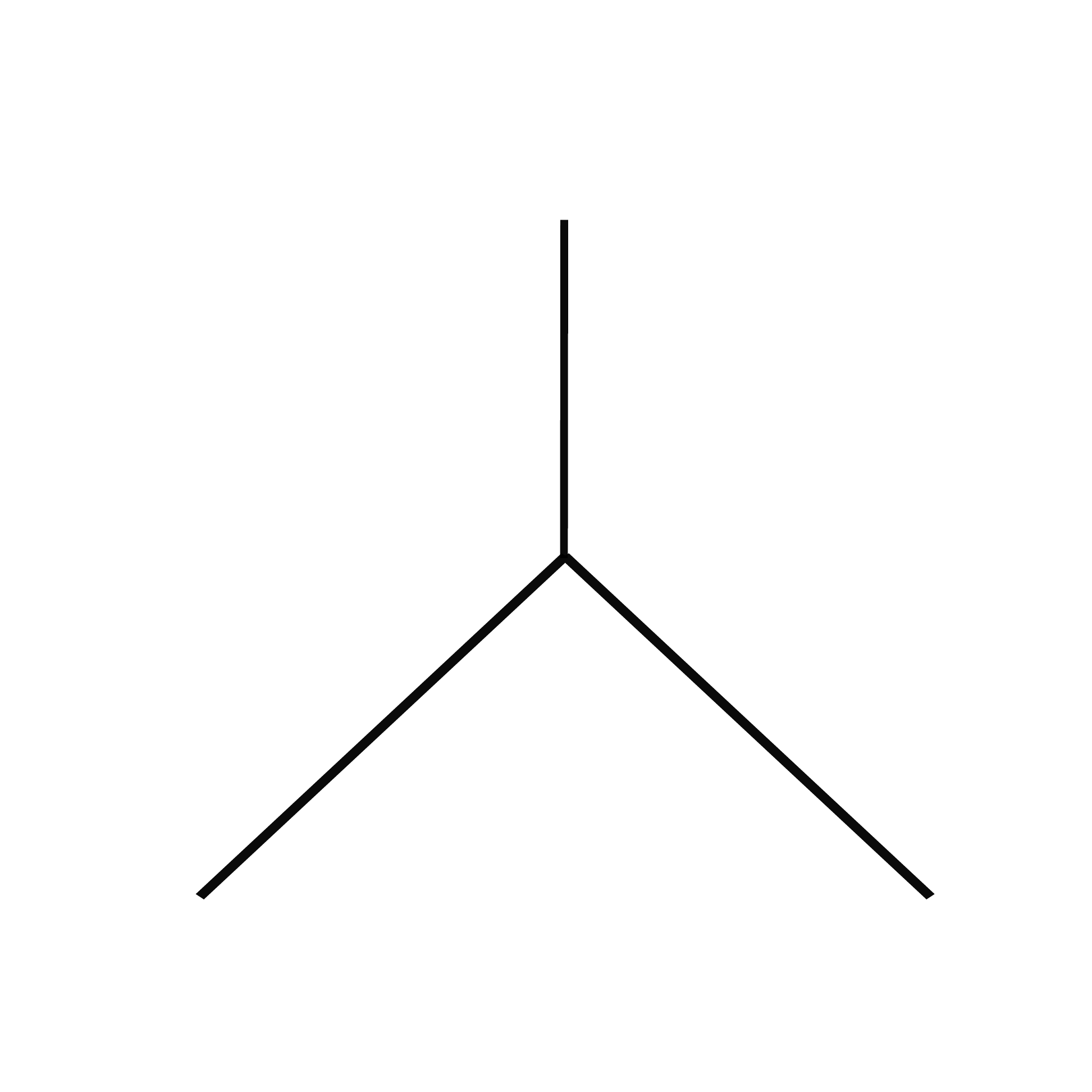}$, where $\graa{TTrivalent}=\graa{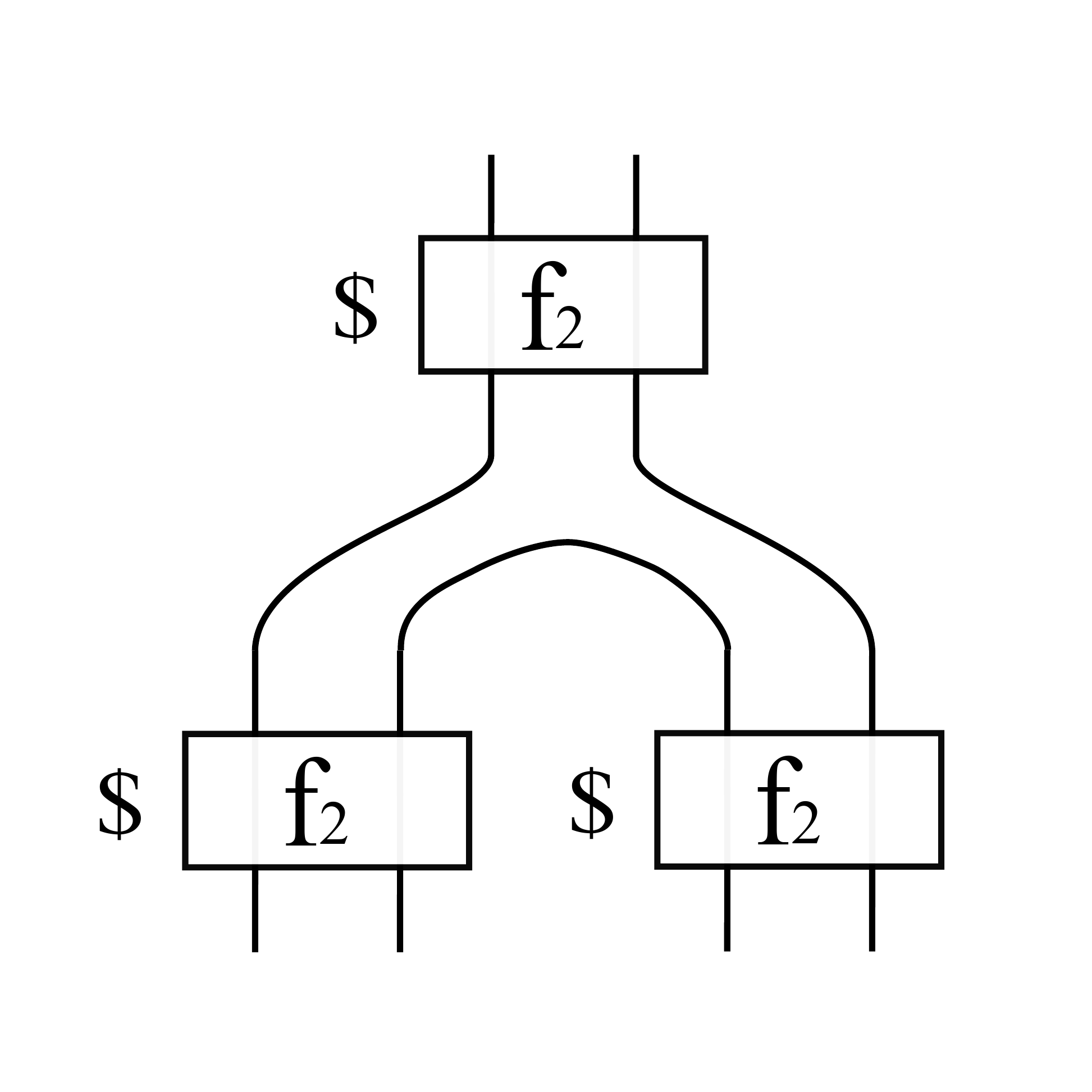}.$
	
	we construct a unitary representation $\pi$ with the trivalent vertex in Approach 2 in \S \ref{sec:pre}. Similar to the case in \S \ref{sec:vecF}, the matrix coefficient with respect to the vacuum vector has a specific chromatic meaning:
	
	Suppose $g\in F$ with $(T_+,T_-)$ a pair of trees representation. We arrange the pair of trees as in \S \ref{sec:vecF} to obtain a cubic planar graph. Let $\Gamma(T_+,T_-)$ be the dual graph of the cubic graph with vertices $\{(0,0),(1,0),\cdots,(n,0)\}$. For instance,
	\begin{equation}
	\grd{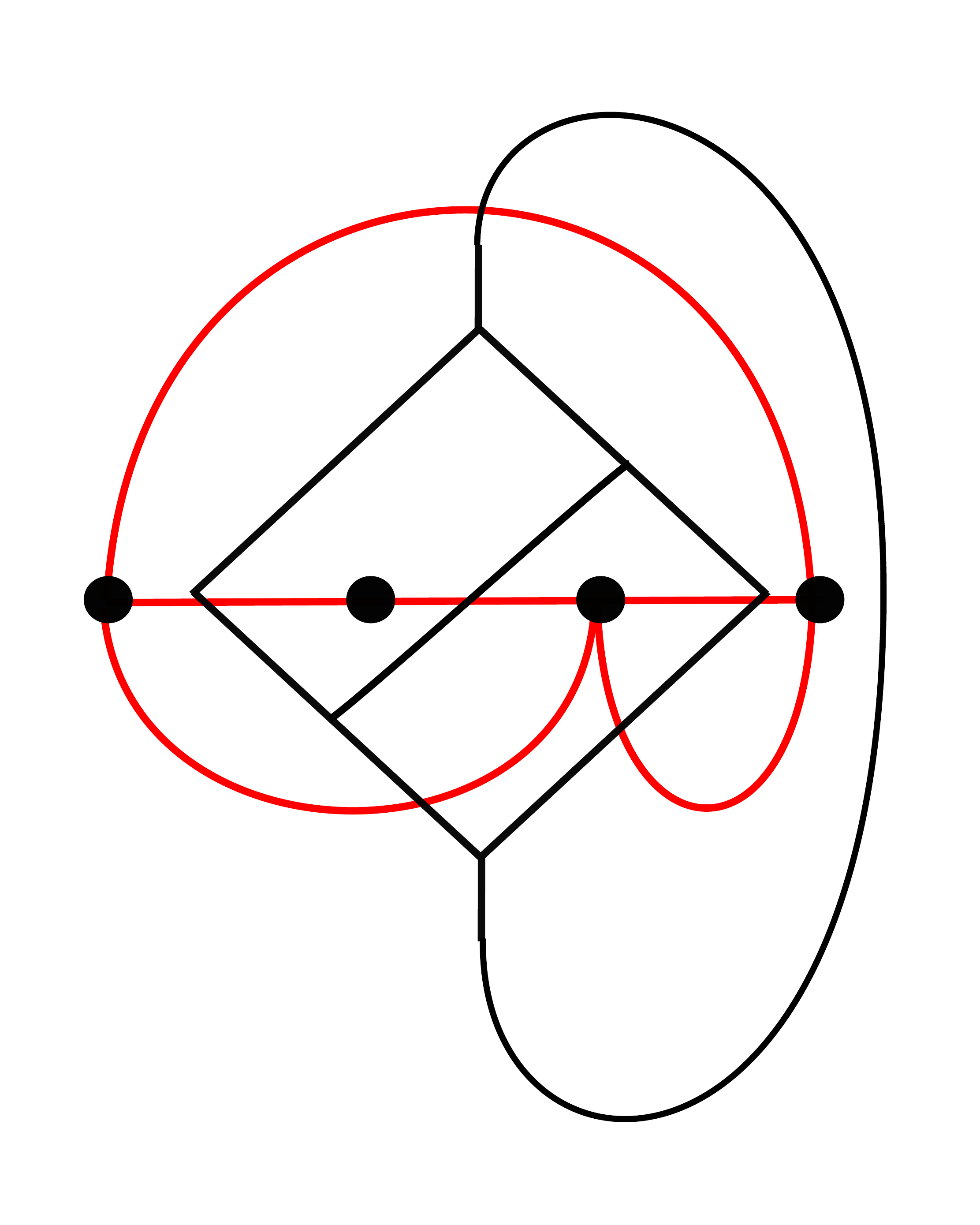}
	\end{equation}
	
	\begin{proposition}
		The matrix coefficient satisfies the following:
		\begin{align*}
		\langle\pi(g)\xi,\xi\rangle=
		\begin{cases}
		1, &~\Gamma(T_+,T_-)~\text{is 3-colorable}\\
		0, &~\Gamma(T_+,T_-)~\text{is not 3-colorable}
		\end{cases}
		\end{align*}    
	\end{proposition}
	\begin{remark}
		We use $\{a,b,c\}$ as the coloring of $\Gamma(T_+,T_-)$ for $g$ having $(T_+,T_-)$ as a pair of representation with $\langle\pi(g)\xi,\xi\rangle=1$.
	\end{remark}
	\begin{definition}[Jones]
		We define the 3-colorable subgroup $\mathcal{F}$ as the stabilizer of the vacuum vector
		\begin{equation*}
		\mathcal{F}=\{g\in F\vert \pi(g)\xi=\xi\}
		\end{equation*}
	\end{definition}
	\begin{lemma}\label{lem:acb}
		Suppose $g\in\mathcal{F}$, then g has a pair of trees representation $(T_+,T_-)$ such that the coloring of the vertices of $\Gamma(T_+,T_-)$ is $acbacb\cdots ac$. 
	\end{lemma}
	\begin{proof}
		First note that this lemma should hold for the coloring after applying any permutation on $\{a,b,c\}$ to $acbacb\cdots ac$. We will prove the lemma by induction. Suppose $(T_+,T_-)$ is a pair of trees representation of $g\in\mathcal{F}$. Consider there exists $i\in\mathbb{N}$ with coloring of the vertex $(i,0)$ is $a$ and the vertex $(i+1,0)$ is $b$. Then we apply the same technique in Lemma \ref{lem:alternating}
		\begin{equation*}
		\Gamma(T_+,T_-):\grb{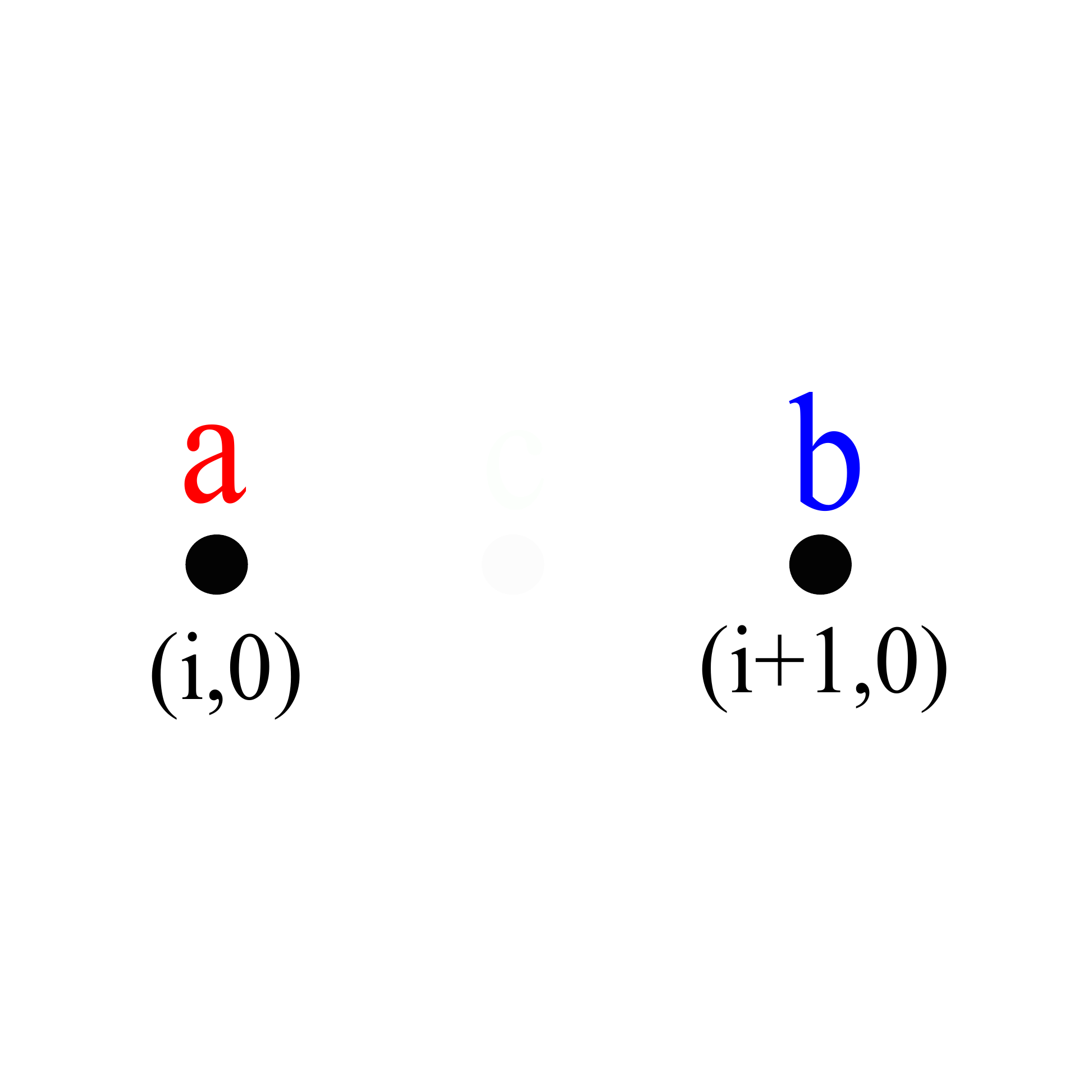}\longleftrightarrow(T_+,T_-):\grb{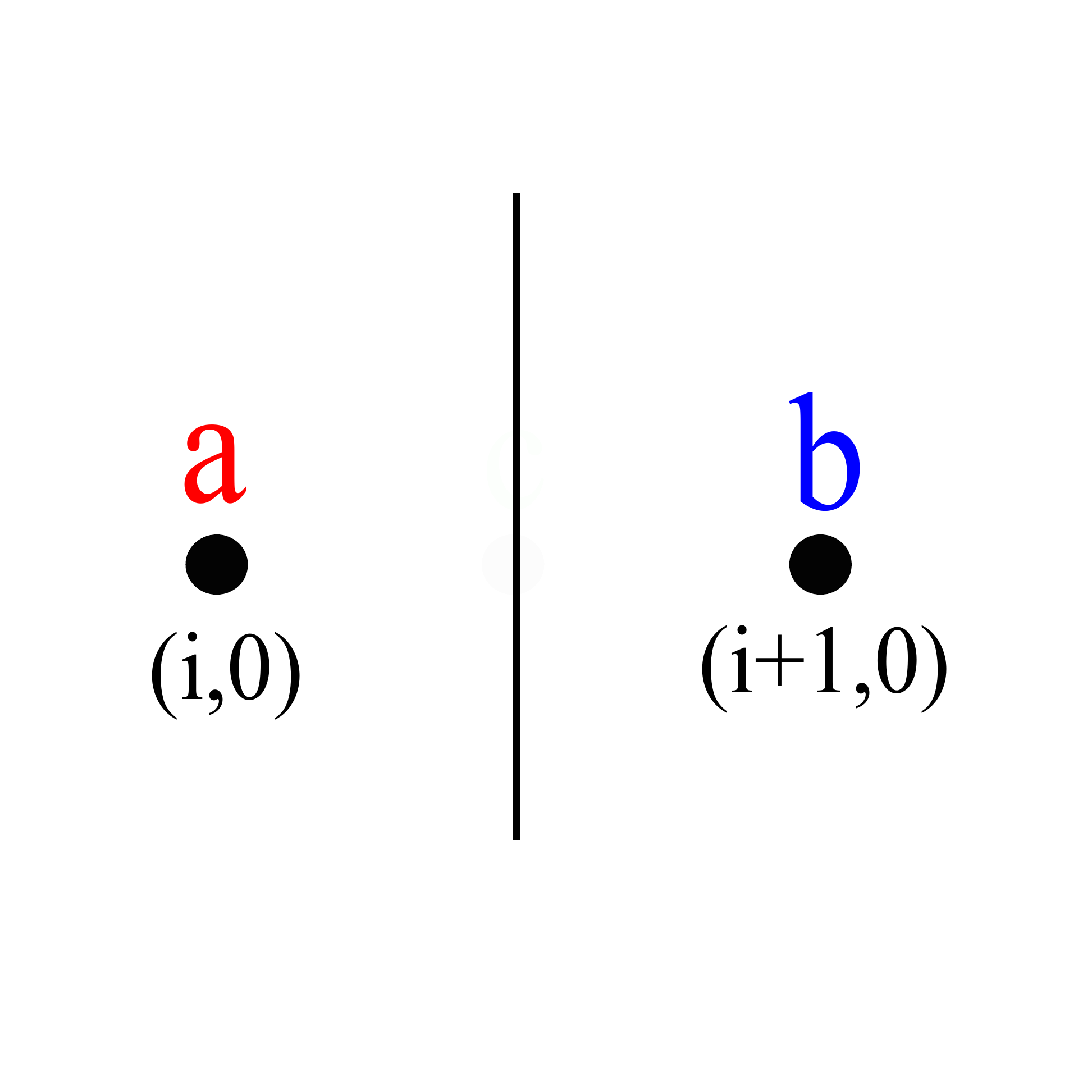}
		\end{equation*}
		We consider the pair of trees with by putting a caret on both $T_\pm$ on the edge with endpoints $(i+1/2,0)$. Therefore, we have
		\begin{equation*}
		\Gamma(T_+,T_-):\grb{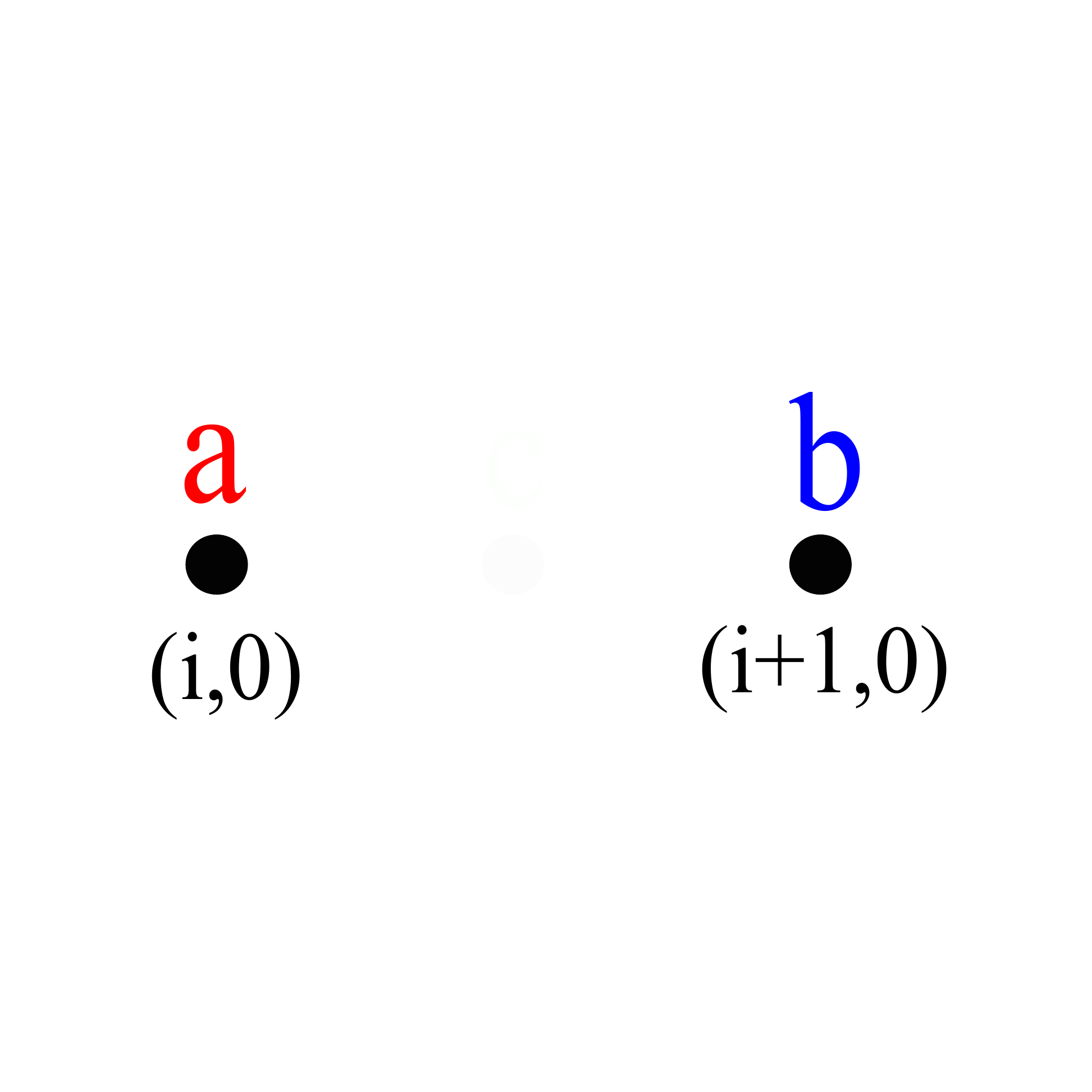}\longleftrightarrow (T_+,T_-):\grb{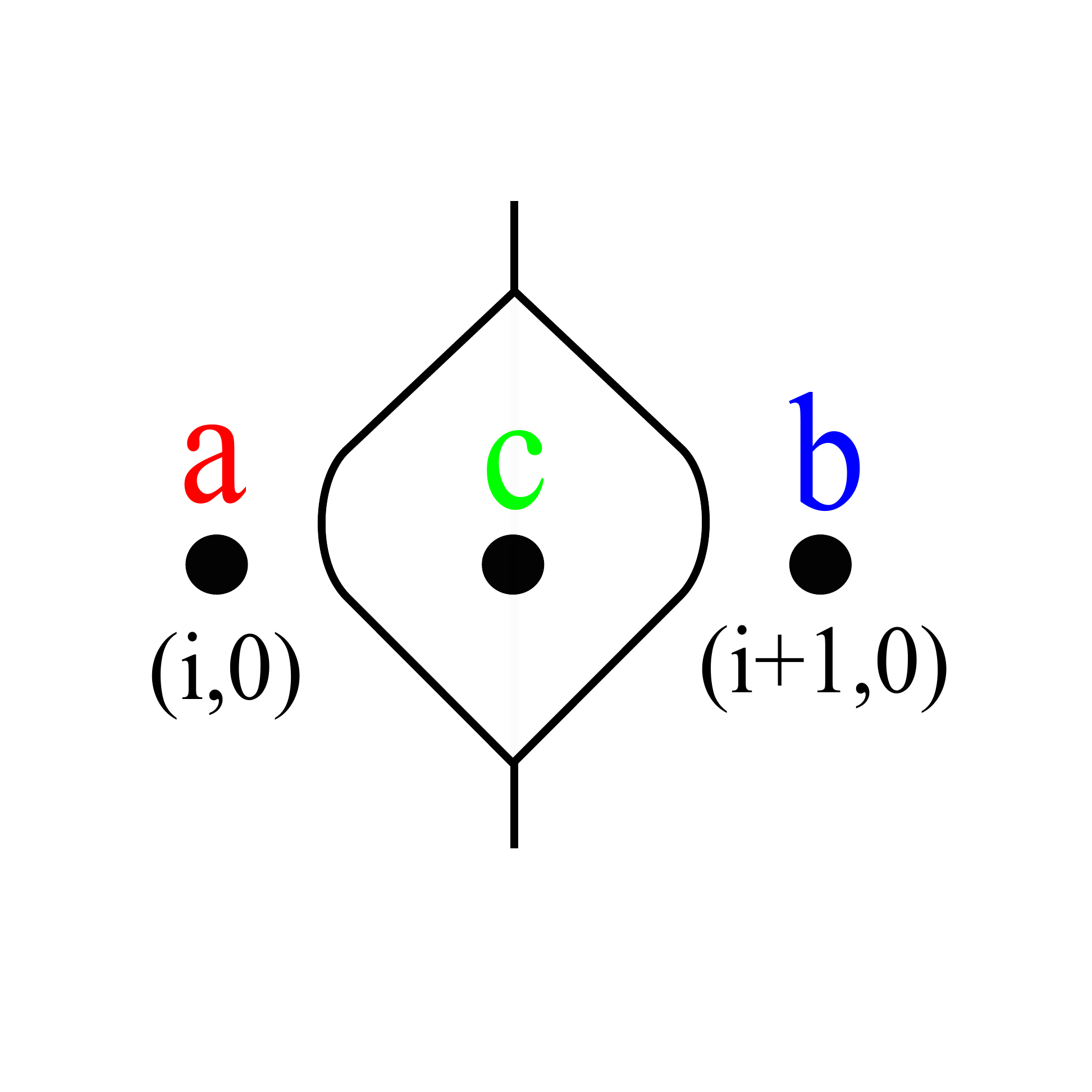}
		\end{equation*}
		
	\end{proof}
	\begin{theorem}
		The 3-colorable subgroup $\mathcal{F}$ is isomorphic to $F_4$.
	\end{theorem}
	\begin{proof}
		Let $X=\graa{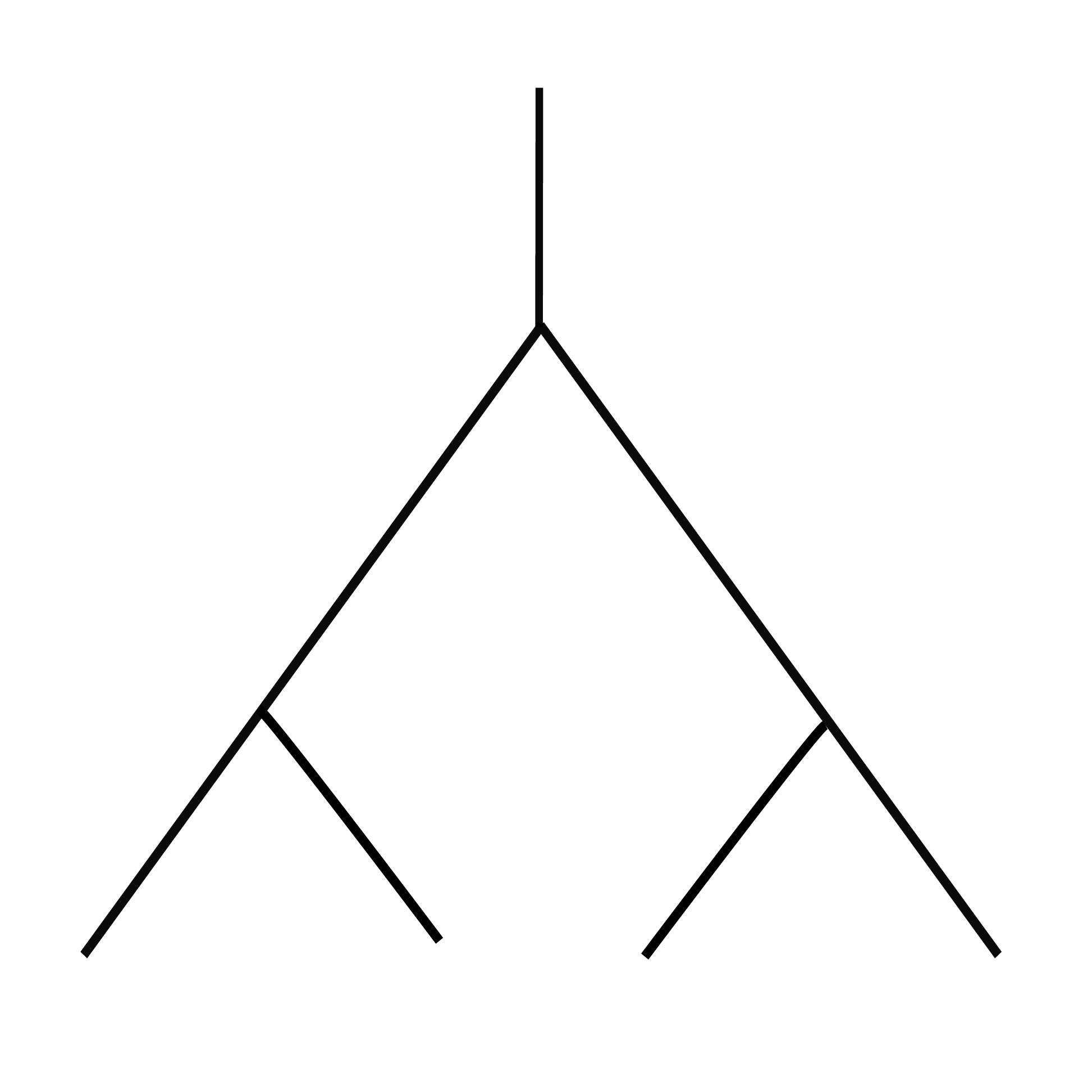}$. Suppose $g\in\mathcal{F}$ having $(T_+,T_-)$ as a pair of trees representation satisfying Lemma \ref{lem:acb} where $T_\pm$ has $3n+1$ leaves.
		\begin{claim}
			There exits $i\in\mathbb{N}$ such that in $T_+$,
			\begin{equation*}    
			\begin{tikzpicture}
			\draw (5,4.6)--(5,4.3);
			\draw (5,4.3)--(3.8,3);
			\draw (5,4.3)--(6.2,3);
			\draw (4.4,3.65)--(4.6,3);
			\draw (5.6,3.65)--(5.4,3);
			\draw[fill] (3.4,3) circle [radius=0.03];
			\draw[fill] (4.2,3) circle [radius=0.03];
			\draw[fill] (5,3) circle [radius=0.03];
			\draw[fill] (5.8,3) circle [radius=0.03];
			\draw[fill] (6.6,3) circle [radius=0.03];
			\node[below] at (3.4,3) {\small $(i,0)$};
			\end{tikzpicture}
			\end{equation*}         
		\end{claim}
		\begin{proof}[\textbf{Proof of Claim}]
			We will prove the claim by induction. The basic step is trivially true since it corresponds to the coloring $acbac$. Suppose this claim holds for binary trees with $3n-2$ leaves. $T_+$ must start with the form    
			\begin{equation*}    
			\begin{tikzpicture}
			\draw (5,4.6)--(5,4.3);
			\draw (5,4.3)--(3.8,3);
			\draw (5,4.3)--(6.2,3);
			\draw (4.4,3.65)--(4.6,3);
			\draw (5.6,3.65)--(5.4,3);
			\node [red] at (4.4,4) {$a$};
			\node [blue] at (5.6,4) {$c$};
			\node [green] at (5,3.65) {$b$};
			\node [blue] at (4.3,3.3) {$c$};
			\node [red] at (5.7,3.3) {$a$};
			\node at (5.3,4.3) {$v_0$};
			\node at (4.0,3.65) {$v_{00}$};
			\node at (5.9,3.65) {$v_{01}$};
			\end{tikzpicture}
			\end{equation*}
			
			If there are no more other vertices connected to $v_{00}$ or $v_{01}$, then the claim holds. 
			
			If there are more vertices on the left edge of $v_{00}$, then the dotted circle part is a tree with $acbacb\cdots ac$. By induction, we know that the claim holds. 
			\begin{equation*}
			\begin{tikzpicture}
			\draw (5,4.6)--(5,4.3);
			\draw (5,4.3)--(3.8,3);
			\draw (5,4.3)--(6.2,3);
			\draw (4.4,3.65)--(4.6,3);
			\draw (5.6,3.65)--(5.4,3);
			\draw (3.8,3)--(3.2,2.35);
			\draw (3.8,3)--(4.0,2.35);
			\node [red] at (4.2,4) {$a$};
			\node [blue] at (5.6,4) {$c$};
			\node [green] at (5,3.65) {$b$};
			\node [blue] at (4.3,3.3) {$c$};
			\node [red] at (5.7,3.3) {$a$};
			\node at (5.3,4.3) {$v_0$};
			\node at (4.0,3.65) {$v_{00}$};
			\node at (5.9,3.65) {$v_{01}$};
			\draw [dashed] (3.6,2.6) circle [radius=0.6];;
			\end{tikzpicture}
			\end{equation*} 
			The same argument holds for the cases that there are more vertices on the right edge of $v_{00}$ or the left (or right) edge of $v_{01}$.
		\end{proof}
		With proving the claim, we obtain that $T_+\in Alg(X)_{3n+1}$. Therefore, $\mathcal{F}\cong F_4$ by Theorem \ref{thm:main}
	\end{proof}
		
	\bibliography{bibliography}
	\bibliographystyle{amsalpha}
\end{document}